\documentclass[11pt,a4paper,twoside]{article}
\usepackage{amsmath}
\usepackage{amssymb}
\usepackage{amsthm}
\usepackage{dsfont}
\usepackage{amsfonts}
\usepackage{color}
\usepackage{thmtools}
\usepackage{enumerate}
\usepackage[margin = 2in]{geometry}
\usepackage{cite}
\usepackage{url}
\usepackage{hyperref}
\usepackage{graphicx}
\usepackage{subfigure}
\usepackage{overpic}
\usepackage{verbatim}
\usepackage[toc,page]{appendix}
\usepackage{listings}
\declaretheorem[name=Theorem, numberwithin=section]{theorem}
\newtheorem{prop}[theorem]{Proposition}
\newtheorem{lem}[theorem]{Lemma}
\newtheorem{ass}[theorem]{Assumption}
\newtheorem{cor}[theorem]{Corollary}
\newtheorem{con}[theorem]{Conjecture}
\theoremstyle{remark}
\newtheorem{rem}[theorem]{Remark}
\theoremstyle{definition}
\newtheorem{defi}[theorem]{Definition}
\newcommand{\eps}{\varepsilon}
\newcommand{\diff}{\mathrm{d}}
\newcommand{\var}{\mathrm{Var}}
\newcommand{\cov}{\mathrm{Cov}}
\newcommand{\Onot}{\mathcal{O}}
\newcommand{\p}{\mathbb{P}}
\newcommand{\e}{\mathbb{E}}
\newcommand{\ind}{\mathds{1}}
\newcommand{\C}{\mathcal{C}}
\newcommand{\R}{\mathbb{R}}
\newcommand{\N}{\mathbb{N}}

\newcommand{\Span}{\mathrm{span}}

\title{Sample Paths Estimates for Stochastic Fast-Slow Systems driven by Fractional Brownian Motion}
\author{Katharina Eichinger\thanks{Technical University of Munich (TUM), 
Faculty of Mathematics, 85748 Garching bei M\"unchen, Germany}~~and~Christian 
Kuehn\thanks{Technical University of Munich (TUM), 
Faculty of Mathematics, 85748 Garching bei M\"unchen, Germany}~~and~Alexandra 
Neam\c tu\thanks{Technical University of Munich (TUM), 
Faculty of Mathematics, 85748 Garching bei M\"unchen, Germany}}

\begin{document}
%\pagenumbering{roman}
%\pagestyle{empty}
\maketitle
%\tableofcontents
\bibliographystyle{plain}
%\thispagestyle{empty}

%%%%%%%%%%%%%%%%%%%%%%%%%%%%%%%%%%%%%%%%%%%%%%%%%%%%%%%%%%%%%%%%%%%%%%%%%%%%%%%%%%%%%%%%%%%%%%%%%%%%%%%%%%%%%%%%%%%%%%%%%%%%%%%%%%%%%

\begin{abstract}
We analyze the effect of additive fractional noise with Hurst parameter $H > \frac{1}{2}$ on fast-slow systems. Our strategy is based on sample paths estimates, similar to the approach by Berglund and Gentz in the Brownian motion case. Yet, the setting of fractional Brownian motion does not allow us to use the martingale methods from fast-slow systems with Brownian motion. We thoroughly investigate the case where the deterministic system permits a uniformly hyperbolic stable slow manifold. In this setting, we provide a neighborhood, tailored to the fast-slow structure of the system, that contains the process with high probability. We prove this assertion by providing exponential error estimates on the probability that the system leaves this neighborhood. We also illustrate our results in an example arising in climate modeling, where time-correlated noise processes have become of greater relevance recently.
\end{abstract}

%%%%%%%%%%%%%%%%%%%%%%%%%%%%%%%%%%%%%%%%%%%%%%%%%%%%%%%%%%%%%%%%%%%%%%%%%%%%%%%%%%%%%%%%%%%%%%%%%%%%%%%%%%%%%%%%%%%%%%%%%%%%%%%%%%%%%
\section{Introduction}
\pagenumbering{arabic}
\pagestyle{plain}
\setcounter{page}{1}

Fast-slow systems naturally arise in the modeling of several phenomena in natural sciences, when processes have widely differing rates~\cite{MultTimeScaleDyn,Jones,Hek}. The standard form of a fast-slow system of ordinary differential equations (ODEs) is given by
\begin{equation}
\label{genfast-slow-fastTimeintro}
\begin{aligned}
\frac{\diff x}{\diff s} = x^{\prime} &= f(x, y, \eps), \\
\frac{\diff y}{\diff s} = y^{\prime} &= \eps g(x, y, \eps),
\end{aligned}
\end{equation}
where $x$ are the fast variables, $y$ are the slow variables, $\eps>0$ is a small parameter, and $f,g$ are sufficiently smooth vector fields; for a more detailed technical introduction regarding the analysis of~\eqref{genfast-slow-fastTimeintro} we refer to Section~\ref{ssec:fastslowdet}. Here we just point out the basic aspects from the modeling perspective. First, note that if $\eps=0$, then~\eqref{genfast-slow-fastTimeintro} becomes a parametrized set of ODEs, where the $y$-variables are parameters. Taking this viewpoint, all bifurcation problems~\cite{GH,Kuznetsov} involving parameters naturally relate to fast-slow dynamics if the parameters vary slowly, which is often a natural assumption in applications. Second, in practice, we also want to couple many dynamical systems. The resulting large/complex system is often multiscale in time and space. For example, in the context of climate modeling~\cite{Dijkstra,KaperEngler} coupled processes can evolve on temporal scales of seconds up to millennial scales. Third, fast-slow systems are the core class of dynamical problems to understand singular perturbations~\cite{Verhulst}, i.e., roughly speaking singular perturbations problems with small parameters are those, which degenerate in the limit of the small parameter into a different class of equations. Combining all these observations, it is not surprising that fast-slow systems have become an important tool in more theoretical as well as application-oriented parts of nonlinear dynamics~\cite{MultTimeScaleDyn}.\medskip  

However, when dealing with real life phenomena certain random influences have to be taken into account and quantified in a suitable way~\cite{Gardiner}. The most common stochastic process used to describe uncertainty is Brownian motion $W=W_t$. One of its key features is the memory-less or Markov property, which means that the behavior of this process after a certain time $T>0$ only depends on the situation at the current time $T$. In certain applications it may be desirable to model long-range dependencies and to take into account the evolution of the process up to time $T$. One of the most famous example is constituted by fractional Brownian motion (fBm) $W^H=W^H_t$ ; see~\cite{Kolmogorov} for its first use. A fBm is a centered stationary Gaussian processes parameterized by the so-called Hurst index/parameter $H\in(0,1)$. For $H=1/2$ one recovers classical Brownian motion. However, for $H\in(1/2,1)$ and $H\in(0,1/2)$, fBm exhibits a totally different behavior compared to Brownian motion. Its increments are no longer independent, but positively correlated for $H>1/2$ and negative correlated for $H<1/2$. The Hurst index does not only influence the structure of the covariance but also the regularity of the trajectories. Fractional Brownian motion has been used to model a wide range of phenomena such as network traffic~\cite{networkt}, stock prices and financial markets~\cite{Mishura,Stone}, activity of neurons~\cite{na,neuron}, dynamics of the nerve growth~\cite{nerve}, fluid dynamics~\cite{Weiss}, as well as various phenomena in geoscience~\cite{Mobergetal,KoutsoyiannisMontanari,RangarajanSant}. However, the mathematical analysis of stochastic systems involving fBm is a very challenging task. Several well-known results for classical Brownian motion are not available. For instance, the distribution of the hitting time $\tau_{a}$ of a level $a$ is explicitly known for a Brownian motion, whereas for fBm, one has only an asymptotic statement, according to which
\begin{align*}
\mathbb{P}(\tau_{a}>t)= t^{-(1-H)+o(1)}, 
\end{align*}
as $t$ goes to infinity, see~\cite{Molchan}. Furthermore, since fBm is not a semi-martingale, It\^o-calculus breaks down. Therefore, it is highly non-trivial to define an appropriate integral with respect to the fBm. This issue has been intensively investigated in the literature. There are numerous approaches that exploit the regularity of the trajectories of the fBm in order to develop a completely path-wise integration theory and to analyze differential equations. For more details, see~\cite{MaslowskiNualart,GarridoLuSchmalfuss1, GubinelliLejayTindel, GubinelliTindel, HesseNeamtu1} and the references specified therein. Furthermore, another ansatz employed to define stochastic integrals with respect to fBm relies on the stochastic calculus of variations (Malliavin calculus) developed in~\cite{Malliavin1}. In summary, fBm is a natural candidate process to aim to improve our understanding of correlated stochastic dynamics.\medskip

Our objective here is to combine the study of fast-slow systems and fBm by starting to study stochastic differential equations of the form 
\begin{equation}\label{eqMDSysinSlowTimeintro}
\begin{aligned}
\diff x &= \frac{1}{\eps} f(x, y,\eps) \diff t + \frac{\sigma}{\eps^H} \diff W^H_t,\\
\diff y &= 1 \diff t,
\end{aligned}
\end{equation}
where we start with the case of additive noise for the fast variable(s) and assume there is a single regularly slowly-drifting variable $y$. For $H=1/2$, i.e., for Brownian motion, there is a very detailed theory, how to analyze stochastic fast-slow systems~\cite{MultTimeScaleDyn}. One particular building block - initially developed by Berglund and Gentz - uses a sample paths viewpoint~\cite{NoiseSlowFastSys}. This approach has recently been extended to broader classes of spatial stochastic fast-slow systems~\cite{GnannKuehnPein} and it has found many successful applications; see e.g.~\cite{BerglundGentz8,KuehnCT2,SadhuKuehn,SuRubinTerman}. Therefore, it is evident that one should also consider the case of correlated noise in the fast-slow setup~\cite{Xu,Hairer}.\medskip

Our key goal is to derive sample paths estimates for fast-slow systems driven by fBm with Hurst index $H\in(1/2,1)$. We restrict ourselves to the case of additive noise and establish the theory for the normally hyperbolic stable case. Due to the technical challenges mentioned above, we need to derive sharp estimates for the exit times for processes solving certain equations driven by fBm. Exploring various properties of general Gaussian processes, we propose two variants to obtain optimal sample paths estimates. 

Then we are going to apply our theory to a climate model describing the North-Atlantic thermohaline circulation forced by fractional Brownian motion. In fact, it is well-established that just using white noise modelling in climate models can be insufficient. The simple reason is that neglecting spatial and temporal correlations does not represent the statistics of large classes of underlying climate measurement data including temperature time series~\cite{Kaerner,Eichneretal}, historical climate data~\cite{Ashkenazyetal,Barbozaetal,DavidsenGriffin}, as well as large-scale simulation data~\cite{BlenderFraedrich}. In all these cases, an elegant way to model temporal correlations in climate science is fractional Brownian motion~\cite{DavidsenGriffin,Kaerner,Sonechkin,Ashkenazyetal,Yuanetal}. The reasoning to use a time correlated process can also be understood in climate dynamics in various intuitive ways. For example, in a larger-scale climate model, stochastic terms often represent unresolved degrees of freedom or small-scale fluctuations. If we consider the weather as a short-lived smaller scale effect in terms of the global long-term climate, then models for the latter must include noise with (positive) time correlations as weather patterns are positively correlated in time on short scales. Similarly, if the noise terms represent external forcing, such as input from another climate subsystem on a macro-scale, then also this input is likely to be correlated in time as there are internal correlations of the long-term behaviour of each larger-scale climate subsystem. In summary, this has motivated us to consider a model from climate dynamics as one possible key application for fast-slow dynamical systems with fractional Brownian motion. As mentioned above, in many other applications, fractional Brownian motion also naturally appears, so our modelling approach via fast-slow systems with with fBm is even more broadly applicable. \medskip
 
This work is structured as follows. In Section~\ref{ch_preliminaries} we introduce basic notions from the theory of fast-slow systems and fractional Brownian motion. Furthermore, we state important estimates for the exit times of Gaussian processes which will be required later on. In Section~\ref{ch_1Dcase}, we generalize the theory of~\cite{NoiseSlowFastSys} by first deriving an attracting invariant manifold of the variance using the fast-slow structure of the system. Based on this manifold we define a region, where the linearization of the process is contained with high probability. In order to prove such statements, we first derive a suitable nonlocal Lyapunov-type equation for the covariance of the solution of a linear equation driven by fBm, the so-called fractional Ornstein-Uhlenbeck process. Thereafter we analyze two variants which entail sharp estimates for the exit times of this process. Furthermore, we consider more complicated dynamics and provide extensions of our results to the non-linear case, more complicated slow dynamics and finally discuss the case of fully coupled dynamics. We apply our theory to a model for the North-Atlantic thermohaline circulation and provide some simulations. Section~\ref{ch_MDcase} generalizes the sample paths estimates to higher dimensions in the autonomous linear case. Our strategy is based on diagonalization techniques, which allow us to go back to the one-dimensional case and apply the results developed in~Section~\ref{ch_1Dcase}. For completeness, we provide an appendix which contains a detailed proof regarding the limit superior of a non-autonomous fractional Ornstein-Uhlenbeck processes. We conclude in Section~\ref{ch_outlook} with an outlook of possible continuations of our results.

%%%%%%%%%%%%%%%%%%%%%%%%%%%%%%%%%%%%%%%%%%%%%%%%%%%%%%%%%%%%%%%%%%%%%%%%%%%%%%%%%%%%%%%%%%%%%%%%%%%%%%%%%%%%%%%%%%%%%%%%%%%%%%%%%%%%%
\section{Background}
\label{ch_preliminaries}

%%%%%%%%%%%%%%%%%%%%%%%%%%%%%%%%%%%%%%%%%%%%%%%%%%%%%%%%%%%%%%%%%%%%%%%%%%%%%%%%%%%%%%%%%%%%%%%%%%%%%%%%%%%%%%%%%%%%%%%%%%%%%%%%%%%%%
\subsection{Deterministic Fast-Slow Systems}
\label{ssec:fastslowdet}

In this section, we will briefly introduce the terminology of fast-slow systems. We restrict ourselves to the most important results tailored to our problem in the upcoming sections. For further details, see~\cite{MultTimeScaleDyn}. For the definition of the setting, all of the equations are to be understood formally. We will later add regularity assumptions sufficient to deduce important results. These also imply that the formal computation we will have performed before are valid.

\begin{defi}
A \emph{fast-slow system} is an (ODE) of the form
\begin{equation}\label{genfast-slow-fastTime}
\begin{aligned}
\frac{\diff}{\diff s} x_s = x_s^{\prime} &= f(x_s, y_s, \eps), \\
\frac{\diff}{\diff s} y_s = y_s^{\prime} &= \eps g(x_s, y_s, \eps),
\end{aligned}
\end{equation}
where $x=x_s$, $y=y_s$ are the unknown functions of the \emph{fast time} variable $s$, the vector fields are 
$f: \R^m \times \R^n \times \R \rightarrow \R^m, g: \R^m \times \R^n \times \R \rightarrow \R^n$, and $\varepsilon > 0$ is a small parameter. The $x$ variables are called the \emph{fast variables}, while $y$ variables are called the \emph{slow variables}. Transforming into another time scale by defining the\emph{slow time} $t = \eps s$ yields the equivalent system
\begin{equation}\label{genfast-slow-slowTime}
\begin{aligned}
\eps \frac{\diff}{\diff t} x_t = \dot{x}_t &= f(x_t, y_t, \eps), \\
\frac{\diff}{\diff t} y_t = \dot{y}_t&= g(x_t, y_t, \eps).
\end{aligned}
\end{equation}
\end{defi}

Depending on the situation both formulations in fast and slow time may be of use. In particular, under certain assumptions, considering them for $\eps \rightarrow 0$ indicates a lot of information for the underlying dynamics for the case $ 0<\eps \ll 1$. The process for $\eps \rightarrow 0$ is called the \emph{singular limit}. The singular limit of~\eqref{genfast-slow-fastTime} for $\eps \rightarrow 0$
\begin{align*}
\frac{\diff}{\diff s} x_s = x_s^{\prime} &= f(x_s, y_s, 0), \\
\frac{\diff}{\diff s} y_s = y_s^{\prime} &= 0,
\end{align*}
is called the \emph{fast subsystem}. The resulting system of the slow time formulation of the fast-slow system \eqref{genfast-slow-slowTime} for $\eps \rightarrow 0$
\begin{align*}
0 &= f(x_t, y_t, 0), \\
\frac{\diff}{\diff t} y_t  = \dot{y}_t &= g(x_t, y_t, 0).
\end{align*}
is called the \emph{slow subsystem}. The set 
\begin{align*}
C_0 := \left\{ (x,y) \in \R^m \times \R^n : ~ f(x,y,0) = 0\right\}
\end{align*}
is called the \emph{critical set}. If $C_0$ is a manifold, it is also called the \emph{critical manifold}. From now on, we assume that $C_0$ is a manifold given by a graph of the slow variables, i.e., 
\begin{align*}
C_0 = \left\{ (x^*(y),y) \in \R^m \times \R^n : ~ x^*: \mathcal{D} \rightarrow \R^m, f(x^*(y),y,0) = 0\right\},
\end{align*}
where $\mathcal{D} \subset \R^n$ is an open subset.

\begin{theorem}[Fenichel--Tikhonov,\cite{Fenichel4,Tikhonov,Jones,MultTimeScaleDyn}]\label{TihonovThm}
Let $f,g \in \mathcal{C}^r(\R^m \times \R^n \times \R)$, $1 \leq r < \infty$, and their derivatives up to order $r$ be uniformly bounded. Assume that $C_0$ is uniformly hyperbolic. Then for an $\eps_0> 0$ there exists a locally invariant $\mathcal{C}^r$-smooth manifold 
\begin{align*}
C_{\eps} = \left\{ (x,y): ~ x =  \bar{x}(y,\eps) \right\},
\end{align*}
for all $\eps \in (0,\eps_0]$, where $\bar{x}(y,\eps) = x^*(y) + \Onot(\eps)$ with respect to the fast variables. Furthermore, the local stability properties of $C_{\eps}$ are the same as the ones for $C_{0}$. 
\end{theorem}

%%%%%%%%%%%%%%%%%%%%%%%%%%%%%%%%%%%%%%%%%%%%%%%%%%%%%%%%%%%%%%%%%%%%%%%%%%%%%%%%%%%%%%%%%%%%%%%%%%%%%%%%%%%%%%%%%%%%%%%%%%%%%%%%%%%%%
\subsection{Fractional Brownian Motion}

In this section we state important properties of fBm, which will be required later on. For further details see~\cite{SelAspOfFBM,StochCalcforfBm} and the references specified therein. We fix a complete probability space $(\Omega, \mathcal{F}, \p)$ and use the abbreviation a.s.~for almost surely.

\begin{defi}
\label{def_1Dfbm}
Let $H \in (0,1]$. A one-dimensional fractional Brownian motion (fBm) of Hurst index/parameter $H$ is a continuous centered Gaussian process $(W_t^H)_{t \geq 0}$ with covariance
\begin{align*}
\e[W_t^H W_s^H] = \frac{1}{2} \left( t^{2H} + s^{2H} - \left| t - s \right|^{2H} \right) \hspace{0.7cm} \text{for all } t,s \geq 0.
\end{align*}
\end{defi}

Note that for $H>1/2$ the covariance of fBm satisfies 
\begin{align*}
\frac{1}{2} (t^{2H}  + s^{2H} - |t-s|^{2H} ) = H(2H-1) \int\limits_{0}^{t} \int\limits_{0}^{s} |v-u|^{2H-2} ~\diff v ~\diff u.
\end{align*}
We further observe that:
\begin{enumerate}
        \item [1)] for $H=1/2$ one obtains Brownian motion;
        \item [2)] for $H=1$ then $W^{H}_{t}= t W^{H}_{1}$ a.s. for all $t\geq 0$. Due to this reason one always considers $H\in(0,1)$.
\end{enumerate}

The following result regarding the structure of the covariance of fBm holds true, see~\cite[Section~2.3]{SelAspOfFBM}. 

\begin{prop} Let $H>1/2$. Then, the covariance of fBm has the integral representation
        \begin{align}\label{int:kernel}
        \e[W_t^H W_s^H]= \int\limits_{0}^{\min\{s,t\}} K(s,r) K(t,r)~\diff r~~\mbox{ for  } s,t\geq 0,
        \end{align}
        where the integral kernel $K$ is given by
        \begin{align*}
        K(t,r)=c_{H}\int\limits_{r}^{t} \Big(\frac{u}{r} \Big)^{H-1/2}(u-r)^{H-3/2}~\diff u,
        \end{align*}
        for a positive constant $c_{H}$ depending exclusively on the Hurst parameter.
\end{prop}

We remark that for suitable square integrable kernels, one obtains different stochastic processes, for instance the multi-fractional Brownian motion or the Rosenblatt process, see~\cite{CMaslowski}. We now focus on the most important properties of fBm. For the complete proofs of the following statements, see~\cite[Chapter~2]{SelAspOfFBM}.

\begin{prop}[Correlation of the increments] 
Let $(W^{H}_{t})_{t\geq 0}$ be a fBm of Hurst index $H\in(0,1)$. Then its increments are:
  \begin{itemize}
    \item [1)] positively correlated for $H>1/2$;
    \item [2)] independent for $H=1/2$;
    \item[3)] negatively correlated for $H<1/2$.
  \end{itemize}
Particularly, for $H>1/2$ fBm exhibits long-range dependence, i.e.
$$\sum\limits_{n=1}^{\infty}\mathbb{E} [W^{H}_{1} (W^{H}_{n+1}-W^{H}_{n})]=\infty ,$$
whereas for $H<1/2$ 
$$\sum\limits_{n=1}^{\infty}\mathbb{E} [W^{H}_{1} (W^{H}_{n+1}-W^{H}_{n})]<\infty.$$
\end{prop}

\begin{prop} Let $(W^{H}_{t})_{t\geq 0}$ be a fBm of Hurst index $H\in(0,1)$. Then:
        \begin{itemize}
                \item [1)] [Self-similarity] For $a\geq 0$
                \begin{align}\label{self:sim}
                (a^HW_t^H)_{t \geq 0} \overset{law}{=} (W_{at}^H)_{t \geq 0},
                \end{align}
                i.e.~fBm is self-similar with Hurst index $H$.
                \item [2)] [Time inversion] ~~$\Big(t^{2H}W^{H}_{1/t} \Big)_{t>0} \overset{law}{=}(W^{H}_t)_{t>0}. $
                \item[3)] [Stationarity of increments] For all $h>0$
                \begin{align*}
                (W_{t+h}^H - W^{H}_{h})_{t \geq 0} \overset{law}{=} (W_{t}^H)_{t \geq 0}.
                \end{align*}
                \item [4)] [Regularity of the increments] fBm has a version which is a.s.~H\"older continuous of exponent $\alpha<H$.
        \end{itemize}     
\end{prop}

We conclude this section emphasizing the following result, which makes fBm very interesting from the point of view of applications, see~\cite[Section~2.4 and~2.5]{SelAspOfFBM}.

\begin{prop} Let $(W^{H}_{t})_{t\geq 0}$ be a fractional Brownian motion with Hurst index $H\in(0,1/2)\cup (1/2,1)$. Then $(W^{H}_{t})_{t\geq 0}$ is neither a semi-martingale nor a Markov process.
\end{prop}

%%%%%%%%%%%%%%%%%%%%%%%%%%%%%%%%%%%%%%%%%%%%%%%%%%%%%%%%%%%%%%%%%%%%%%%%%%%%%%%%%%%%%%%%%%%%%%%%%%%%%%%%%%%%%%%%%%%%%%%%%%%%%%%%%%%%%
\subsubsection{Integration Theory for $H > \frac{1}{2}$}\label{int:theory}
Since fBm is not a semi-martingale, the standard It\^o calculus is not applicable. Due to this reason, the construction of a stochastic integral of a \emph{random} function with respect to fBm has been a challenging question, see~\cite{Malliavin1, StochCalcforfBm} and the references specified therein. However, for \emph{deterministic} integrands and for $H>1/2$ the theory essentially simplifies. We deal exclusively with this case and indicate for the sake of completeness the theory of Wiener integrals of deterministic functions with respect to fBm, see~\cite{Malliavin1}. Let $T>0$ and 
\begin{align*}
\mathcal{E}: = \left\{ h : ~ h(s) = \sum_{k=1}^{N-1} h_k \mathds{1}_{[t_k, t_{k+1})}(s), N \in \N, 0 = t_1 < t_2 < \ldots < t_N = T, h_k \in \R \text{ for } k \in \{1, \ldots, N \} \right\}
\end{align*}
be the set of step functions on $[0,T]$. For $h \in \mathcal{E}$ define the linear mapping $I(h;T): \mathcal{E} \rightarrow L^2(\Omega)$
\begin{align*}
I(h;T) := \int_0^T h(r) \diff W_r^H := \sum_{k=1}^{N-1} h_k \left( W_{t_{k+1}}^H - W_{t_k}^H\right).
\end{align*}
Observe that $I(h;T)$ defines a Gaussian random variable with
\begin{align}
\e \left[ \int_0^T h(r) \diff W_r^H \right] &= 0, \nonumber\\
\var \left[ \int_0^T h(r) \diff W_r^H \right] &= H(2H-1) \int_0^T \int_0^T h(u) h(v) \left| u - v \right|^{2H-2} \diff u \diff v < \infty \nonumber\\
&= \int_0^T \int_0^T h(u) h(v) \phi( u - v ) \diff u \diff v, \label{variance}
\end{align}
where 
\begin{align}\label{phi}
\phi(s) := H(2H-1) \left| s \right|^{2H-2}
\end{align}
The representation of the variance can be easily verified by noting the following identity
\begin{align*}
\e\left[ \left( W_{t_{k+1}}^H - W_{t_k}^H \right) \left( W_{t_{l+1}}^H - W_{t_l}^H \right)\right] = H(2H-1) \int_{t_{k}}^{t_{k+1}} \int_{t_{l}}^{t_{l+1}} \left| u - v \right|^{2H-2} \diff u \diff v.
\end{align*}
Note that $H>1/2$ is crucial here. For $p > \frac{1}{H}$ we can bound the $L^2(\Omega)$-norm of $h \mapsto I(h;T)$ as follows
\begin{align*}
\left\Vert I(h;T)\right\Vert^2_{L^2(\Omega)} 
&= H(2H-1) \int_0^T \int_0^T h(u) h(v) \left| u - v \right|^{2H-2} \diff u \diff v \\
&\leq \| h \|_{L^p(0,T)} \| h * \phi\|_{L^{p/(p-1)}(0,T)} \\
&\leq \| \phi \|^2_{L^{p/(2p-2)}(0,T)} \| h \|^2_{L^p(0,T)},
\end{align*}
where we have obtained the estimate by applying H\"older's inequality and Young's inequality for convolutions \cite[Theorem~ 3.9.4]{bogachev2007measure}. The boundedness claim now follows as $\| \phi \|^2_{L^{p/(2p-2)}(0,T)} < \infty$ for $p > \frac{1}{H}$. This means that $I(\cdot,T)$ is a bounded linear operator defined on the dense subspace $\mathcal{E} \subset L^p(0,T)$, so it can be uniquely extended to a bounded operator
\begin{align*}
I_p(h;T): L^p(0,T) \rightarrow L^2(\Omega).
\end{align*}
This discussion justifies the following definition:
\begin{defi}
For $f \in L^p(0,T)$ and $t \in [0,T]$ we set 
\begin{align*}
\int_0^t f(r) \diff W_r^H := I_p(f\mathds{1}_{[0,t]};T)
\end{align*}
\end{defi}
The integral process $\left(I_p(f\mathds{1}_{[0,t]};T) \right)_{t \in [0,T]}$ is by construction centered Gaussian. Regarding~\eqref{variance}, its covariance can be immediately computed as follows.

\begin{prop}[Covariance of the integral]\label{cov:integral}
Let $a,b>0$ and $f, g \in L^{p}(0,T)$ for $p>1/H$. Then 
\begin{align*}
\cov \left( \int_0^a f(r) \diff W_r^H, \int_0^b g(r) \diff W_r^H \right)
= \int_0^a \int_0^b f(u) g(v) \phi(u-v) \diff u \diff v.
\end{align*}
\end{prop}

%%%%%%%%%%%%%%%%%%%%%%%%%%%%%%%%%%%%%%%%%%%%%%%%%%%%%%%%%%%%%%%%%%%%%%%%%%%%%%%%%%%%%%%%%%%%%%%%%%%%%%%%%%%%%%%%%%%%%%%%%%%%%%%%%%%%%
\subsubsection{Stochastic Differential Equations Driven by Fractional Brownian Motion}

After establishing a suitable stochastic integral with respect to the fractional Brownian motion, we consider stochastic differential equations (SDEs) given by:
\begin{align}\label{eqGenSDE}
\diff X_t = b(t, X_t) \diff t + \sigma(t) \diff W_t^H, \quad X_0 = x_0 \in \R,
\end{align}
The solution satisfies the integral formulation
\begin{align*}
X_t = x_0 + \int_0^t b(r, X_r) \diff r + \int_0^t \sigma(r) \diff W_r^H, ~~\mbox{a.s.},
\end{align*}
where the stochastic integral was constructed in Section~\ref{int:theory}.
Under certain classical regularity assumptions, existence and uniqueness of solutions for~\eqref{eqGenSDE} can be proven. For more details, see \cite[Theorem~D.2.4]{StochCalcforfBm}.

\begin{theorem}\label{thmExUniqOfSDE}
Let $b: [0, \infty) \times \R \rightarrow \R$ be globally Lipschitz in both variables, $\sigma \in \C^1([0,\infty))$ with $\sigma$ and $\frac{\diff}{\diff t}\sigma$ globally Lipschitz. Then for every $T > 0$ the SDE \eqref{eqGenSDE} has a unique continuous solution on $[0,T]$ a.s..
\end{theorem}

In this work one case we have to consider is a time-dependent linear drift, i.e., $b(t,\cdot):\mathbb{R}\to\mathbb{R}$ is linear with $b(t,x):=A(t)x$ for every $t\in[0,\infty)$ and $x\in\mathbb{R}$. In this case, the solution of~\eqref{eqGenSDE} is given by the variation of constants formula/Duhamel's formula and is called non-autonomous fractional\ Ornstein-Uhlenbeck process.

\begin{theorem}[Non-autonomous Fractional Ornstein-Uhlenbeck Process]\label{solOUprocess}
Let $A, B : [0, \infty) \rightarrow \mathbb{R}$. Suppose that $A$ is globally Lipschitz and uniformly bounded, and $B \in \mathcal{C}^1 ([0, \infty))$ with $B$ as well as $\frac{\diff}{\diff t}B$ globally Lipschitz.
Then there exists an a.s.~unique solution to the stochastic differential equation
\begin{align}\label{lin:eq}
\diff X_t = A(t) X_t \diff t + B(t) \diff W_t^H, \quad X_0 = x_0 \in \R
\end{align}
which satisfies the variation of constants formula
\begin{align*}
X_t = e^{A(t)} x_0 + \int_0^t e^{\int_r^t A(u) \diff u} B(r) \diff W_r^H,~~\mbox{\em{a.s}}.
\end{align*}
\end{theorem}

\begin{rem}
Note that all the results discussed in this subsection extend to higher dimensions, since all previous steps can be done component-wise. Namely, for $m\geq 1$ we mention.
\begin{itemize}
  \item [(R1)] We call $(W^{H}_{t})_{t\geq 0}$ an $m$-dimensional fractional Brownian motion if $W^{H}_{t}:=\sum\limits_{k=1}^{m} W^{k,H}_{t} e_{k}$, where $(e_{k})_{k\geq 1}$ is a basis in $\mathbb{R}^{m}$ and $(W^{k,H}_{t})_{t\geq 0}$, $k=1\ldots m$, are independent one-dimensional fractional Brownian motions with the same Hurst index $H$.
  \item [(R2)]\label{rem_uniquenesSDEhigherDim} Naturally, existence and uniqueness of SDEs in higher dimension carry over from Theorem \ref{thmExUniqOfSDE} under the same assumptions respectively. In particular, for coefficients $A,B:[0,\infty)\to\mathbb{R}^{m}$ with $m\geq 1$, satisfying the same assumptions as in Theorem~\ref{solOUprocess}, the solution of~\eqref{lin:eq} is given by
\begin{align*}
  X_t = \Phi(t,0) x_0 + \int_0^t \Phi(t,r) B(r) \diff W_r^H,~~\mbox{ a.s.,}
\end{align*}
where $\Phi$ denotes the fundamental solution of $x_t' = A(t)x_t$ and  $(W^{H}_{t})_{t\geq 0}$ is an $m$-dimensional fractional Brownian motion.
\end{itemize}
\end{rem}

%%%%%%%%%%%%%%%%%%%%%%%%%%%%%%%%%%%%%%%%%%%%%%%%%%%%%%%%%%%%%%%%%%%%%%%%%%%%%%%%%%%%%%%%%%%%%%%%%%%%%%%%%%%%%%%%%%%%%%%%%%%%%%%%%%%%%
\subsection{Useful Estimates of Gaussian Processes}\label{secEstGaussProc}

The fact that fBm is not a semi-martingale restricts the repository of known inequalities (such as Doob or Burkholder-Davies-Gundy) to establish sample paths estimates. A crucial property of fBm we shall exploit is its Gaussianity. In this section we will describe some useful estimates for exit times of certain Gaussian processes, which will be helpful for our analysis in the upcoming sections.\medskip

We first state the next auxiliary result regarding the Laplace transform of a Gaussian process. This was established in~\cite{decreusefond2008hitting} by means of Malliavin calculus. 

\begin{lem}{\em (Proposition~3.5~\cite{decreusefond2008hitting})}\label{laplace}
Let $(Y_t)_{t\geq 0}$ be a centered Gaussian process with $Y_0 = 0$ and covariance function $R(s,t) := \e[Y_s Y_t]$ satisfying the following conditions:
\begin{enumerate}[(i)]
\item [i)] $\frac{\partial}{\partial s}R(s,t)$ exists and is continuous as a function on $[0, \infty) \times [0, \infty)$,
\item [ii)] $\frac{\partial}{\partial s}R(s,t) \geq 0$ for all $t,s \geq 0$,
\item [iii)] $\e[\left| Y_t - Y_s \right|^2] > 0$ for all $t > s \geq 0$,
\item [iv)] $\limsup_{t \rightarrow \infty} Y_t = \infty$ a.s.
\end{enumerate}
Then for any $\alpha>0$:
\begin{align}\label{b:laplace}
\e[\exp(-\alpha V_{\tau_c})] \leq \exp(-c \sqrt{2\alpha}),
\end{align}
where $V_{t} := R(t,t)$ and $\tau_{c} := \inf \{r > 0: Y_r \geq c \}$.
\end{lem}

In addition, we require the following form of Chebychev's inequality.

\begin{lem}\label{lem_chebychev}
Let $\varphi : \R \rightarrow [0, \infty)$ be measurable, $Z$ a random variable and $A \in \mathcal{B}(\R)$. Then
\begin{align*}
\inf\{\varphi(y): \:y \in A\} \p(Z \in A)\leq \e[\varphi(Z)].
\end{align*}
\end{lem}

\begin{proof}
Under these assumptions we have
\begin{align*}
\inf\{\varphi(y): \:y \in A\} \ind_{\{Z \in A \}} 
\leq \varphi(Z) \ind_{\{Z \in A \}}
\leq \varphi(Z).
\end{align*}
Taking expectation in the above inequality yields the result.
\end{proof}

\begin{lem}\label{bernsteintypeineq} Let $c>0$ and $(Y_{t})_{t\geq 0}$ be a centered Gaussian process with $Y_0=0$ satisfying the assumptions i)-iv) of Lemma~\ref{laplace}. Then, for its exit time $\tau_{c} := \inf \{r > 0: Y_r \geq c \}$, the following estimate holds:
\begin{align*}
\p( \tau_c < t) \leq \exp \left( - \frac{1}{2}\frac{c^2}{\var(Y_t)} \right).
\end{align*}
\end{lem}
\begin{proof}
Applying Lemma~\ref{lem_chebychev} for $Z:=\tau_c$, $\varphi(r) := \exp\left(\alpha V_r\right)$ and $A := (0,t)$ we can bound the probability $\p( \tau_c < t)$ together with~\eqref{b:laplace} as follows:
\begin{align*}
\p( \tau_c < t) &\leq \exp\left( \alpha \sup_{0 < r < t} V_r \right) \e\left[\exp \left(-\alpha V_{\tau_c}\right) \right] \\
&\leq  \exp\left( \alpha \sup_{0 < r < t} V_r -c \sqrt{2\alpha}\right),
\end{align*}
for all $\alpha > 0$. Optimizing over $\alpha$ and noticing that $\sup_{0 < r < t} V_r = \var(Y_t)$ proves the statement.
\end{proof}

The previous lemma established a Bernstein-type inequality solely relying on certain properties of the covariance function of Gaussian processes. Another useful estimate is given by \cite[Theorem~D.4]{AsMethforGaussProc}, which is based on Slepian's Lemma~\cite{Slepian}.  

\begin{theorem}
\label{gaussineqforHolderCont}
Let $T>0$ and $(Y_t)_{t\in[0,T]}$ be a centered Gaussian process with a.s.~continuous trajectories. Assume that $(Y_t)_{t\in[0,T]}$ is a.s.~mean-square H\"older continuous, i.e.~there are constants $G$ and $\gamma$ such that
\begin{align*}
\e \left[ (Y_t - Y_s)^2 \right] \leq G \left| t - s \right|^{\gamma} \hspace{0.2cm} \text{ for all }t,s \in [0,T].
\end{align*}
Then there exists a constant $K := K(G,\gamma)$ such that for $c > 0$ and $A \subset [0,T]$
\begin{align*}
\p \left( \sup_{t \in A} Y_t > c\right) \leq K T c^{\frac{2}{\gamma}}\exp\left(- \frac{c^2}{2\sigma^2(A)} \right),
\end{align*}
where $\sigma^2(A) := \sup_{t \in A} \var \left( Y_t\right)$.
\end{theorem}

This estimate can be sharpened if we restrict ourselves to the interval of interest.

\begin{cor} 
\label{CorgaussineqforHolderCont}
Let $T>0$ and $(Y_t)_{t\in[0,T]}$ be a centered Gaussian process with a.s.~continuous trajectories.
Assume that $(Y_t)_{t\in[0,T]}$ is a.s.~mean-square H\"older continuous, i.e.~there are constants $G$ and $\gamma$, such that
\begin{align*}
\e \left[ (Y_t - Y_s)^2 \right] \leq G \left| t - s \right|^{\gamma} \hspace{0.2cm} \text{ for all }t,s \in [0,T].
\end{align*}
Then there exists a constant $K := K(G,\gamma)$ such that for $c > 0$ and $0 \leq a < b \leq T$
\begin{align*}
\p \left( \sup_{a \leq t < b} Y_t > c\right) \leq K (b-a) c^{\frac{2}{\gamma}}\exp\left(- \frac{c^2}{2\sigma^2} \right),
\end{align*}
where $\sigma^2 := \sup_{a \leq t < b} \var \left( Y_t\right)$.
\end{cor}

\begin{proof}
$(Z_t)_t$ with $Z_t := Y_{t+a}$ satisfies the assumptions of Theorem~\ref{gaussineqforHolderCont} on $[0,b-a]$.
\end{proof}

%%%%%%%%%%%%%%%%%%%%%%%%%%%%%%%%%%%%%%%%%%%%%%%%%%%%%%%%%%%%%%%%%%%%%%%%%%%%%%%%%%%%%%%%%%%%%%%%%%%%%%%%%%%%%%%%%%%%%%%%%%%%%%%%%%%%%
\section{The One-Dimensional Case}
\label{ch_1Dcase}

In this section, we investigate the dynamics of a planar stochastic fast-slow system driven by fractional Brownian motion $(W_s^H)_{s\geq 0}$ with Hurst parameter $H >\frac{1}{2}$:
\begin{align*}
\diff x_s &= f(x_s,y_s,\eps) \diff s + \sigma F(y_s)  \diff W^H_s, \\
\diff y_s &= \eps \diff s.
\end{align*}
Its equivalent formulation in slow time, i.e. for $t = \eps s$ is
\begin{equation}\label{1DSysinSlowTime}
\begin{aligned}
\diff x_t &= \frac{1}{\eps} f(x_t, y_t,\eps) \diff t + \frac{\sigma}{\eps^H} F(y_t)  \diff W^H_t, \\
\diff y_t &= 1 \diff t,
\end{aligned}
\end{equation}
using the self-similarity of fBm~\eqref{self:sim}. We are interested in the normally hyperbolic stable case and therefore make the following assumptions.

\begin{ass} \label{assStable}
Stable Case
\begin{enumerate}
\item [1)] \emph{Regularity:} The functions $f \in \C^2(\R \times [0,\infty)^2;\R)$ and $F \in \C^1([0,\infty);(0,\infty))$, as well as all their existing derivatives up to order two are uniformly bounded on an interval $I = [0, \infty)$ or $I = [0,T]$, $T > 0$, by a constant $M > 0$.
% in $\R$
\item [2)]\emph{Critical manifold:} There is an $x^*: [0,\infty) \rightarrow \R$ such that
\begin{align*}
f(x^*(t),t,0) = 0
\end{align*}
for all $t \in [0,\infty)$.
\item [3)] \emph{Stability:} For $a(t) := \partial_x f(x^*(t),t,0)$ there is $a > 0$ such that
\begin{align*}
a(t) \leq -a
\end{align*}
for all $t \in [0,\infty)$.
\end{enumerate}
\end{ass}
Under these assumptions,~\eqref{1DSysinSlowTime} has a unique global solution according to Theorem~\ref{thmExUniqOfSDE}. Furthermore, the deterministic system, i.e., for $\sigma = 0$, given by
\begin{align*}
\eps \frac{\diff}{\diff t} x_t = \eps \dot{x_t}&= f(x_t, t,\eps) 
\end{align*}
has an asymptotically slow manifold $\bar{x}(t, \eps) = x^*(t) + \Onot(\eps)$ for $\eps > 0$ small enough due to Fenichel-Tikhonov (Theorem \ref{TihonovThm}). We expect that, given small noise $0 < \sigma \ll 1$, the trajectories of \eqref{1DSysinSlowTime} starting sufficiently close to $\bar{x}(0, \eps)$ remain in a properly chosen neighborhood of $\bar{x}(t, \eps)$ for a long time with high probability. Our goal will be to make this idea rigorous  by pursuing the following steps. We first linearize the system around the slow manifold to get an SDE describing the deviations induced by the noise. This helps us obtain a simple description of a suitable neighborhood by using the fast-slow structure inherited by the variance of the system. Then, using this neighborhood, we deduce sample paths estimates for the linear case starting on the slow manifold. To complete the discussion we generalize the result to the non-linear case starting sufficiently close to the slow manifold, that is, such that in the deterministic case solutions are still attracted by the slow manifold. This general strategy inspired by~\cite{NoiseSlowFastSys}, where a similar system driven by Brownian motion (Hurst parameter $H = \frac{1}{2}$) is analyzed. Yet, the several techniques used in~\cite{NoiseSlowFastSys} do not generalize to fBm. 

%%%%%%%%%%%%%%%%%%%%%%%%%%%%%%%%%%%%%%%%%%%%%%%%%%%%%%%%%%%%%%%%%%%%%%%%%%%%%%%%%%%%%%%%%%%%%%%%%%%%%%%%%%%%%%%%%%%%%%%%%%%%%%%%%%%%%
\subsection{The Linearized System}
The deterministic system
\begin{align*}
\eps \frac{\diff}{\diff t} x_t = \eps \dot{x_t} &= f(x_t, t, \eps) 
\end{align*}
has an asymptotically stable slow manifold $\bar{x}(t, \eps) = x^*(t) + \Onot(\eps)$ due to Fenichel-Tikhonov (Theorem \ref{TihonovThm}). 
As already outlined, our first step is to examine the behavior of the linearized system around $\bar{x}(t, \eps)$. For a solution $(x_t)_{t \in I}$ of $\eqref{1DSysinSlowTime}$ we set $\xi_t: = x_t - \bar{x}(t, \eps)$. Then $(\xi)_{t \in I}$ satisfies the equation
\begin{equation}\label{1DTaylor}
\begin{aligned}
\diff \xi_t &= \frac{1}{\eps} \left[ f(\xi_t + \bar{x}(t, \eps), t,\eps) - f(\bar{x}(t, \eps), t,\eps)\right] \diff t + \frac{\sigma}{\eps^H} F(t)  \diff W^H_t \\
&= \frac{1}{\eps} [a(t,\eps)\xi_t + b(\xi_t,t,\eps)] \diff t + \frac{\sigma}{\eps^H} F(t) \diff W^H_t,
\end{aligned}
\end{equation}
where
\begin{align*}
a(t,\eps) &= \partial_x f(\bar{x}(t, \eps), t, \eps) = \partial_x f(x^*(t), t, 0) + \Onot(\eps), \\
\left| b(x,t,\eps) \right| &\leq M \left| x \right|^2,
\end{align*}
by Taylor's remainder theorem. Due to the uniform boundedness of the derivatives of $f$ one can show that the $\Onot(\eps)$-term is negligible on finite time scales as $\eps \rightarrow 0.$ Therefore, we restrict ourselves without loss of generality to the analysis of the linearization
\begin{align}\label{1Dlinearization}
\diff \xi_t = \frac{1}{\eps} a(t) \xi_t \diff t + \frac{\sigma}{\eps^H} F(t) \diff W^H_t.
\end{align}
Examining the process starting on the slow manifold now corresponds to investigating the unique explicit solution of \eqref{1Dlinearization} for initial value $\xi_0 = 0$, which is given by the fractional Ornstein-Uhlenbeck process (recall Theorem~\ref{solOUprocess})
\begin{align*}
\xi_t = \int_0^t e^{\alpha(t,u)/\eps}\frac{\sigma}{\eps^H} F(u) \diff W^H_u,
\end{align*}
where $\alpha(t,u) := \int_u^t a(r) \diff r$. In order to define a proper neighborhood, where the fractional Ornstein-Uhlenbeck process $(\xi_t)_{t \in I}$ is going to stay with high probability, we use the variance $\var(\xi_t)$ as an indicator for the deviations at time $t$. According to Proposition~\ref{cov:integral}, the variance is given by
\begin{align*}
\sigma^2 w(t) := \var(\xi_t) = \frac{\sigma^2}{\eps^{2H}} \int_0^t \int_0^t e^{\alpha(t,u)/\eps} e^{\alpha(t,v)/\eps} F(u) F(v) H (2H - 1) \left| u - v \right|^{2H-2} \diff u \diff v.
\end{align*}
As we would like to see dynamics of $t \mapsto \var(\xi_t)$, we rescale it by $\frac{1}{\sigma^2}$ to get rid of the small parameter $\sigma \ll 1$, which only changes the order of magnitude of the system. It turns out that $t \mapsto w(t)$ inherits the fast-slow structure from the SDE, which yields a particularly simple approximation of the variance.

\begin{prop}\label{fastSlowVariance}
The so-called renormalized variance $w$ satisfies the fast-slow ODE
\begin{align}\label{fastSlowVarianceEqn}
\eps\frac{\diff}{\diff t} w(t) = \eps \dot{w}(t) &= 2 a(t) w(t) + 2 F(t) H (2H - 1) \int_0^t \frac{1}{\eps^{2H-1}} e^{\alpha(t,u)/\eps} F(u) (t-u)^{2H-2} \diff u.
\end{align}
In particular, there is a (globally) asymptotically stable slow manifold of the system of the form
\begin{align}\label{SlowManifVariance}
\zeta(t) = \frac{F(t)^2}{\left|a(t)\right|^{2H}} H \Gamma(2H) + \Onot (\eps).
\end{align}
\end{prop}

\begin{proof}
Differentiating $t \mapsto w(t)$ yields
\begin{align*}
\frac{\diff}{\diff t} w(t) = \dot{w}(t) &= 2 \frac{a(t)}{\eps} w(t) + 2 \frac{F(t)}{\eps^{2H}} H (2H - 1) \int_0^t e^{\alpha(t,u)/\eps} F(u) (t-u)^{2H-2} \diff u \\
\Longleftrightarrow \eps \frac{\diff}{\diff t} w(t) = \eps \dot{w}(t) &= 2 a(t) w(t) + 2 F(t) H (2H - 1) \int_0^t \frac{1}{\eps^{2H-1}} e^{\alpha(t,u)/\eps} F(u) (t-u)^{2H-2} \diff u.
\end{align*}
In order to be able to take the singular limit $\eps \rightarrow 0$ and apply Fenichel-Tikhonov (Theorem \ref{TihonovThm}) we need to prove sufficient regularity in $\eps= 0$; continuous differentiability will be enough for the approximation of the slow manifold with the critical manifold up to order $\Onot(\eps)$. To do this, rewrite the integral by substituting $v = \frac{t-u}{\eps}$
\begin{align*}
\int_0^t \frac{1}{\eps^{2H-1}} e^{\alpha(t,u)/\eps} F(u) (t-u)^{2H-2} \diff u 
=& \int^{\frac{t}{\eps}}_0 e^{\alpha(t,t-\eps v)/\eps}  F(t - \eps v) v^{2H-2} \diff v \\
\underset{\eps \rightarrow 0}{\longrightarrow}& F(t) \int^{\infty}_0 e^{a(t)v}  v^{(2H-1)-1} \diff v \\
=& \frac{F(t)}{\left| a(t) \right|^{2H-1}} \Gamma (2H-1).
\end{align*}
To see that the right hand side of \eqref{fastSlowVarianceEqn} is continuously differentiable in $\eps = 0$ it is sufficient to check it for the integral term
\begin{align*}
&\frac{\diff}{\diff \eps} \left( \int^{\frac{t}{\eps}}_0 e^{\alpha(t,t-\eps v)/\eps}  F(t - \eps v) v^{2H-2} \diff v \right) \\
=\: &-\frac{t}{\eps^2} e^{\int_0^{\frac{t}{\eps}}a(t-\eps r) \diff r} F(0)\left(\frac{t}{\eps}\right)^{2H-2} \\
&+ \int_0^{\frac{t}{\eps}} e^{\int_0^va(t-\eps r) \diff r} \left( -\int_0^va'(t-\eps r)r \diff r F(t - \eps v) - F'(t - \eps v)v \right) v^{2H-2} \diff v,
\end{align*}
which has an existing limit for $\eps \rightarrow 0,$ because the exponential term goes to $0$ faster than the polynomial term diverges.
Now taking the singular limit $\eps \rightarrow 0$ gives the slow subsystem
\begin{align*}
0 &= 2a(t)w(t) + 2 \frac{F(t)^2}{\left|a(t)\right|^{2H-1}} H(2H-1) \Gamma(2H-1).
\end{align*}
The critical manifold is hence given by
\begin{align*}
w^*(t) = \frac{F(t)^2}{\left|a(t)\right|^{2H}} H (2H -1) \Gamma(2H-1).
\end{align*}
Using integration by parts we can rewrite $(2H-1)\Gamma(2H-1) =  \Gamma(2H)$, so that the critical manifold can also be written as
\begin{align*}
w^*(t) = \frac{F(t)^2}{\left|a(t)\right|^{2H}} H \Gamma(2H).
\end{align*}
By Theorem~\ref{TihonovThm}, the ODE~\eqref{fastSlowVarianceEqn} has a solution of the form
\begin{align*}
\zeta(t) = \frac{F(t)^2}{\left|a(t)\right|^{2H}} H \Gamma(2H) + \Onot (\eps),
\end{align*}
which is asymptotically stable due to Assumption~\ref{assStable} {\em 3}. This stability property is even global in this case because the ODE \eqref{fastSlowVarianceEqn} is linear.
\end{proof}

As expected, the critical manifold depends on the Hurst parameter $H$. For $H \in (\frac{1}{2},1)$ we have $H\Gamma(2H) \in (\frac{1}{2},1)$. This means that the only possible structural change of the critical manifold under variation of $H$ is induced by the factor $\frac{1}{|a(t)|^{2H}}$. Furthermore, as $\Gamma(x) \rightarrow 1$ for $x \rightarrow 1$, the slow subsystem for $H \rightarrow \frac{1}{2}$ reads
\begin{align*}
0 &= 2a(t)v(t) + F(t)^2, \\
\dot{t} &= 1,
\end{align*}
which coincides with the slow subsystem we would obtain in the case of Brownian motion noise, which exactly corresponds to $H = \frac{1}{2}$. 
 
\begin{rem}
The proof of Proposition \ref{fastSlowVariance} only shows that $\zeta$ is  $\mathcal{C}^1$ in $\eps$ and $\mathcal{C}^1$ in the time $t$. Depending on the properties of $f,$ we expect $\zeta$ to even have higher regularity. However, this fact is not required in the following considerations.
\end{rem}

Proposition~\ref{fastSlowVariance} already states that the slow manifold is a good indicator for the size of the set we are looking for as $t \mapsto \frac{1}{\sigma^2} \var (\xi_t)$ (as a solution of \eqref{fastSlowVarianceEqn} with initial datum $w(0) = 0$) is attracted by the slow manifold.
In this particular case we can explicitly state the exponentially fast approach due to the structure of the linear equation
\begin{align} \label{relSlowManVar}
\var(\xi_t) = \sigma^2 \left(\zeta(t) - e^{2\alpha(t)/\eps}\zeta (0) \right),
\end{align}
where $\alpha(t) := \alpha(t,0)$. Even more is known about the properties of $\zeta$. Due to the uniform boundedness assumption on $f$ and $F$ we get that the difference between $\zeta$ and $\frac{F(t)^2}{\left|a(t)\right|^{2H}} H \Gamma(2H)$ is actually in uniform $t.$ 
This implies that for $\eps$ small enough there are $\zeta^+$ and $\zeta^-$ such that 
\begin{align*}
\zeta^+ \geq \zeta(t) \geq \zeta^- >0 \quad \text{ for all } t \in I.
\end{align*}
The goal is now to prove that the stochastic process $(\xi)_{t \in I}$ is concentrated in sets of the form
\begin{align*}
\mathcal{B}(h) := \{ (x,t) \in \R \times I: \left| x \right| < h \sqrt{\zeta(t)}\}.
\end{align*}
To get a better understanding of what to expect, note that the probability that $\xi$ leaves $\mathcal{B}(h)$ at time $t$ can be bounded by using the inequality $\p(X>c) \leq \exp\left(- \frac{c^2}{2\var(X)}\right)$, which holds for any centered Gaussian random variable $X$. This further leads to
\begin{align}\label{estimateatt}
\p \left(\frac{\left| \xi_t \right|}{\sqrt{\zeta(t)}} \geq h \right) 
\leq \exp\left(- \frac{h^2 \zeta(t)}{2\var(\xi_t)}\right) 
\leq \exp\left(- \frac{h^2}{2\sigma^2}\right).
\end{align}
Of course, the probability that $(\xi)_{t \in I}$ has exited $\mathcal{B}(h)$ in the interval $[0,t]$ at least once 
\begin{align*}
\p \left( \sup_{0 \leq r \leq t} \frac{\left| \xi_r \right|}{\sqrt{\zeta(r)}} \geq h\right) =
\p \left( \tau_{\mathcal{B}(h)} < t \right)
\end{align*}
is larger, where $\tau_{\mathcal{B}(h)} := \inf \{ r > 0: (\xi_r,r) \notin \mathcal{B}(h)\}$ is the first time $(\xi)_{t \in I}$ has exited $\mathcal{B}(h)$.
We will present a few approaches to estimate this probability in the following, using the inequalities we have established in Section \ref{secEstGaussProc}. The increase of probability compared to \eqref{estimateatt} is simply indicated by the prefactor of $\exp\left(- \frac{h^2}{2\sigma^2}\right)$.

\begin{rem}
\label{remCriticalManifAlsoOK}
It could also have been possible to define the neighborhood $\mathcal{B}(h)$ by considering the critical manifold $w^*$, i.e.~to define
\begin{align*}
\mathcal{B^*}(h) := \{ (x,t) \in \R \times I: \left| x \right| < h \sqrt{w^*(t)}\}.
\end{align*}
This will yield the same bounds on the exit times, which we will establish in the following because the difference between $\zeta(t)$ and $w^*(t)$ is only in $\Onot(\eps),$ which is of the same order as the order we obtain by approximating with the slow manifold $\zeta$ anyways.
\end{rem}

%%%%%%%%%%%%%%%%%%%%%%%%%%%%%%%%%%%%%%%%%%%%%%%%%%%%%%%%%%%%%%%%%%%%%%%%%%%%%%%%%%%%%%%%%%%%%%%%%%%%%%%%%%%%%%%%%%%%%%%%%%%%%%%%%%%%%
\subsubsection{Variant 1}
\label{sec_var1}

The first approach is based on the result on exit times of Gaussian processes with sufficiently regular and increasing covariance function, as stated in Lemma~\ref{bernsteintypeineq}. 

\begin{theorem} 
\label{thmvariantbernsteinineq}
Let $t \in I.$ Then under Assumption \ref{assStable} for any $h>0$ the following estimate holds true for $\eps>0$ sufficiently small
\begin{align*}
\p \left( \sup_{0 \leq r < t} \frac{\left| \xi_r \right|}{\sqrt{\zeta(r)}} \geq h \right)
\leq 2e \left\lceil \frac{\left| \alpha(t) \right|}{\eps} \frac{h^2}{\sigma^2} (1 + \Onot (\eps))\right\rceil\exp\left( -\frac{h^2}{2\sigma^2} \right).
\end{align*}
\end{theorem}
\begin{proof}
In order to apply the estimate given in Lemma \ref{bernsteintypeineq} to our problem observe that $\xi$ may not satisfy all the assumptions. First of all we need well-definedness of the Ornstein-Uhlenbeck process over the whole non-negative real line $[0,\infty)$; this is guaranteed by Assumption \ref{assStable}. In addition, we consider the process given by $X_t = e^{-\alpha(t)/\eps}\xi_t$. Note that the event that $\xi_t$ exceeds a certain level $c \geq 0$ corresponds exactly to the event of $X_t$ exceeding $e^{-\alpha(t)/\eps}c$, that is $\{ \xi_t \geq c\} = \{ X_t \geq e^{-\alpha(t)/\eps}c\}$. Unfortunately this observation does not carry over to $\sup_{0 \leq r \leq t} X_r$ and $\sup_{0 \leq r \leq t} \xi_r$. In fact, a priori we only have the relation $\{ \sup_{0 \leq r \leq t} \xi_r \geq c\} \subset \{ \sup_{0 \leq r \leq t} X_r \geq c\}$, which yields a too strong estimate in the end as we are increasing the variance by an exponentially increasing factor, while maintaining the same exit level! A way to overcome this is to partition the interval $[0,t]$ to suitable subintervals $[t_i,t_{i+1})$ to obtain the relation $\{\sup_{t_i \leq r \leq t_{i+1}} \xi_r \geq c\} \subset \{ \sup_{t_i \leq r \leq t_{i+1}} X_r \geq e^{-\alpha(t_i)/\eps}c\}$. This partition will also turn out to be useful to control the variance with the slow manifold in the exponential of the estimate. The covariance of $(X_t)_{t \in I}$ is given by
\begin{align*}
R(t,r) := \e[X_tX_r] = \frac{\sigma^2}{\eps^{2H}} \int_0^t \int_0^r e^{-\alpha(u)/\eps} e^{-\alpha(v)/\eps} F(u) F(v) H (2H - 1) \left| u - v \right|^{2H-2} \diff u \diff v.
\end{align*}
By the theorem for differentiation of parameter dependent integrals we deduce that $\frac{\partial}{\partial t} R(t,r)$ exists and is continuous in $[0,\infty)\times[0,\infty)$.
Furthermore, $\frac{\partial}{\partial t} R(t,r) \geq 0$ for all $t,r \geq 0$ is an immediate consequence of the fundamental theorem of calculus as the integrand is always greater equal than $0$. This already implies assumption (i) und (ii) of Lemma \ref{bernsteintypeineq}. For assumption (iii) it suffices to observe that
\begin{align*}
X_{t_1} - X_{t_2} = \int_{t_2}^{t_1} e^{-\alpha(u)/\eps}\frac{\sigma}{\eps^H} F(u) \diff W^H_u
\end{align*}
is a Gaussian process with nonzero variance.
Assumption (iv) follows by Corollary \ref{cor_limsupxi}. This implies that $(e^{-\alpha(t)/\eps}\xi_t)_{t \in I}$ satisfies a Bernstein-type inequality due to Lemma \ref{bernsteintypeineq}
\begin{align} \label{bernsteinForXt}
\p\left( \sup_{0 \leq r \leq t} e^{-\alpha(r)/\eps}\xi_r > c\right) \leq \exp \left( - \frac{1}{2}\frac{c^2}{e^{-2\alpha(t)/\eps} \var(\xi_t)} \right).
\end{align}
After having established this result, we can proceed as in the proof of Proposition 3.1.5 \cite{NoiseSlowFastSys}. For $\gamma \in (0, 1/2)$ let $0 = t_0 < t_1 < \ldots < t_N$ be a partition containing the interval $[0,t]$ such that 
\begin{align*}
-\alpha(t_{i+1},t_i) = \eps \gamma \quad \quad \text{for} \quad \quad 0 \leq i < N = \left\lceil \frac{\left| \alpha(t) \right|}{\eps \gamma} \right\rceil .
\end{align*}
(Note that $t_N \geq t$ is possible. But this only increases the estimate on the probability slightly. As we would like to optimize over $\gamma$ in the end, it does not make sense to fix it to obtain $t_N = t$.)
For each $i \in \{0, \ldots, N-1\}$ we have
\begin{align*}
&\p \left( \sup_{t_i \leq r < t_{i+1}} \frac{\left| \xi_r \right|}{\sqrt{\zeta(r)}} \geq h \right) \\
\leq\: & 2 \p \left( \sup_{0 \leq r < t_{i+1}} e^{-\alpha(r)/\eps}\xi_r \geq h \inf_{t_i \leq r < t_{i+1}} e^{-\alpha(r)/\eps}\sqrt{\zeta(r)} \right) \\
\overset{\eqref{bernsteinForXt}}{\leq} & 2 \exp\left( -\frac{1}{2} \frac{h^2 \inf_{t_i \leq r < t_{i+1}} e^{-2\alpha(r)/\eps}\zeta(r)}{e^{-2\alpha(t_{i+1})/\eps} \var(\xi_{t_{i+1}})} \right) \\
\overset{\eqref{relSlowManVar}}{\leq} & 2 \exp\left( -\frac{h^2}{2\sigma^2} \frac{ \inf_{t_i \leq r < t_{i+1}}\zeta(r)}{\zeta(t_{i+1})} e^{2\alpha(t_{i+1},t_{i})/\eps}\right) \\
\leq\: & 2 \exp\left( -\frac{h^2}{2\sigma^2} e^{-2\gamma}(1 - \Onot (\eps \gamma))\right),
\end{align*}
where the last inequality follows by $\dot{\zeta}(r) = \Onot(1)$, which is proven in Lemma \ref{derivOfOrder1}.
Now by subadditivity of the probability measure
\begin{align*}
\p \left( \sup_{0 \leq r < t} \frac{\left| \xi_r \right|}{\sqrt{\zeta(r)}} \geq h \right)
&\leq \sum_{i=0}^{N-1} \p \left( \sup_{t_i \leq r < t_{i+1}} \frac{\left| \xi_r \right|}{\sqrt{\zeta(r)}} \geq h \right) \\
&\leq 2 \left\lceil \frac{\left| \alpha(t) \right|}{\eps \gamma} \right\rceil \exp\left( \frac{h^2}{\sigma^2} \gamma(1 + \Onot (\eps))\right) \exp\left( -\frac{h^2}{2\sigma^2} \right),
\end{align*}
where the last inequality is due to $e^{-2\gamma} \geq 1 - 2\gamma$.
Due to monotonicity of $\lceil \cdot \rceil$ , it suffices to minimize
\begin{align*}
\gamma ~ \mapsto ~ \frac{\left| \alpha(t) \right|}{\eps \gamma} \exp\left( \frac{h^2}{\sigma^2} \gamma(1 + \Onot (\eps))\right)
\end{align*}
in order to find the minimal value of this estimate. Optimizing over $\gamma$ hence yields
\begin{align*}
\p \left( \sup_{0 \leq r < t} \frac{\left| \xi_r \right|}{\sqrt{\zeta(r)}} \geq h \right)
\leq 2e \left\lceil \frac{\left| \alpha(t) \right|}{\eps} \frac{h^2}{\sigma^2} (1 + \Onot (\eps))\right\rceil\exp\left( -\frac{h^2}{2\sigma^2} \right),
\end{align*}
which finishes the proof.
\end{proof}

\begin{lem}
\label{derivOfOrder1}
The slow manifold $\zeta(t)$ satisfies $\dot{\zeta}(t) = \Onot(1)$.
\end{lem}
\begin{proof}
Note that
\begin{align*}
w^*(t,\eps) = \frac{F(t)}{\left| a(t)\right|} H (2H - 1) \int_0^t \frac{1}{\eps^{2H-1}} e^{\alpha(t,u)/\eps} F(u) (t-u)^{2H-2} \diff u
\end{align*}
satisfies the corresponding invariance equation for the fast-slow ODE \eqref{fastSlowVarianceEqn} up to error $\Onot(\eps).$ 
This implies that $\zeta(t) = w^*(t,\eps) + \Onot(\eps)$. Plugging this representation of $\zeta$ into the ODE \eqref{fastSlowVarianceEqn} yields directly $\eps \dot{\zeta}(t) = \Onot(\eps)$.
\end{proof}

%%%%%%%%%%%%%%%%%%%%%%%%%%%%%%%%%%%%%%%%%%%%%%%%%%%%%%%%%%%%%%%%%%%%%%%%%%%%%%%%%%%%%%%%%%%%%%%%%%%%%%%%%%%%%%%%%%%%%%%%%%%%%%%%%%%%%
\subsubsection{Variant 2}\label{sec_var2}

The second approach uses the fact that $(\xi_t)_{t \in I}$ is mean-square H\"older continuous. This is also going to enable to control the deviations based upon Theorem \ref{gaussineqforHolderCont}.

\begin{theorem} \label{thmvariantgaussineqHoldercont}
Let $t \in I.$ Then under Assumption \ref{assStable} there is a constant $K = K(t,\eps,\sigma,H) > 0$, such that for any $h>0$ the following estimate holds true for $\eps>0$ sufficiently small
\begin{align*}
\p \left( \sup_{0 \leq r < t} \frac{\left| \xi_r \right|}{\sqrt{\zeta(r)}} \geq h \right) 
\leq K t \exp\left( \frac{h^2}{2\sigma^2}\Onot (\eps)\right) h^{\frac{1}{H}} \left( \frac{F_+^2}{a^{2H}} H \Gamma(2H) + \Onot (\eps) \right)^{\frac{1}{2H}} \exp\left( -\frac{h^2}{2\sigma^2} \right),
\end{align*}
where $F_+ := \sup_{r \in [0,\infty)} F(r) < \infty.$
\end{theorem}
\begin{proof}
Let $t>0$. In order to apply Theorem \ref{gaussineqforHolderCont} we have to prove mean-square H\"older continuity of $\left(\xi_r\right)_{r\in [0,t]}$. For $t\geq r_1 > r_2 \geq 0$
\begin{align}
&\e\left[ \left( \xi_{r_1} - \xi_{r_2} \right)^2\right]\notag \\
&=\frac{\sigma^2}{\eps^{2H}} \int_{0}^{r_1} \int_{0}^{r_1} \left( e^{\alpha(r_1,u)/\eps} - \mathds{1}_{\{u \leq r_2\}}e^{\alpha(r_2,u)/\eps}\right) \left( e^{\alpha(r_1,v)/\eps} - \mathds{1}_{\{v \leq r_2\}}e^{\alpha(r_2,v)/\eps}\right) \notag\\
& \qquad \cdot F(u) F(v) \phi(u-v) \diff u \diff v \notag \\
&=\frac{\sigma^2}{\eps^{2H}} \int_{0}^{r_2} \int_{0}^{r_2} \left( e^{\alpha(r_1,u)/\eps} - e^{\alpha(r_2,u)/\eps}\right) \left( e^{\alpha(r_1,v)/\eps} -e^{\alpha(r_2,v)/\eps}\right) F(u) F(v) \phi(u-v) \diff u \diff v \label{proofHoldercont1} \\
&\qquad + 2\frac{\sigma^2}{\eps^{2H}} \int_{r_2}^{r_1} \int_{0}^{r_2} \left( e^{\alpha(r_1,u)/\eps} - e^{\alpha(r_2,u)/\eps}\right)e^{\alpha(r_1,v)/\eps} F(u) F(v) \phi(u-v) \diff u \diff v \label{proofHoldercont2}\\
& \qquad + \frac{\sigma^2}{\eps^{2H}} \int_{r_2}^{r_1} \int_{r_2}^{r_1}  e^{\alpha(r_1,u)/\eps} e^{\alpha(r_1,v)/\eps} F(u) F(v) \phi(u-v) \diff u \diff v. \label{proofHoldercont3}
\end{align}
Lipschitz continuity of $r \mapsto e^{\alpha(r,u)/\eps}$ for arbitrary $u \geq 0$ (with Lipschitz constant $L = \frac{a_+}{\eps}$, where $a_+ := \sup_{r \in I} |a(r)|$) yields for \eqref{proofHoldercont1} 
\begin{align*}
&\frac{\sigma^2}{\eps^{2H}} \int_{0}^{r_2} \int_{0}^{r_2} \left( e^{\alpha(r_1,u)/\eps} - e^{\alpha(r_2,u)/\eps}\right) \left( e^{\alpha(r_1,v)/\eps} -e^{\alpha(r_2,v)/\eps}\right) F(u) F(v) \phi(u-v) \diff u \diff v \\
&\leq \frac{\sigma^2a_+^2}{\eps^{2H+2}} \int_{0}^{r_2} \int_{0}^{r_2} F(u) F(v) \phi(u-v) \diff u \diff v \left| r_1 - r_2 \right|^2 \\
&\leq \frac{\sigma^2a_+^2}{\eps^{2H+2}} F_+^2 \int_{0}^{t} \int_{0}^{t} \phi(u-v) \diff u \diff v \left| r_1 - r_2 \right|^2 
\leq \frac{\sigma^2a_+^2}{\eps^{2H+2}} F_+^2 t^{2} \left| r_1 - r_2 \right|^{2H}. 
\end{align*}
Similarly we can show for \eqref{proofHoldercont2} 
\begin{align*}
&2\frac{\sigma^2}{\eps^{2H}} \int_{r_2}^{r_1} \int_{0}^{r_2} \left( e^{\alpha(r_1,u)/\eps} - e^{\alpha(r_2,u)/\eps}\right)e^{\alpha(r_1,v)/\eps} F(u) F(v) \phi(u-v) \diff u \diff v \\
&\leq 2\frac{\sigma^2}{\eps^{2H}} \int_{r_2}^{r_1} e^{\alpha(r_1,v)/\eps} F(v) \int_{0}^{r_2}  F(u) \phi(u-v) \diff u \diff v \left| r_1 - r_2 \right| \\
&\leq \frac{\sigma^2}{\eps^{2H}} F_+^2 \left( (r_1 - r_2)^{2H} - (r_1^{2H} - r_2^{2H}) \right) \left| r_1 - r_2 \right| 
\leq \frac{\sigma^2}{\eps^{2H}} F_+^2 t (r_1 - r_2)^{2H}.
\end{align*}
Last but not least \eqref{proofHoldercont3} can be estimated as follows
\begin{align*}
&\frac{\sigma^2}{\eps^{2H}} \int_{r_2}^{r_1} \int_{r_2}^{r_1}  e^{\alpha(r_1,u)/\eps} e^{\alpha(r_1,v)/\eps} F(u) F(v) \phi(u-v) \diff u \diff v \\
&\leq \frac{\sigma^2}{\eps^{2H}} F_+^2 \int_{r_2}^{r_1} \int_{r_2}^{r_1}  \phi(u-v) \diff u \diff v 
= \frac{\sigma^2}{\eps^{2H}} F_+^2 \left| r_1 - r_2 \right|^{2H}.
\end{align*}
By combining the three estimates we obtain that, for a constant $G = G(t,\eps,\sigma,H) > 0$, it holds
\begin{align*}
&\e\left[ \left( \xi_{r_1} - \xi_{r_2} \right)^2\right] \leq G \left| r_1 - r_2 \right|^{2H}.
\end{align*}
Let $0 = t_0 < t_1 < \ldots < t_N = t$ be a partition of the interval $[0,t]$ such that
\begin{align*}
t_{i+1} - t_i = \Onot(\eps) \quad \text{for} \quad 0 \leq i < N.
\end{align*}
For each $i \in \{0, \ldots, N-1\}$ we have by Corollary \ref{CorgaussineqforHolderCont} for $K = K(G, 2H)$
\begin{align*}
&\p \left( \sup_{t_i \leq r < t_{i+1}} \frac{\left| \xi_r \right|}{\sqrt{\zeta(r)}} \geq h \right) \\
\leq\: & 2 \p \left( \sup_{t_i \leq r < t_{i+1}} \xi_{r} \geq h \inf_{t_i \leq r < t_{i+1}} \sqrt{\zeta(r)} \right) \\
\leq\: & K (t_{i+1}-t_i) \left( h \inf_{t_i \leq r < t_{i+1}} \sqrt{\zeta(r)} \right)^{\frac{1}{H}} \exp\left(- \frac{h^2 \inf_{t_i \leq r \leq t_{i+1}} \zeta(r)}{2\sup_{t_i \leq r \leq t_{i+1}}\var(\xi_{r})} \right) \\
\overset{\eqref{relSlowManVar}}{\leq} & K (t_{i+1}-t_i) h^{\frac{1}{H}} \left(  \sqrt{\frac{F_+^2}{a^{2H}} H \Gamma(2H) + \Onot (\eps)} \right)^{\frac{1}{H}} \exp\left(- \frac{h^2 \inf_{t_i \leq r \leq t_{i+1}} \zeta(r)}{2\sigma^2\sup_{t_i \leq r \leq t_{i+1}}\zeta(r)} \right) \\
\leq\: & K (t_{i+1}-t_i) h^{\frac{1}{H}} \left( \frac{F_+^2}{a^{2H}} H \Gamma(2H) + \Onot (\eps) \right)^{\frac{1}{2H}} \exp\left( -\frac{h^2}{2\sigma^2} (1 - \Onot (\eps))\right),
\end{align*}
where the last inequality follows by $\dot{\zeta}(r) = \Onot(1)$, see Lemma \ref{derivOfOrder1}.
This yields by the subadditivity of the probability measure
\begin{align*}
\p \left( \sup_{0 \leq r < t} \frac{\left| \xi_r \right|}{\sqrt{\zeta(r)}} \geq h \right)
&\leq \sum_{i=0}^{N-1} \p \left( \sup_{t_i \leq r < t_{i+1}} \frac{\left| \xi_r \right|}{\sqrt{\zeta(r)}} \geq h \right) \\
&\leq K t h^{\frac{1}{H}} \left( \frac{F_+^2}{a^{2H}} H \Gamma(2H) + \Onot (\eps) \right)^{\frac{1}{2H}} \exp\left( -\frac{h^2}{2\sigma^2} (1 - \Onot (\eps))\right) \\
&\leq K t \exp\left( \frac{h^2}{2\sigma^2}\Onot (\eps)\right) h^{\frac{1}{H}} \left( \frac{F_+^2}{a^{2H}} H \Gamma(2H) + \Onot (\eps) \right)^{\frac{1}{2H}} \exp\left( -\frac{h^2}{2\sigma^2} \right).
\end{align*}
Therefore, the proof is finished.
\end{proof}

%%%%%%%%%%%%%%%%%%%%%%%%%%%%%%%%%%%%%%%%%%%%%%%%%%%%%%%%%%%%%%%%%%%%%%%%%%%%%%%%%%%%%%%%%%%%%%%%%%%%%%%%%%%%%%%%%%%%%%%%%%%%%%%%%%%%%
\subsection{Comparison of the two Variants}

In this section, we will compare the two variants in view of varying the noise intensity given by $\sigma > 0$ and the time scale parameter $\eps > 0$. In order to better understand what to expect under these variations we first heuristically describe their effect on the underlying SDE
\begin{align*}
\diff \xi_t = \frac{1}{\eps} a(t) \xi_t \diff t + \frac{\sigma}{\eps^H} F(t) \diff W^H_t.
\end{align*}
We can directly see that a smaller $\sigma$ reduces the intensity of the fraction Brownian motion noise. In particular, as $\sigma$ is decreasing, the probability that $(\xi_r)_{r \in I}$ exits $\mathcal{B}(h)$ on some interval $[0,t]$ should become smaller. For a smaller $\eps$ the attraction towards the slow manifold becomes stronger, however also the noise intensity increases. We expect small deviations if $\frac{\sigma}{\eps^H}$ is sufficiently small.
\par
Suppose we are in the situation of Theorem \ref{thmvariantbernsteinineq} and Theorem \ref{thmvariantgaussineqHoldercont}.
To simplify the comparison the results of both variant 1
\begin{align*}
\p \left( \sup_{0 \leq r < t} \frac{\left| \xi_r \right|}{\sqrt{\zeta(r)}} \geq h \right)
&\leq 2e \left\lceil \frac{\left| \alpha(t) \right|}{\eps} \frac{h^2}{\sigma^2} (1 + \Onot (\eps))\right\rceil\exp\left( -\frac{h^2}{2\sigma^2} \right),
\intertext{and variant 2}
\p \left( \sup_{0 \leq r < t} \frac{\left| \xi_r \right|}{\sqrt{\zeta(r)}} \geq h \right) 
&\leq K t \exp\left( \frac{h^2}{2\sigma^2}\Onot (\eps)\right) h^{\frac{1}{H}} \left( \frac{F_+^2}{a^{2H}} H \Gamma(2H) + \Onot (\eps) \right)^{\frac{1}{2H}} \exp\left( -\frac{h^2}{2\sigma^2} \right),
\end{align*}
where $K = K(t,\eps,\sigma,H) > 0$, are displayed here.
Unfortunately, we do not know the dependence of $K$ on the other parameters, so we can only do a qualitative comparison up to some extent. By looking at the proof of Theorem \ref{thmvariantgaussineqHoldercont} we guess that $K$ is increasing in $G$, so that we assume for the forthcoming analysis that $K$ is increasing in $\sigma$, decreasing in $\eps$ and increasing in $t$.
In variant 1, we see the same interplay of $\sigma$ and $h$ as already observed in the analysis of $\eqref{estimateatt}$ because the exponential dominates the linear term in the prefactor, and the same holds true for variant 2 as long as $\eps \ll \frac{h^2}{\sigma^2}$ (for $\exp\left( \frac{h^2}{2\sigma^2}\Onot (\eps)\right)$ to remain relatively small), which is true in many applications.
For $\eps \rightarrow 0$ the estimate in variant 1 becomes larger, whereas in variant 2 it does not seem to have a huge effect on the bound. However, the increase might be hidden in $K$.
As the time $t$ increases, it obviously becomes more likely that $\xi_t$ has already exited $\mathcal{B}(h)$ at least once. In variant 1 this increase is displayed linearly in $t$ as $r \mapsto a(r)$ is uniformly bounded and thus $\alpha(t) = \Theta(t)$. Variant 2 shows an increase which is at least linear in $t$ because $K$ might be increasing in $t$ as well.
This means that in variant 1 we have to pick $h$ large enough such that $\frac{h^2}{\sigma^2}$ is significantly larger than $\ln\left(\frac{t}{\eps}\right)$. For variant 2 we have to choose $h$ in a suitable way that $\frac{h^2}{\sigma^2}$ is larger than $\frac{h^2}{\sigma^2}\Onot(\eps) + \frac{1}{H}\ln(h) + \ln(Kt)$. Although we cannot prove it, it seems that variant 1 yields a sharper bound.
Last but not least, note that the estimate in variant 1 coincides with the estimate derived for the Brownian motion case, see \cite[Proposition 3.1.5]{NoiseSlowFastSys}. The dependence of the Hurst parameter $H$ is completely hidden in the structure of the neighborhood $\mathcal{B}(h)$, which depends on the slow manifold. This also intuitively makes sense because we are ``almost'' dividing by the variance. Furthermore, in the Brownian motion case this estimate is quite close to the actual distribution of the exit time $\tau_{\mathcal{B}(h)}$, see \cite[Theorem 3.1.6]{NoiseSlowFastSys} and the comments below.

%%%%%%%%%%%%%%%%%%%%%%%%%%%%%%%%%%%%%%%%%%%%%%%%%%%%%%%%%%%%%%%%%%%%%%%%%%%%%%%%%%%%%%%%%%%%%%%%%%%%%%%%%%%%%%%%%%%%%%%%%%%%%%%%%%%%%
\subsection{Back to the Original System}
Now that we have convinced ourselves that the most promising estimate is given in Theorem \ref{thmvariantbernsteinineq} it remains to generalize the result to different scenarios which may be of interest.

%%%%%%%%%%%%%%%%%%%%%%%%%%%%%%%%%%%%%%%%%%%%%%%%%%%%%%%%%%%%%%%%%%%%%%%%%%%%%%%%%%%%%%%%%%%%%%%%%%%%%%%%%%%%%%%%%%%%%%%%%%%%%%%%%%%%%
\subsubsection{The Nonlinear Case}

Recall that we have rewritten the SDE \eqref{1DTaylor} satisfied by the deviations around the slow manifold $\xi_t = x_t - \bar{x}(t, \eps)$
\begin{align*}
\diff \xi_t 
&= \frac{1}{\eps} [a(t,\eps)\xi_t + b(\xi_t,t,\eps)] \diff t + \frac{\sigma}{\eps^H} F(t) \diff W^H_t,
\end{align*}
with $a$ being the linear drift term and $b$ containing the (possible) nonlinearities of the equation satisfying $\left| b(x,t,\eps) \right| \leq M \left| x \right|^2$.
As we expect that the nonlinear term $b$ does not influence the deviations too strongly near the critical manifold, we use the same neighborhood $\mathcal{B}(h)$. This in particular implies that we keep the same simple description of it, which we have derived in Proposition \ref{fastSlowVariance}. The bound on $b$ will help us to control it inside of $\mathcal{B}(h)$. For the case that $(x_t)_{t \in I}$ is starting on the slow manifold, i.e. $\xi_0 = 0$ we obtain:

\begin{theorem}
\label{thmNonlinearCase}
Let $t \in I$. For $h$ sufficiently small it holds
\begin{align*}
\p \left( \sup_{0 \leq r < t} \frac{\left| \xi_r \right|}{\sqrt{\zeta(r)}} \geq h \right)
\leq 2e \left\lceil \frac{\left| \alpha(t) \right|}{\eps} \kappa^2 \frac{h^2}{\sigma^2} (1 + \Onot (\eps))\right\rceil\exp\left( -\kappa^2\frac{h^2}{2\sigma^2} \right),
\end{align*}
where $\kappa = 1 - \Onot(h)$.
\end{theorem}
\begin{proof}
As previously motivated before we treat the nonlinear drift term $b$ as perturbation of the linear system, i.e. split the solution of \eqref{1DTaylor}
\begin{align*}
\xi_t = \xi^0_t + \xi^1_t,
\end{align*}
where $\xi^0_t$ is a solution to the linear system, which we have already studied in detail, and
\begin{align*}
\xi^1_t = \frac{1}{\eps} \int_0^t e^{\alpha(t,u)/\eps} b(\xi_u,u,\eps) \diff u.
\end{align*}
Then for $h = h_0 + h_1$, $\:h_0, h_1 \geq 0$ we consider
\begin{align*}
&\p\left( \sup_{0 \leq r < t} \frac{\left| \xi_r \right|}{\sqrt{\zeta(r)}} \geq h \right) 
= \p\left( \sup_{0 \leq r < t \wedge \tau_{\mathcal{B}(h)}} \frac{\left| \xi_r \right|}{\sqrt{\zeta(r)}} \geq h \right) \\
\leq \:&\underbrace{\p\left( \sup_{0 \leq r < t} \frac{\left| \xi^0_r \right|}{\sqrt{\zeta(r)}} \geq h_0 \right)}_{\substack{=:P_0(h_0)}}
+ \underbrace{\p\left( \sup_{0 \leq r < t \wedge \tau_{\mathcal{B}(h)}} \frac{\left| \xi^1_r \right|}{\sqrt{\zeta(r)}} \geq h_1 \right)}_{\substack{=:P_1(h_1)}}
\end{align*}
It remains to prove that $P_1(h_1)$ is small. Observe that due to continuity of $(\xi_t)_{t \in I}$ we have for $r < \tau_{\mathcal{B}(h)}$ 
\begin{align*}
\left| \xi_r \right| \leq h\sqrt{\zeta(r)} \leq h\sqrt{\zeta^+}.
\end{align*}
This enables us to control $\frac{\left| \xi^1_r \right|}{\sqrt{\zeta(r)}}$ inside of  $\mathcal{B}(h)$ thanks to the bound on the Taylor remainder term $b$
\begin{align*}
\frac{\left| \xi^1_r \right|}{\sqrt{\zeta(r)}} 
&\leq \frac{1}{\eps \sqrt{\zeta(r)}} \int_0^r e^{\alpha(r,u)/\eps} \left| b(\xi_u, u,\eps)\right| \diff u \\
&\leq \frac{M}{\eps \sqrt{\zeta(r)}} \int_0^r e^{\alpha(r,u)/\eps} \underbrace{\left| \xi_u\right|^2}_{\substack{\leq h^2}\zeta^+} \diff u \leq \frac{Mh^2}{a}\frac{\zeta^+}{\sqrt{\zeta^-}}.
\end{align*}
Hence, choosing $h_1 = 2 \frac{Mh^2}{a}\frac{\zeta^+}{\sqrt{\zeta^-}}$ results in $P_1(h_1) = 0$. Note that this is choice is possible as long as $h \leq \frac{a}{2M} \frac{\sqrt{\zeta^-}}{\zeta^+}$, which is in $\Onot(1)$, so requiring $h \gg \sigma$ is possible. Indeed, the choice of $h$ is usually even ``smaller'' than $\Onot(1)$, so that $h_0 = h - h_1 = h(1-\Onot(h))$. Applying Theorem~\ref{thmvariantbernsteinineq} to $h_0$ now yields the claim.
\end{proof}

\begin{rem}
With this approach we lose some accuracy ($\kappa = 1 - \Onot(h)$ instead of $\kappa = 1$) in the exponential. This has more effect on the increase of probability than in the prefactor. To overcome this difficulty it might be better to adapt the neighborhood depending on the nonlinearities. 
\end{rem}

%%%%%%%%%%%%%%%%%%%%%%%%%%%%%%%%%%%%%%%%%%%%%%%%%%%%%%%%%%%%%%%%%%%%%%%%%%%%%%%%%%%%%%%%%%%%%%%%%%%%%%%%%%%%%%%%%%%%%%%%%%%%%%%%%%%%%
\subsubsection{Behavior Close to the Slow Manifold}

For the deterministic system we get in the case of a uniformly asymptotically stable slow manifold that solutions starting close to it are attracted exponentially fast. Given low enough noise intensity a similar behavior can be observed in the noisy system, i.e., solutions have small deviations around the deterministic solution and after some (small) time $t_0$ we can again observe small deviations around the slow manifold.

\begin{theorem}
Let $t > 0$, $d>0$. There is $\delta > 0$ and some time $t_0 > 0$ such that the solutions $(x_t)_{t \in I}$ of~\eqref{1DSysinSlowTime} with initial condition $x_0$ satisfying
\begin{align*}
\left| x_0 - \bar{x}(0,\eps) \right| < \delta
\end{align*}
are attracted by the slow manifold. That is, up to time $t_0$ the solution $(x_t)_{t \in I}$ is close to the deterministic solution $x^{\det}$ 
\begin{align*}
\p \left( \sup_{0 \leq r < t_0} \frac{\left| x_r - x_r^{\det} \right|}{\sqrt{\tilde{\zeta}(r)}} \geq h \right)
\leq 2e \left\lceil \frac{\left| \tilde{\alpha}(t_0) \right|}{\eps} \kappa^2 \frac{h^2}{\sigma^2} (1 + \Onot (\eps))\right\rceil\exp\left( -\kappa^2\frac{h^2}{2\sigma^2} \right),
\end{align*}
where the $\tilde{}$ denotes the different values due to linearization around $x^{\det}$ instead of the slow manifold $\bar{x}(t,\eps)$.
After $t_0$ we obtain almost the same behavior as in the case where $x_0 = x(0,\eps)$, i.e. for $t \geq t_0$
\begin{align*}
\p \left( \sup_{t_0 \leq r < t} \frac{\left| x_r - \bar{x}(r,\eps) \right|}{\sqrt{\zeta(r)}} \geq h \right)
\leq 2e \left\lceil \frac{\left| \alpha(t) \right|}{\eps} \kappa^2 \frac{(h-d)^2}{\sigma^2} (1 + \Onot (\eps))\right\rceil\exp\left( -\kappa^2\frac{(h-d)^2}{2\sigma^2} \right).
\end{align*}
\end{theorem}
\begin{proof}
Exponentially fast attraction means that there are constants $\delta,C,\kappa > 0$ such that for $x_0 \in \R$ with $\left| x_0 - \bar{x}(0,\eps) \right| < \delta$ it holds
\begin{align*}
\left| x_t - x(t,\eps) \right| \leq C \left| x_0 - \bar{x}(0,\eps) \right| e^{-\kappa t /\eps}.
\end{align*}
Consider an initial value $x_0 \in \R$ with $\left| x_0 - x(0,\eps) \right| < \delta$ and denote by $x^{\det}$ the solution to the deterministic system ($\sigma = 0$) in \eqref{1DSysinSlowTime}. Instead of linearizing $\eqref{1DSysinSlowTime}$ around $\bar{x}(t,\eps)$, like we did in \eqref{1DTaylor}, we linearize it around $x^{\det}.$ This procedure yields qualitatively the same linearization, with $\tilde{a}(t) := \partial_x f(x_t^{\det}, t, 0)$ instead of $a(t,\eps),$ or respectively $a(t)$, see discussion before, and $\tilde{\zeta}$ adapted accordingly. In particular, even for the nonlinear case we obtain by Theorem \ref{thmNonlinearCase}
\begin{align*}
\p \left( \sup_{0 \leq r < t} \frac{\left| x_r - x_r^{\det} \right|}{\sqrt{\tilde{\zeta}(r)}} \geq h \right)
\leq 2e \left\lceil \frac{\left| \tilde{\alpha}(t) \right|}{\eps} \kappa^2 \frac{h^2}{\sigma^2} (1 + \Onot (\eps))\right\rceil\exp\left( -\kappa^2\frac{h^2}{2\sigma^2} \right),
\end{align*}
where $\tilde{\alpha}(t) = \int_0^t \tilde{a}(r) \diff r$ and $\kappa = 1 - \Onot(h).$
Choose $t_0$ such that for distance $d$
\begin{align*}
\left| x_r^{\det} - \bar{x}(r,\eps) \right| < d ~ \text{ for all }r \geq t_0,
\end{align*}
that is $t_0 \geq \frac{\eps}{\kappa} \ln\left( \frac{C\delta}{d}\right)$.
Furthermore, we have by the mean value theorem
\begin{align*}
\tilde{a}(t,\eps) = \partial_x f(x_t^{\det}, t,\eps) = \partial_x f(\bar{x}(t,\eps),t,\eps) + \partial_{xx}f(\tilde{x},t,\eps)(x_t^{\det} - \bar{x}(t,\eps))
\end{align*}
for some $\tilde{x} = \lambda x_t + (1- \lambda) \bar{x}(t,\eps), \: \lambda \in [0,1].$ So that for $a(t,\eps) = \partial_x f(\bar{x}(t,\eps),t,\eps)$
\begin{align*}
\left| \tilde{a}(t,\eps) - a(t,\eps) \right| \leq CM\delta e^{-\kappa t/\eps}.
\end{align*}
We want this distance to be of order at most $\eps$, so that in total $t_0 \geq \max \left\{ \frac{\eps}{\kappa} \ln\left( \frac{C\delta}{d}\right), \frac{\eps}{\kappa} \ln\left( \frac{CM\delta}{\eps}\right) \right\}.$
Then, up to time $t_0$, we can use the estimate above for $(x_t)_{t \in I}$ close to its deterministic solution. And after $t_0$ the process is already close to the slow manifold and its dynamics, so it makes sense to look at the deviations around $\bar{x}(r,\eps)$.
Splitting again $h = h_0 + h_1, \: h_0, h_1 \geq 0$
\begin{align*}
\p \left( \sup_{t_0 \leq r < t} \frac{\left| x_r - \bar{x}(r,\eps) \right|}{\sqrt{\zeta(r)}} \geq h \right) 
&\leq \p \left( \sup_{t_0 \leq r < t} \frac{\left| x_r - x_r^{\det} \right|}{\sqrt{\zeta(r)}} \geq h_0 \right) 
+ \p \left( \sup_{t_0 \leq r < t} \frac{\left| x_r^{\det} - \bar{x}(r,\eps) \right|}{\sqrt{\zeta(r)}} \geq h_1 \right). 
\end{align*}
Choosing $h_1 = d$ we obtain
$\p \left( \sup_{t_0 \leq r < t} \frac{\left| \bar{x}(r,\eps) - x_r^{\det} \right|}{\sqrt{\zeta(r)}} \geq h_1 \right) = 0.$ Furthermore, since $h_0 = h - d$, we have
\begin{align*}
\p \left( \sup_{t_0 \leq r < t} \frac{\left| x_r - \bar{x}(r,\eps) \right|}{\sqrt{\zeta(r)}} \geq h \right) \leq 2e \left\lceil \frac{\left| \alpha(t) \right|}{\eps} \kappa^2 \frac{(h-d)^2}{\sigma^2} (1 + \Onot (\eps))\right\rceil\exp\left( -\kappa^2\frac{(h-d)^2}{2\sigma^2} \right),
\end{align*}
which finishes the proof.
\end{proof}

%%%%%%%%%%%%%%%%%%%%%%%%%%%%%%%%%%%%%%%%%%%%%%%%%%%%%%%%%%%%%%%%%%%%%%%%%%%%%%%%%%%%%%%%%%%%%%%%%%%%%%%%%%%%%%%%%%%%%%%%%%%%%%%%%%%%%
\subsubsection{More Complicated Slow Dynamics}

So far, we have considered the case where the slow dynamics is completely uniform and regular, i.e.,
\begin{align*}
\diff y_t &= 1 \diff t.
\end{align*}
However, in applications many interesting systems contain more complicated slow variables. In fact, this is particularly relevant if one wants to reduce the dynamics to the slow manifold. The reduced equation usually qualitatively describes the dynamics of the slow variables around the slow manifold quite well. This section will clarify that the theory developed so far can be extended to more complicated slow dynamics in two steps. We first generalize our result to deterministic slow dynamics, which may also influence the diffusion term and then consider a fully coupled system.

%%%%%%%%%%%%%%%%%%%%%%%%%%%%%%%%%%%%%%%%%%%%%%%%%%%%%%%%%%%%%%%%%%%%%%%%%%%%%%%%%%%%%%%%%%%%%%%%%%%%%%%%%%%%%%%%%%%%%%%%%%%%%%%%%%%%%
\subsubsection*{Deterministic Slow Dynamics}

Hence, we consider systems of the form
\begin{equation}
\label{1DSysinSlowTimeComplDyn}
\begin{aligned}
\diff x_t &= \frac{1}{\eps} f(x_t, y_t,\eps) \diff t + \frac{\sigma}{\eps^H} F(y_t)  \diff W^H_t \\
\diff y_t &= g(y_t,\eps) \diff t.
\end{aligned}
\end{equation}
We have to adapt the assumptions a bit.
\begin{ass}
Stable Case, Non-trivial Slow Dynamics
\begin{enumerate}
\item \emph{Regularity:} The functions $f \in \C^2(\mathcal{D} \times [0, \infty);\R)$, $g \in \C^2(\pi_2(\mathcal{D})\times [0, \infty);\R)$ and $F \in \C^1(\pi_2(\mathcal{D});(0,\infty))$, as well as their derivatives up to order $2$ are uniformly bounded on an open subset $\mathcal{D} \subset \R^2$ by a constant $M \geq 0$. Here $\pi_2$ is the projection onto the second coordinate.
% in $\R$
\item \emph{Critical manifold:} There is an $x^*: \mathcal{D}_0 \rightarrow \R$ for $\mathcal{D}_0 \subset \pi_2(\mathcal{D})$ open such that
\begin{align*}
C_0 = \{ (x,y) \in \mathcal{D}: ~ y \in \mathcal{D}_0, \: x = x^*(y)\}
\end{align*}
is a critical manifold of the system \eqref{1DSysinSlowTimeComplDyn}.
\item \emph{Stability:} For $a(y) := \partial_x f(x^*(y),y, 0)$ there is $a > 0$ such that
\begin{align*}
a(y) \leq -a
\end{align*}
for all $y \in \mathcal{D}_0$.
\item \emph{Global existence:} The solutions $(x,y)$ of \eqref{1DSysinSlowTimeComplDyn} are defined for all $t \in [0, \infty)$.
\end{enumerate}
\end{ass}
Under these assumptions the system \eqref{1DSysinSlowTimeComplDyn} has an attracting slow manifold 
\begin{align*}
C_{\eps} = \left\{ (x,y) \in \mathcal{D}: ~ y \in \mathcal{D}_0, \: x = \bar{x}(y,\eps) \right\},
\end{align*}
where $\bar{x}(y, \eps) = x^*(y) + \Onot(\eps)$ due to Theorem \ref{TihonovThm} (Fenichel-Tikhonov). (Again, this $\Onot(\eps)$ is uniform in $y$.) 
We linearize the fast variable around $C_\eps$.
For a solution $(x,y)$ of $\eqref{1DSysinSlowTimeComplDyn}$ set $\xi_t = x_t - \bar{x}(y_t, \eps)$, then $(\xi,y)$ satisfies the equation
\begin{equation}
\begin{aligned}
\diff \xi_t &= \frac{1}{\eps} \left[ f(\xi_t + \bar{x}(y_t, \eps), y_t, \eps) - f(\bar{x}(y_t, \eps), y_t,\eps) \right] \diff t + \frac{\sigma}{\eps^H} F(y_t)  \diff W^H_t \\
&= \frac{1}{\eps} [a(y_t,\eps)\xi_t + b(\xi_t,y_t,\eps)] \diff t + \frac{\sigma}{\eps^H} F(y_t) \diff W^H_t, \\
\diff y_t &= g(y_t,\eps) \diff t,
\end{aligned}
\end{equation}
where due to Taylor's remainder theorem
\begin{align*}
a(y,\eps) &= \partial_x f(\bar{x}(y, \eps), y,\eps) = \partial_x f(x^*(y), y, 0) + \Onot(\eps), \\
\left| b(x,y,\eps) \right| &\leq M \left| x \right|^2.
\end{align*}
Now we can proceed as in the case for trivial slow dynamics by first considering solutions starting on the slow manifold, i.e. $(\xi_0, y_0) = (0, y_0), \:y_0 \in \mathcal{D}_0$, and using the terms
\begin{align*}
\tilde{a}(t) &:= a(y_t, 0), \\
\tilde{F}(t) &:= F(y_t).
\end{align*}
This way we obtain the same qualitative bound (also for the nonlinear case) as before, which also coincides with intuition, as more involved dynamics on the slow manifold should not influence the attracting behavior of it.
\subsubsection*{Fully Coupled Dynamics}
Now that we have seen the idea how to generalize to more complicated dynamics we give an exposition of the more general case, where the slow variables may even be random, particularly be perturbed by fBms with different Hurst parameters $H_1, H_2 \in (\frac{1}{2},1)$. We consider the following system
\begin{equation}\label{eqFullyCoupledSys1D}
\begin{aligned}
\diff x_t &= \frac{1}{\eps} f(x_t, y_t, \eps) \diff t + \frac{\sigma_1}{\eps^{H_1}} \diff W_t^{H_1} \\
\diff y_t &= g(x_t,y_t, \eps) \diff t + \sigma_2 \diff W_t^{H_2},
\end{aligned}
\end{equation}
which will turn out to be interesting in applications, see Section~\ref{example}.
The following assumptions will suffice to obtain a qualitatively similar result to the one-dimensional case, analyzed in Section~\ref{ch_1Dcase}.
\begin{ass}\label{assStableFullyCoupled}
Stable Case, fully Coupled System
\begin{enumerate}
\item \emph{Regularity:} The functions $f \in \C^2(\mathcal{D} \times [0, \infty);\R)$ and $g \in \C^2(\mathcal{D}\times [0, \infty);\R)$ as well as their derivatives up to order $2$ are uniformly bounded on an open subset $\mathcal{D} \subset \R^2$.
% in $\R$
\item \emph{Critical manifold:} There is an $x^*: \mathcal{D}_0 \rightarrow \R$ for $\mathcal{D}_0 \subset \pi_2(\mathcal{D})$ open such that
\begin{align*}
C_0 = \{ (x,y) \in \mathcal{D}: ~ y \in \mathcal{D}_0, \: x = x^*(y)\}
\end{align*}
is a critical manifold of the system \eqref{eqFullyCoupledSys1D}. Here $\pi_2$ is the projection onto the second coordinate.
\item \emph{Stability:} For $a(y) := \partial_x f(x^*(y),y, 0)$ there is $a > 0$ such that
\begin{align*}
a(y) \leq -a
\end{align*}
for all $y \in \mathcal{D}_0$.
\item \emph{Global existence:} The solutions $(x,y)$ of \eqref{eqFullyCoupledSys1D} are defined for all $t \in [0, \infty)$.
\end{enumerate}
\end{ass}
\begin{rem}
Note that the theory discussed in Section~\ref{int:theory} does not provide the technical details regarding the existence and uniqueness of solutions for coupled systems driven by fractional Brownian motion. However, this can be extended to systems of the form~\eqref{eqFullyCoupledSys1D}. We refer to \cite[Theorem~3.3]{SelAspOfFBM} for further details.
\end{rem}
Similarly as before the system~\eqref{eqFullyCoupledSys1D} has an attracting slow manifold given by
\begin{align*}
C_{\eps} = \left\{ (x,y) \in \mathcal{D}: ~ y \in \mathcal{D}_0, \: x = \bar{x}(y,\eps) \right\},
\end{align*}
where $\bar{x}(y, \eps) = x^*(y) + \Onot(\eps)$ due to Theorem \ref{TihonovThm}, where again the $\Onot(\eps)$-term is uniform in $y$.
The strategy to establish sample paths estimates for \eqref{eqFullyCoupledSys1D} is to successively linearize both fast and slow variables around their deterministic counterpart (i.e.~the solution for $\sigma_1 = \sigma_2 = 0$) denoted by $y_t^{\det}$, $\bar{x}(y_t^{\det},\eps)$. The deviations are then described by
\begin{equation}
\begin{aligned}
\xi_t &= x_t - \bar{x}(y_t^{\det} + \eta_t,\eps), \\
\eta_t &= y_t - y_t^{\det}.
\end{aligned}
\end{equation}
They satisfy the following SDE, whose form is obtained by successively applying Taylor's theorem ($\mathrm{int.}$ always stands for the appropriate intermediate value)
\begin{align*}
\diff \xi_t &= \frac{1}{\eps} \left[ f(\xi_t + \bar{x}(y_t, \eps), y_t, \eps) - f(\bar{x}(y_t, \eps), y_t,\eps) \right] \diff t + \frac{\sigma_1}{\eps^{H_1}} \diff W_t^{H_1} \\
&= \frac{1}{\eps} \left[ \partial_x f (\bar{x}(y_t, \eps),y_t,\eps) \xi_t + \partial_{xx} f(\mathrm{int.},y_t, \eps) \xi_t^2 \right] + \frac{\sigma_1}{\eps^{H_1}} \diff W_t^{H_1} \\
&= \frac{1}{\eps} \left[ \tilde{a}(t,\eps) \xi_t + \tilde{b}(\xi_t, \eta_t,t,\eps) \right] + \frac{\sigma_1}{\eps^{H_1}} \diff W_t^{H_1},
\end{align*}
where
\begin{align*}
\tilde{a}(t,\eps) &= \partial_x f (\bar{x}(y_t^{\det}, \eps),y_t^{\det},\eps) = a(t) + \Onot(\eps), \\
b(\xi_t, \eta_t,t,\eps) &= \partial_{xx} f(\mathrm{int.},y_t, \eps) \xi_t^2 + \left[ \partial_{xx} f(\bar{x}(\mathrm{int.},\eps),\mathrm{int.},\eps) \partial_y \bar{x}(\mathrm{int.},\eps) + \partial_y f(\bar{x}(\mathrm{int.},\eps),\mathrm{int.},\eps) \right] \xi_t \eta_t,
\end{align*}
so that $\left| b(x, y,t,\eps) \right| \leq M_1 (x^2+|xy|),$ where $M_1$ is uniform in the variables due to the uniform boundedness assumption.
\begin{align*}
\diff \eta_t &= \left[ g(x_t, y_t, \eps) - g(\bar{x}(y_t, \eps), y_t,\eps) + g(\bar{x}(y_t, \eps), y_t,\eps)- g(\bar{x}(y_t^{\det}, \eps), y_t^{\det},\eps) \right] \diff t + \sigma_2 \diff W_t^{H_2} \\
&= \Big[ \partial_x g(\bar{x}(y_t, \eps), y_t,\eps) \xi_t + \partial_{xx} g(\mathrm{int.}, y_t,\eps)\xi_t^2 \\
&+ \left( \partial_x g(\bar{x}(y_t^{\det}, \eps), y_t^{\det},\eps) \partial_y \bar{x}(y_t^{\det}, \eps) + \partial_y g(\bar{x}(y_t^{\det}, \eps), y_t^{\det},\eps)\right) \eta_t + \tilde{R}(y_t,y_t^{\det},\eps) \eta_t^2 \Big] \diff t + \sigma_2 \diff W_t^{H_2}
 \\
&= \left[ c(y^{\det}, \eps) \xi_t  + d(y^{\det}, \eps) \eta_t + R(\xi_t,\eta_t,\eps)\right] \diff t + \sigma_2 \diff W_t^{H_2},
\end{align*}
where
\begin{align*}
c(y^{\det}, \eps) &= \partial_x g(\bar{x}(y_t^{\det}, \eps), y_t^{\det},\eps), \\
d(y^{\det}, \eps) &= \partial_x g(\bar{x}(y_t^{\det}, \eps), y_t^{\det},\eps) \partial_y \bar{x}(y_t^{\det}, \eps) + \partial_y g(\bar{x}(y_t^{\det}, \eps), y_t^{\det},\eps), \\
R(\xi_t,\eta_t,t,\eps) &= \partial_{xx} g(\mathrm{int.}, y_t,\eps)\xi_t^2 + \tilde{R}(y_t,y_t^{\det},\eps) \eta_t^2 + \frac{\diff}{\diff y}c(\mathrm{int.}, \eps) \xi_t \eta_t.
\end{align*}
In particular the nonlinearity term satisfies for some $M_2 \geq 0$ (again uniform in the variables)
\begin{align*}
|R(x,y,t,\eps)| \leq M_2 (x^2+y^2+|xy|).
\end{align*}
In order to prove that $x_t$ is concentrated  in the neighborhood $\mathcal{B}(h) := \{ (x,y) \in \R^2: \left| x \right| < h \sqrt{\zeta(y)}\}$ around the slow manifold $C_{\eps}$ with high probability define the exit times
\begin{align*}
\tau_{\mathcal{B}(h)} &:= \inf \{ r > 0: (\xi_r,y_r^{\det}) \notin \mathcal{B}(h)\},\\
\tau_{\eta,\tilde{h}} &:= \inf \{ r > 0: |\eta_r| > \tilde{h} \}.
\end{align*}
Then we partition the event of $x_t$ in the following way
\begin{align*}
\p \left( \tau_{\mathcal{B}(h)} < t\right) 
&= \p \left( \tau_{\mathcal{B}(h)} < t, \tau_{\eta,\tilde{h}} > \tau_{\mathcal{B}(h)}\right) + \p \left( \tau_{\mathcal{B}(h)} < t, \tau_{\eta,\tilde{h}} \leq \tau_{\mathcal{B}(h)}\right) \\
&\leq \p \left( \tau_{\mathcal{B}(h)} < t \wedge \tau_{\eta,\tilde{h}}\right) + \p \left( \tau_{\eta,\tilde{h}} \leq t \wedge \tau_{\mathcal{B}(h)} \right).
\end{align*}
Note that the first probability is of the form
\begin{align*}
&\p \left( \tau_{\mathcal{B}(h)} < t \wedge \tau_{\eta,\tilde{h}}\right) =\p \left( \sup_{0 \leq r < t \wedge \tau_{\eta,\tilde{h}}} \frac{\left| \xi_r \right|}{\sqrt{\zeta(y_r^{\det})}} \geq h\right).
\end{align*}
By the same technique used to prove Theorem \ref{thmNonlinearCase} we get
\begin{align*}
\p \left( \sup_{0 \leq r < t} \frac{\left| \xi_r \right|}{\sqrt{\zeta(y_r^{\det})}} \geq h \right)
\leq 2e \left\lceil \frac{\left| \alpha(t) \right|}{\eps} \kappa^2 \frac{h^2}{\sigma^2} (1 + \Onot (\eps))\right\rceil\exp\left( -\kappa^2\frac{h^2}{2\sigma^2} \right),
\end{align*}
which is valid as long as $h \zeta^+ + \sqrt{\zeta^-} \tilde{h} \leq \frac{a \sqrt{\zeta^-}}{2M_1}.$ It remains to estimate $\p \left( \tau_{\eta,\tilde{h}} \leq t \wedge \tau_{\mathcal{B}(h)} \right)$. This issue is however tightly linked to investigating the behavior of non-stable or even non-hyperbolic dynamics under fractional noise because we have no additional assumptions on the slow dynamics. In the event $\{\tau_{\eta,\tilde{h}} \leq t\}$ we conjecture that a reduction to the slow variables should be possible. The reduced equation is then given by
\begin{align*}
\diff y_t &= g(x(t,\eps),y_t, \eps) \diff t + \sigma_2 \diff W_t^{H_2},
\end{align*}
which will be illustrated by the simulations presented at the end of Section~\ref{example}.

%%%%%%%%%%%%%%%%%%%%%%%%%%%%%%%%%%%%%%%%%%%%%%%%%%%%%%%%%%%%%%%%%%%%%%%%%%%%%%%%%%%%%%%%%%%%%%%%%%%%%%%%%%%%%%%%%%%%%%%%%%%%%%%%%%%%%
\subsection{Example}
\label{example}

We consider the climate model analyzed in \cite[Section~6.2.1]{NoiseSlowFastSys}. It is a simple model describing the difference of temperature $\Delta T = T_1 - T_2$ and salinity $\Delta S = S_1 - S_2$ between low latitude ($T_1,$ $S_1$) and high latitude ($T_2,$ $S_2$) by a system of coupled differential equations
\begin{align*}
\frac{\diff}{\diff s} \Delta T &= - \frac{1}{\tau_r} (\Delta T - \theta) - \left( \frac{1}{\tau_d} + \frac{q}{V}(\alpha_S \Delta S - \alpha_T \Delta T)^2 \right) \Delta T, \\
\frac{\diff}{\diff s} \Delta S &= \frac{F}{H} S_0 - \left( \frac{1}{\tau_d} + \frac{q}{V}(\alpha_S \Delta S - \alpha_T \Delta T)^2 \right) \Delta S.
\end{align*}
Here $\tau_r$ stands for the relaxation time of $\Delta T$ to its reference value $\theta$, $F$ is the freshwater flux, $H$ the depth of the ocean, $S_0$ a reference salinity. Furthermore, $\tau_d$ is the diffusion timescale, $q$ the Poiseuille transport coefficient and $V$ the volume of the box the system is contained in. The influence of external sources, internal fluctuations, and/or microscopic effects can be incorporated into the model via noise terms. For example, daily weather variations certainly influence the temperature $T$ and salinity $S$. Yet, a precise/detailed modelling of these terms would be far too expensive computationally and would make the model intractable analytically. We know, as discussed in the introduction to this work, that using white noise generally does not represent temperature fluctuations correctly but these are usually positively correlated. Since we have no further basic knowledge of the stochastic process it is quite natural to start by considering fBm with Hurst index $H>1/2$. This allows us to model Gaussianity and positive correlations in time. After transforming our model into dimensionless variables $x = \Delta T / \theta$, $y = \alpha_S \Delta S / (\alpha_T \theta)$, rescaling time by $\tau_d$ and taking into consideration fractional noise with Hurst parameter $H > \frac{1}{2}$, this yields the system
\begin{equation}\label{eq_climateModel}
\begin{aligned}
\diff x_t &= \frac{1}{\eps} \left[ -(x_t-1) - \eps x_t (1+ \eta^2(x_t - y_t)^2) \right] \diff t + \frac{\sigma_1}{\eps^H} \diff W_t^H, \\
\diff y_t &= \left[ \mu - y_t (1 + \eta^2(x_t - y_t)^2)\right] \diff t + \sigma_2 \diff W_t^H,
\end{aligned}
\end{equation}
where $\eps = \tau_r/\tau_d$, $\eta^2 = \tau_d(\alpha_T \theta)^2 q / V$, and $\mu = F \cdot \alpha_S S_0 \tau_d / (\alpha_T \theta H)$.
Note that the previous system is of the form~\eqref{eqFullyCoupledSys1D}. We consider the solution on a bounded time interval $[0,T]$, $T >0$ to ensure the uniform boundedness of the corresponding functions, as imposed in Assumption~\ref{assStableFullyCoupled}.\\
The slow subsystem of the deterministic system is given by
\begin{align*}
0 &= -(x_t-1), \\
\frac{\diff}{\diff t} y_t = \dot{y}_t &= \mu - y_t (1 + \eta^2(x_t - y_t)^2).
\end{align*}
In particular, it has a normally hyperbolic critical manifold, namely
\begin{align*}
C_0 = \{ (x,y) \in \R^2: \: x = 1\},
\end{align*}
which is even stable, as $\frac{\diff}{\diff x} \left( -(x-1)\right) = -1.$ By Theorem~\ref{TihonovThm} there exists an invariant manifold
\begin{align*}
C_\eps = \{ (x,y) \in \R^2: \: x = \bar{x}(y,\eps) := 1 + \Onot(\eps)\}.
\end{align*}
In order to apply the estimate from Theorem \ref{thmvariantbernsteinineq} note that $f(x,y,\eps) = -(x_t-1) - \eps x_t (1+ \eta^2(x_t - y_t)^2)$, so that
\begin{align*}
\partial_x f(x,y,\eps) = -1 - \eps(1 + \eta^2(x-y)^2) - 2\eps\eta^2(x-y).
\end{align*}
Hence we have
\begin{align*}
a(y_t,\eps) = \partial_x f(\bar{x}(y_t,\eps), y_t, \eps) = -1 + \Onot(\eps).
\end{align*}
By Proposition \ref{fastSlowVariance} there is an attracting slow manifold for the variance of the form
\begin{align*}
\zeta(t) = H \Gamma(2H) + \Onot(\eps).
\end{align*}
We conclude that, in the case that $y_t$ is deterministic, sample paths starting on the slow manifold $\bar{x}(y,\eps)$ are concentrated in the set
\begin{align*}
\mathcal{B}(h) = \{ (x,y) \in \R^2: \: |x - 1 - \Onot(\eps)| < h (H\Gamma(2H) + \Onot(\eps)) \},
\end{align*}
or, more precisely, for $0 < t \leq T$ and initial data $(x_0, y_0) = (\bar{x}(y_0,\eps),y_0)$
\begin{align*}
\p \left( \exists \:0 \leq r \leq t: \: (x_r,y_r) \notin \mathcal{B}(h) \right) 
\leq 2e \left\lceil \frac{t}{\eps} \kappa \frac{h^2}{\sigma^2} (1 + \Onot (\eps))\right\rceil\exp\left( -\kappa\frac{h^2}{2\sigma^2} \right).
\end{align*}
Figure \ref{fig_SimClimate} indicates that for small enough noise the dynamics around the slow manifold should be governed by the equation
\begin{align}\label{model}
\frac{\diff}{\diff t} y_t = \dot{y}_t &= \mu - y_t (1 + \eta^2(1 - y_t)^2) + \sigma_2 \diff W_t^H.
\end{align}

\begin{figure}
\centering
\begin{overpic}[scale = 0.45, ,tics=10]{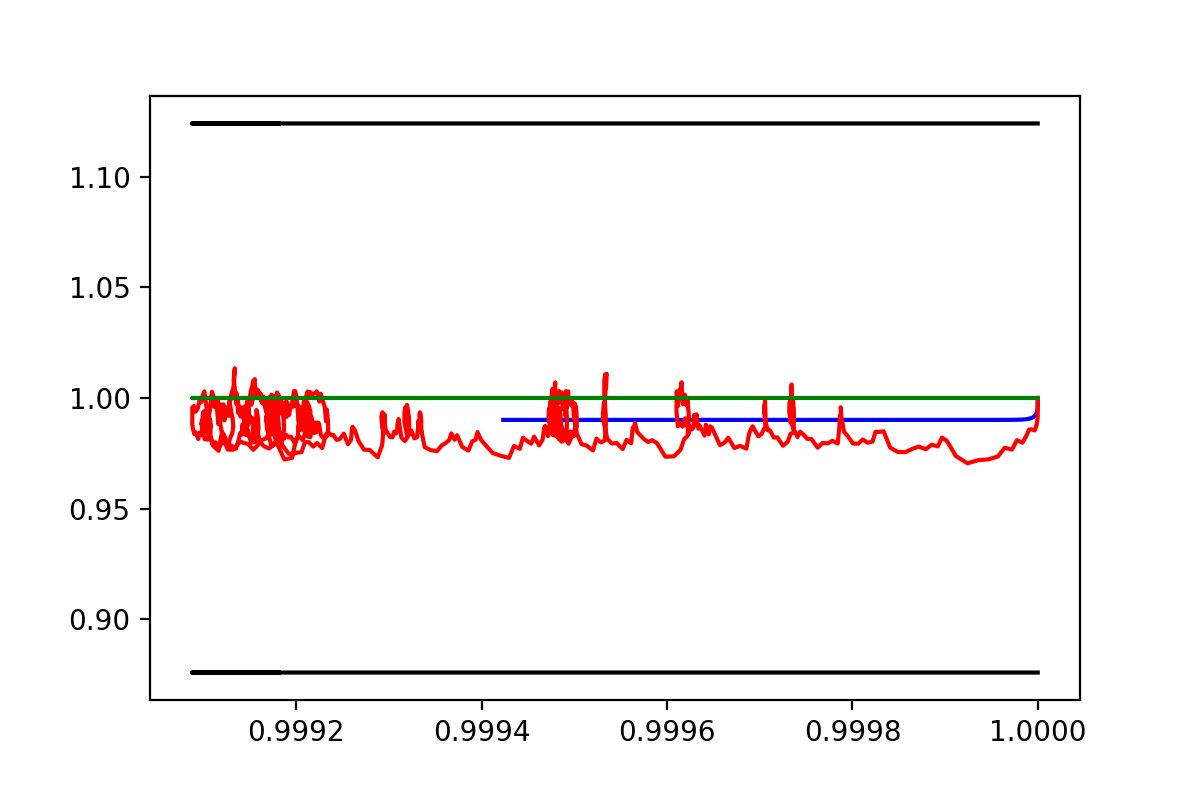}
\put(63,-1){\footnotesize slow variable}
\put(-1,33){\rotatebox{90}{\footnotesize fast variable}}
\put(12,-7){Dynamics for $\sigma_1 = 0.01$}
\end{overpic} \qquad
\begin{overpic}[scale = 0.45, ,tics=10]{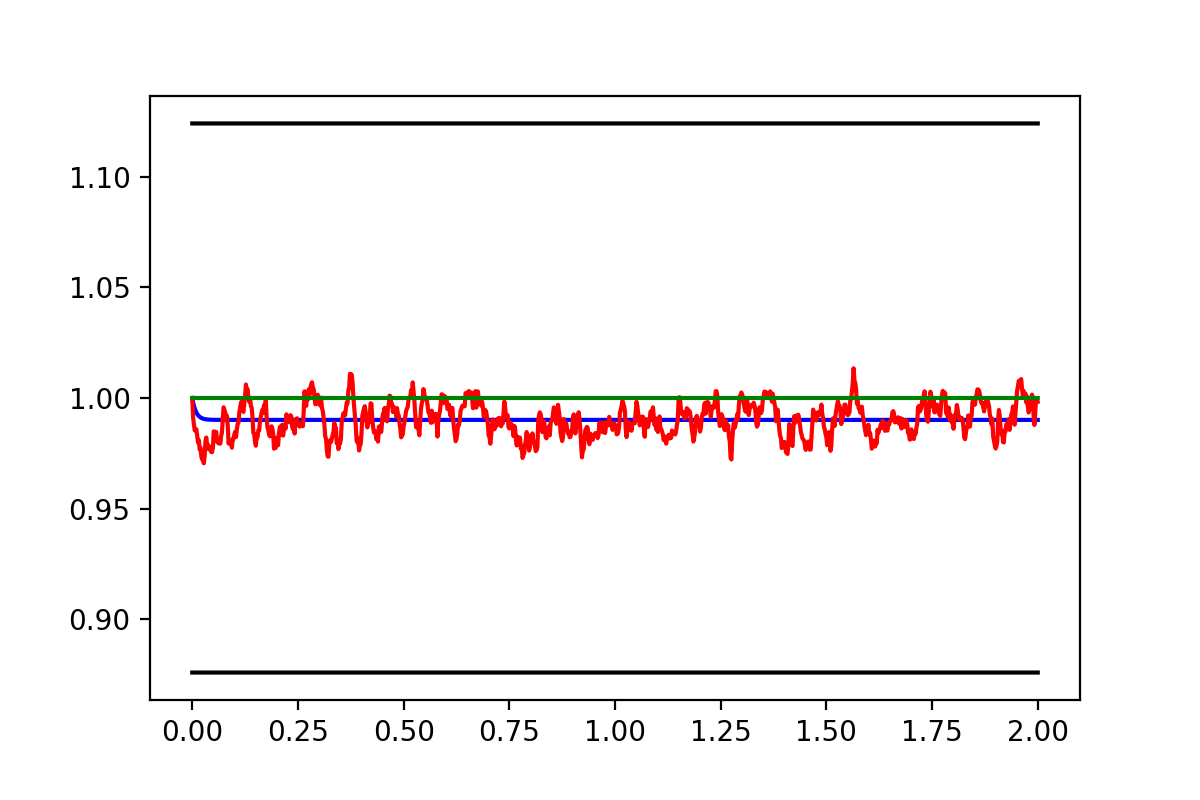}
\put(80,-1){\footnotesize time}
\put(-1,33){\rotatebox{90}{\footnotesize fast variable}}
\put(12,-7){Time series for $\sigma_1 = 0.01$}
\end{overpic} \\
\vspace{1cm}
\begin{overpic}[scale = 0.45, ,tics=10]{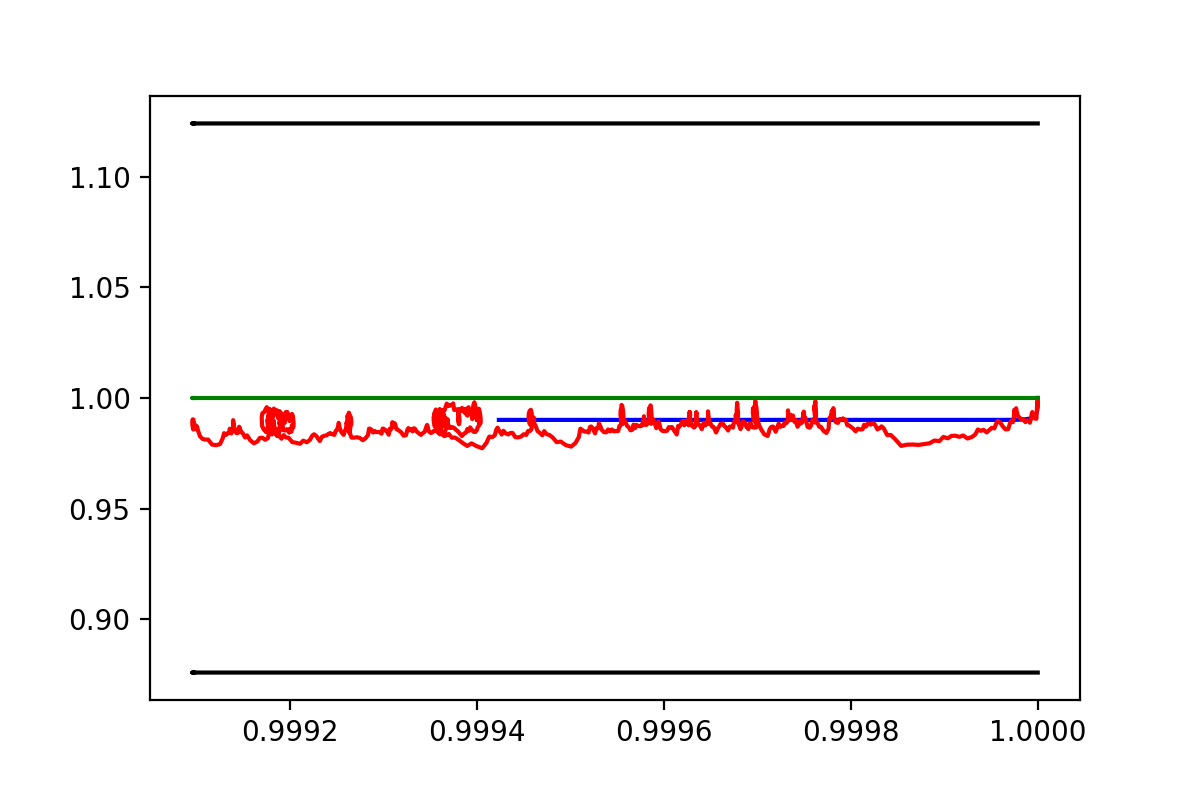}
\put(63,-1){\footnotesize slow variable}
\put(-1,33){\rotatebox{90}{\footnotesize fast variable}}
\put(12,-7){Dynamics for $\sigma_1 = 0.005$}
\end{overpic} \qquad
\begin{overpic}[scale = 0.45, ,tics=10]{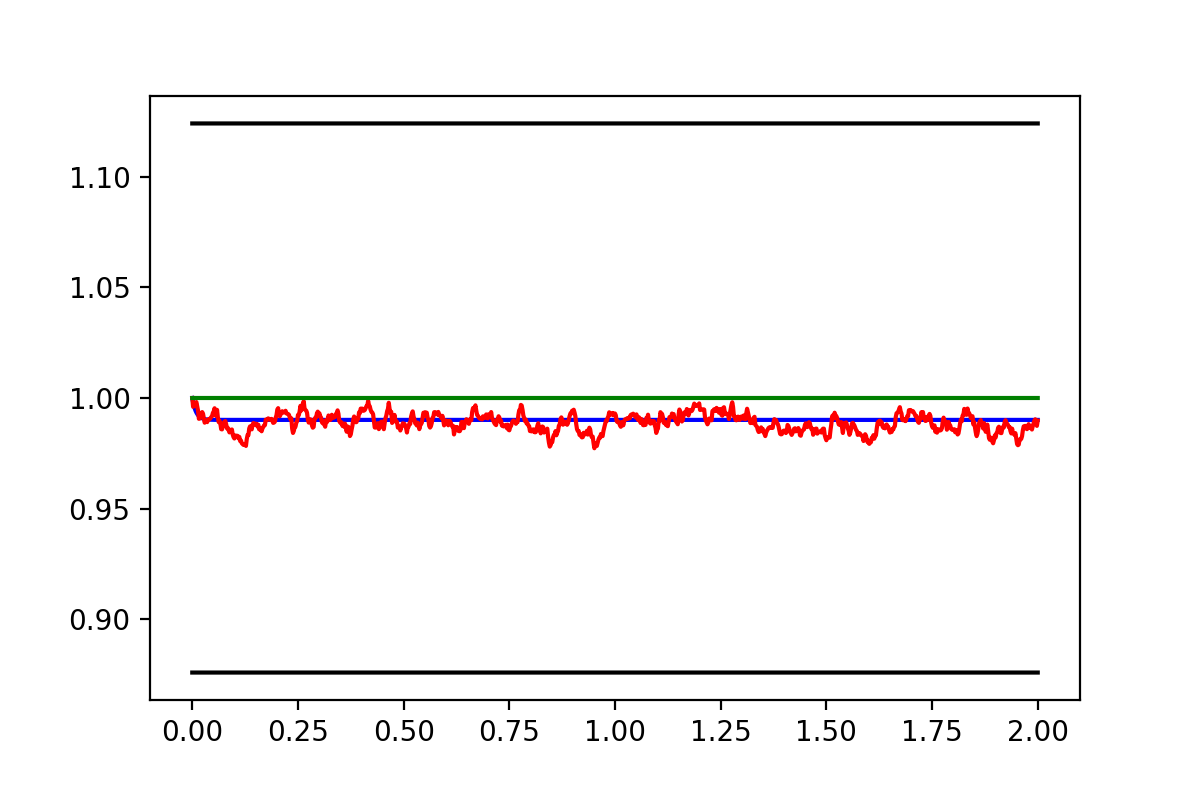}
\put(80,-1){\footnotesize time}
\put(-1,33){\rotatebox{90}{\footnotesize fast variable}}
\put(12,-7){Time series for $\sigma_1 = 0.005$}
\end{overpic} \\
\vspace{1cm}
\begin{overpic}[scale = 0.45, ,tics=10]{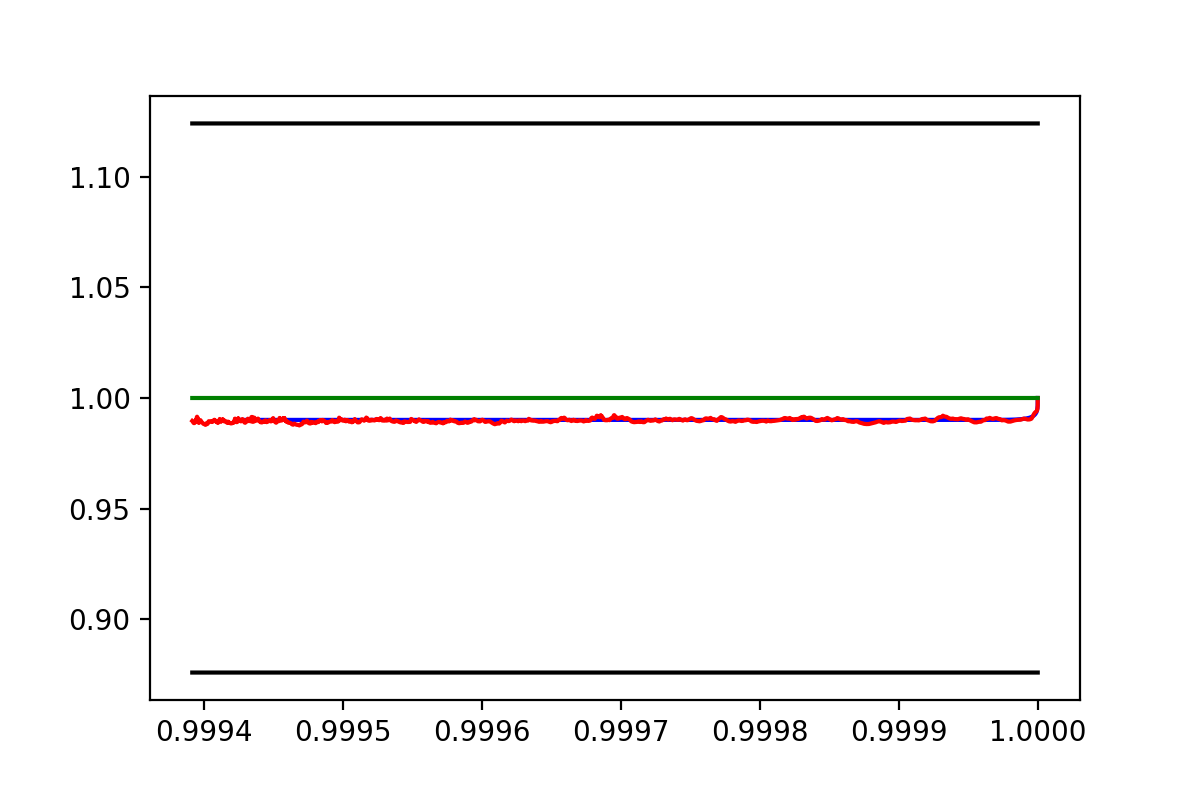}
\put(63,-1){\footnotesize slow variable}
\put(-1,33){\rotatebox{90}{\footnotesize fast variable}}
\put(12,-7){Dynamics for $\sigma_1 = 0.001$}
\end{overpic} \qquad
\begin{overpic}[scale = 0.45, ,tics=10]{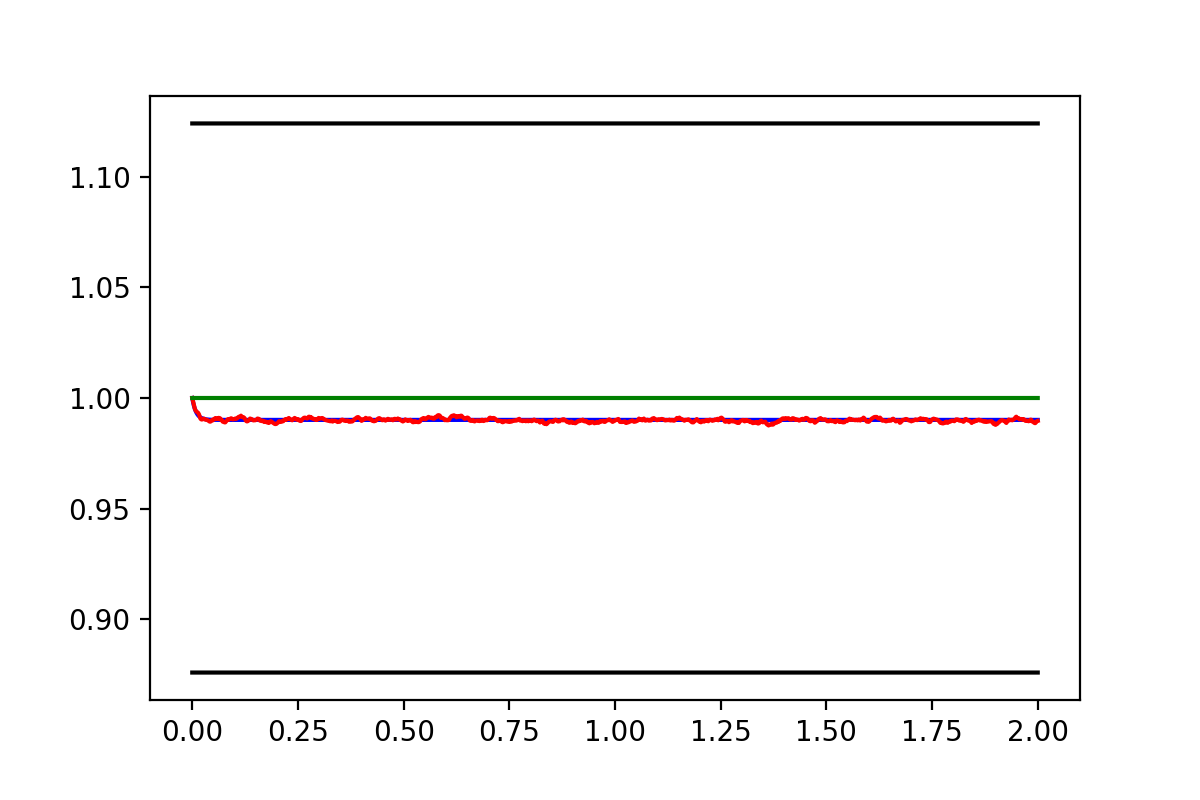}
\put(80,-1){\footnotesize time}
\put(-1,33){\rotatebox{90}{\footnotesize fast variable}}
\put(12,-7){Time series for $\sigma_1 = 0.001$}
\end{overpic}
\vspace{0.5cm}
\caption{\label{fig_SimClimate}Equation \eqref{eq_climateModel} simulated for Hurst parameter $H=0.7$, $\eps = 0.01$, $\sigma_2 = 0$ and different $\sigma_1.$ The stochastic solution is displayed red, the deterministic one is blue, the critical manifold is in green and the neighborhood $\mathcal{B}(h)$ for $h = 0.2$ is in black.}
\end{figure}
In~\eqref{model} $\eta^2$ is a fixed parameter, while $\mu$ is proportional to the freshwater flux. It can be hence treated as a slowly varying parameter compared to the rescaled salinity. By setting $X := y$ and $Y := \mu$ we obtain another fast slow system subject to some noise
\begin{equation}\label{eq_redDyn}
\begin{aligned}
\frac{\diff}{\diff t} X_t = \dot{X}_t &= Y_t - X_t (1 + \eta^2(1 - X_t)^2) + \sigma_2(t) \diff W_t^H, \\
\frac{\diff}{\diff t} Y_t = \dot{Y}_t &= \eps g(X_t, Y_t),
\end{aligned}
\end{equation}
where $g \in \C^2(\R^2;\R).$
In particular, we can apply our theory again on the two stable branches of the new slow manifold. We simulate now the reduced equation, similarly to \cite[Section 7.1]{KuehnCT2}, where the same system was considered with respect to the Brownian motion. The results for two different Hurst parameters (i.e.~$H=0.6$ and $H=0.8$) are illustrated in Figure~\ref{fig_H06} and~\ref{fig_H08}. 

\begin{figure}
\centering
\begin{overpic}[scale = 0.4, ,tics=10]{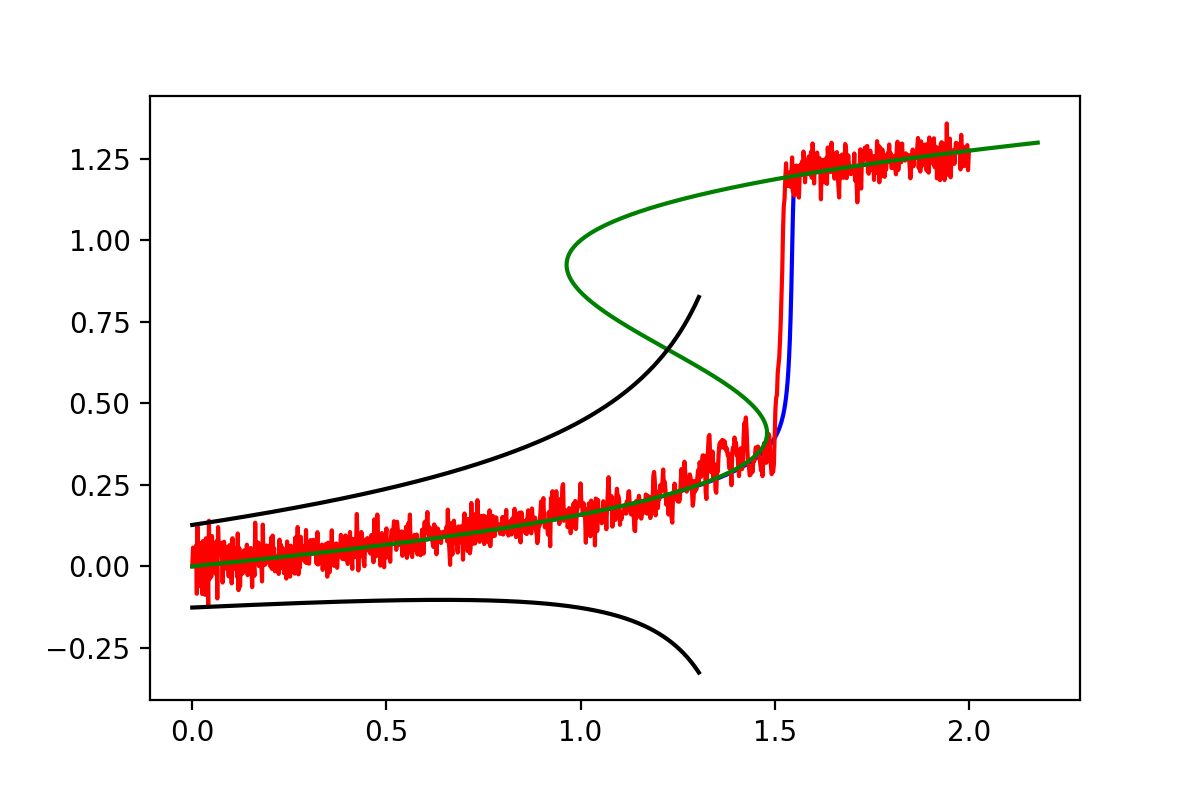}
\put(60,-1){\footnotesize slow variable}
\put(-1,30){\rotatebox{90}{\footnotesize fast variable}}
\put(12,-7){$\sigma_2 = 0.1$, $h=3$}
\end{overpic} \qquad
\begin{overpic}[scale = 0.4, ,tics=10]{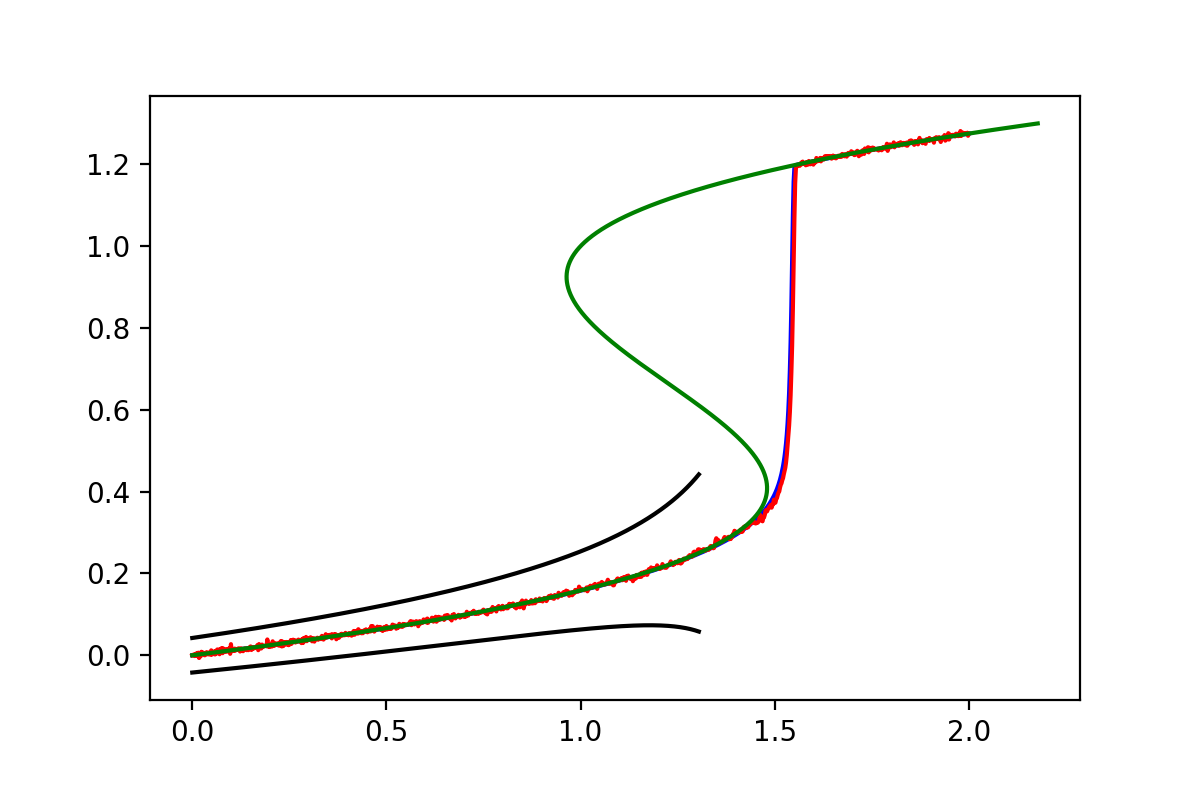}
\put(60,-1){\footnotesize slow variable}
\put(-1,30){\rotatebox{90}{\footnotesize fast variable}}
\put(12,-7){$\sigma_2 = 0.01$, $h=1$}
\end{overpic} \\
\vspace{1cm}
\begin{overpic}[scale = 0.4, ,tics=10]{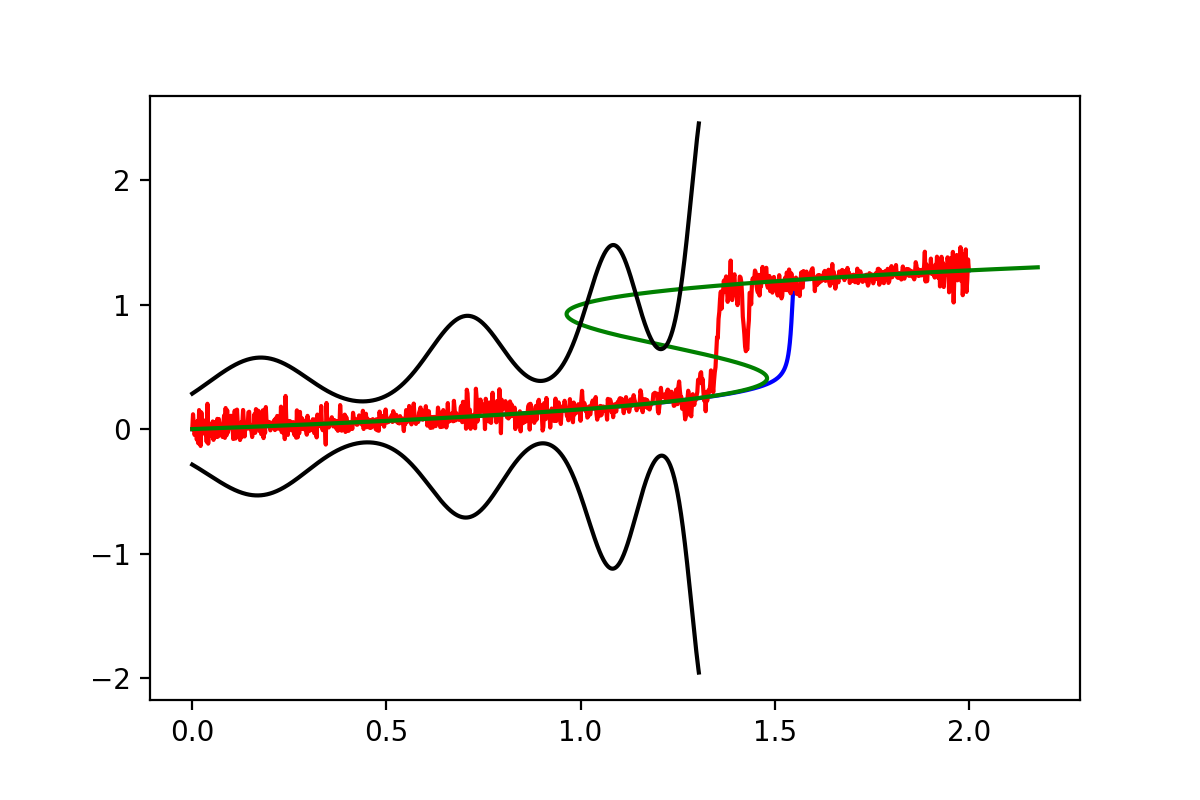}
\put(60,-1){\footnotesize slow variable}
\put(-1,30){\rotatebox{90}{\footnotesize fast variable}}
\put(10,-16){\parbox{4.5cm}{\raggedright Periodic noise with \hspace{0.5cm} $\sigma_2(t) = 0.05\sin(10t)+0.15$, $h=3$}}
\end{overpic} \qquad
\begin{overpic}[scale = 0.4, ,tics=10]{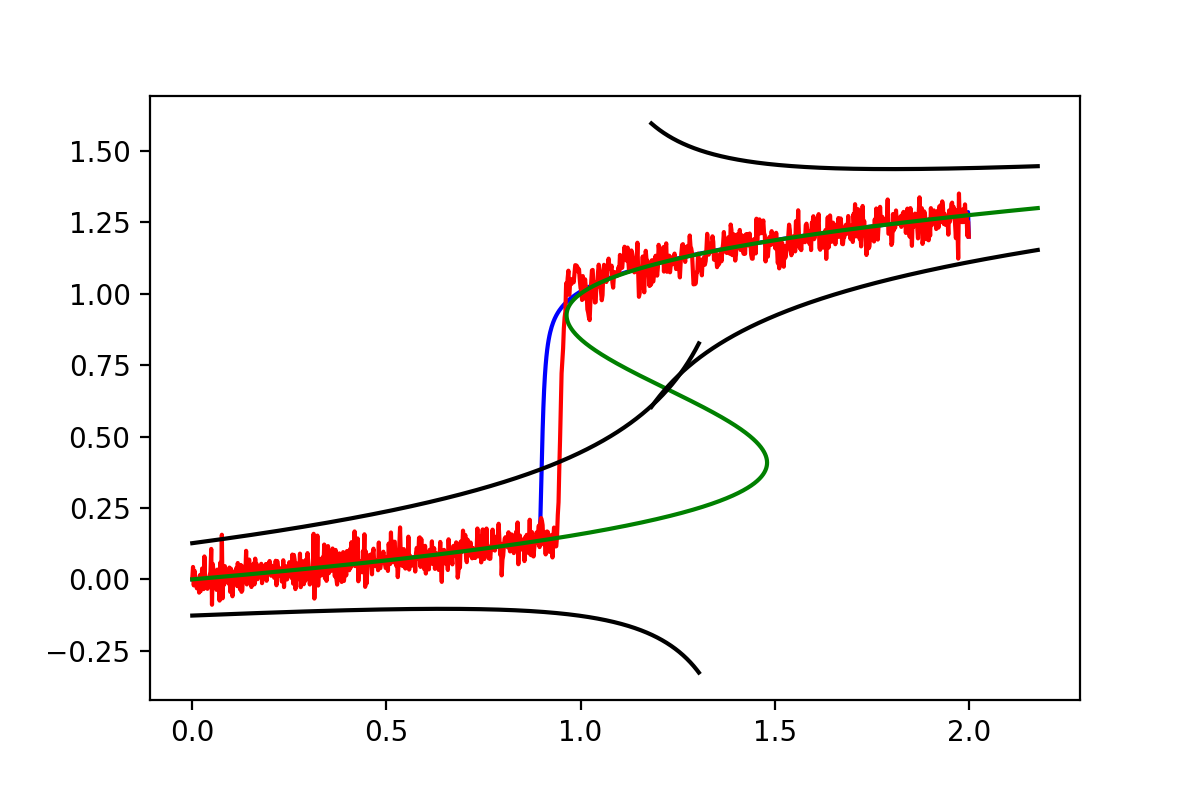}
\put(60,-1){\footnotesize slow variable}
\put(-1,30){\rotatebox{90}{\footnotesize fast variable}}
\put(10,-24){\parbox{5.5cm}{Process starting on $(1.2,2)$ and moving backward in time.\newline We also give a neighborhood for the second stable branch; $\sigma_2 = 0.1$, $h=3$}}
\end{overpic}
\vspace{2.3cm}
\caption{\label{fig_H06}Equation~\eqref{eq_redDyn} simulated for Hurst parameter $H=0.6$, $\eps = 0.01$. and different noise. The stochastic solution is displayed red, the deterministic one is blue, the critical manifold is in green and the neighborhood $\mathcal{B}(h)$ for varying $h$ is in black.}
\end{figure}
\begin{figure}
\centering
\begin{overpic}[scale = 0.4, ,tics=10]{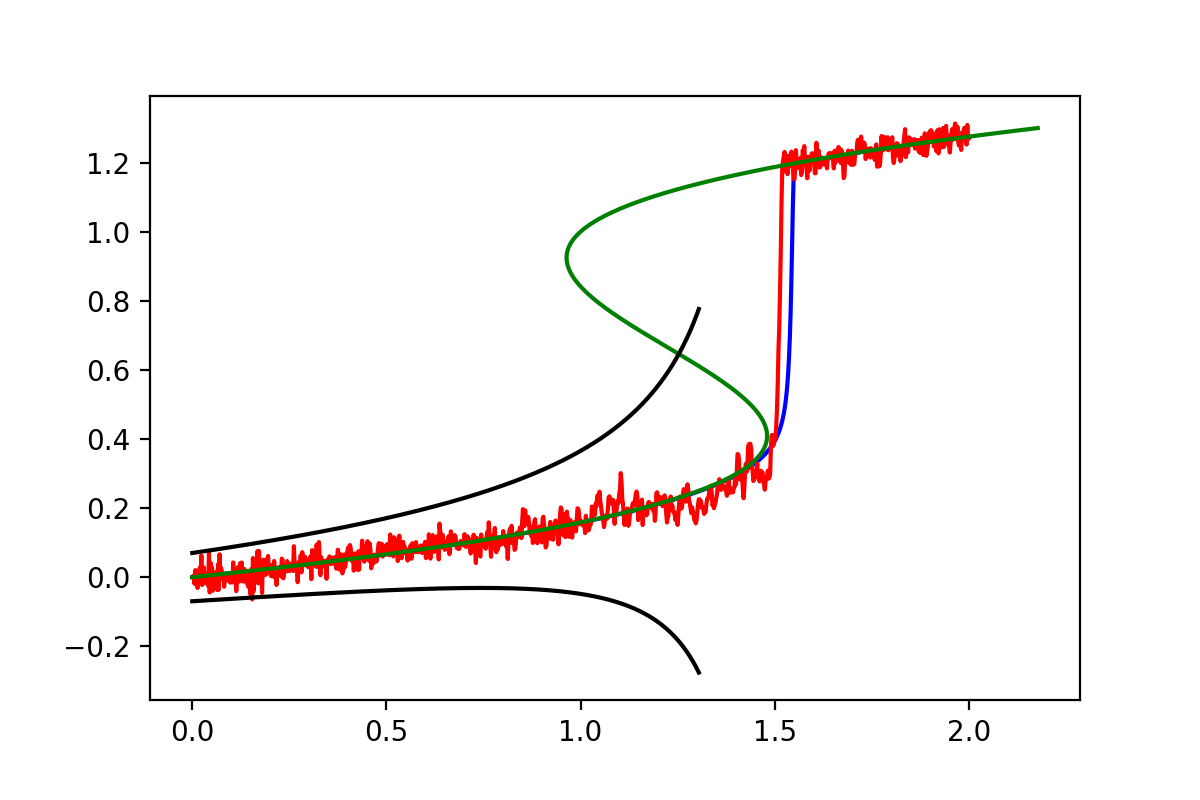}
\put(60,-1){\footnotesize slow variable}
\put(-1,30){\rotatebox{90}{\footnotesize fast variable}}
\put(12,-7){$\sigma_2 = 0.1$, $h=3$}
\end{overpic} \qquad
\begin{overpic}[scale = 0.4, ,tics=10]{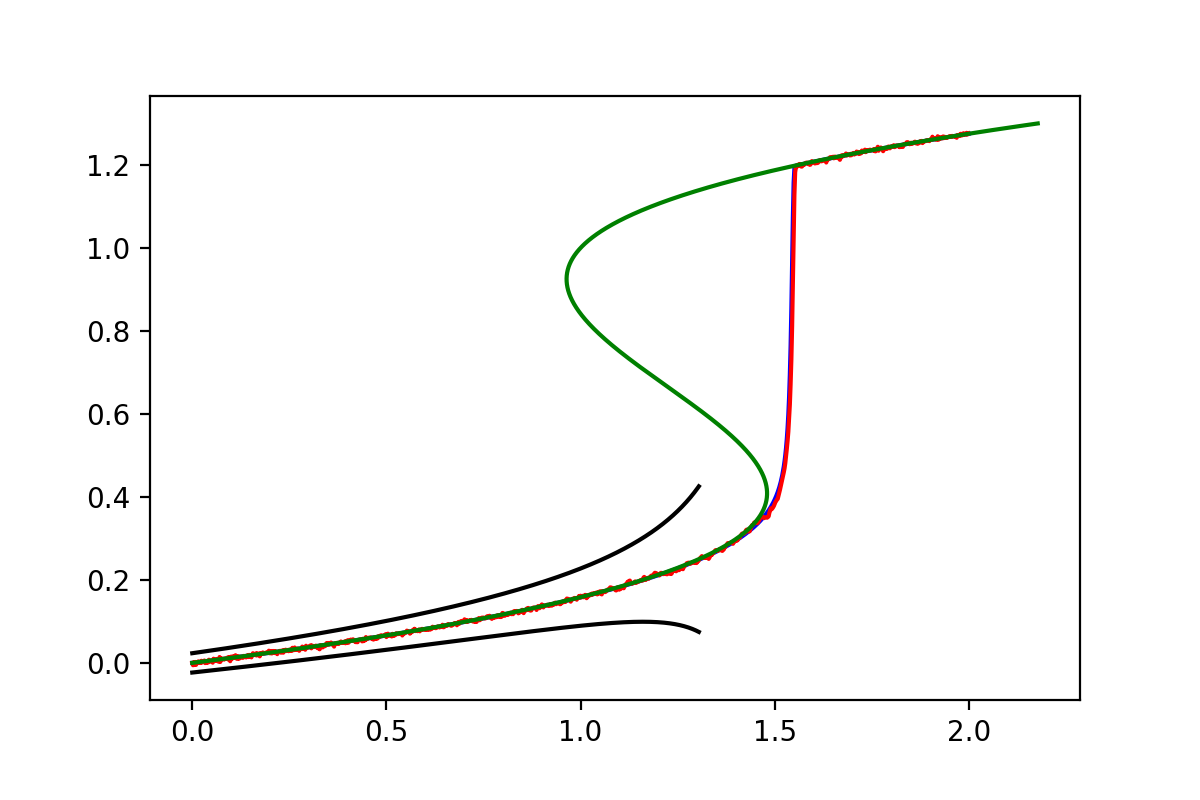}
\put(60,-1){\footnotesize slow variable}
\put(-1,30){\rotatebox{90}{\footnotesize fast variable}}
\put(12,-7){$\sigma_2 = 0.01$, $h=1$}
\end{overpic} \\
\vspace{1cm}
\begin{overpic}[scale = 0.4, ,tics=10]{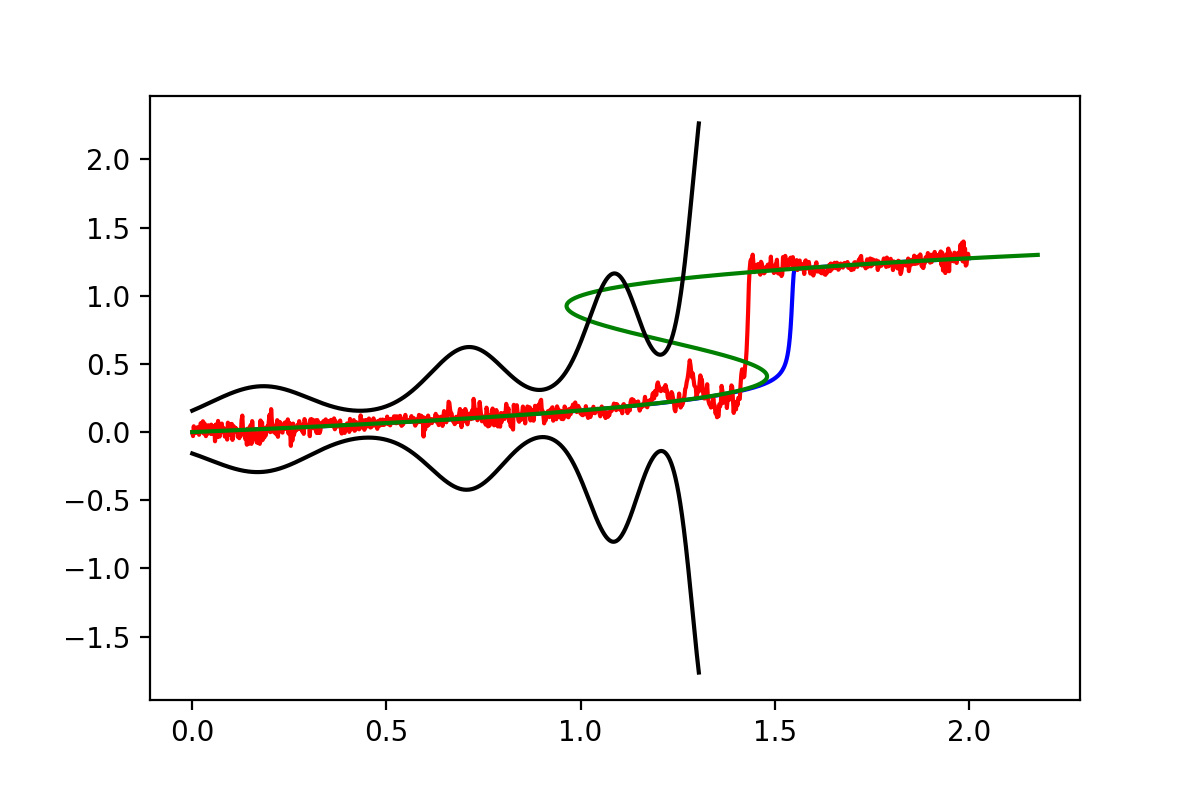}
\put(60,-1){\footnotesize slow variable}
\put(-1,30){\rotatebox{90}{\footnotesize fast variable}}
\put(10,-16){\parbox{4.5cm}{\raggedright Periodic noise with \hspace{0.5cm} $\sigma_2(t) = 0.05\sin(10t)+0.15$, $h=3$}}
\end{overpic} \qquad
\begin{overpic}[scale = 0.4, ,tics=10]{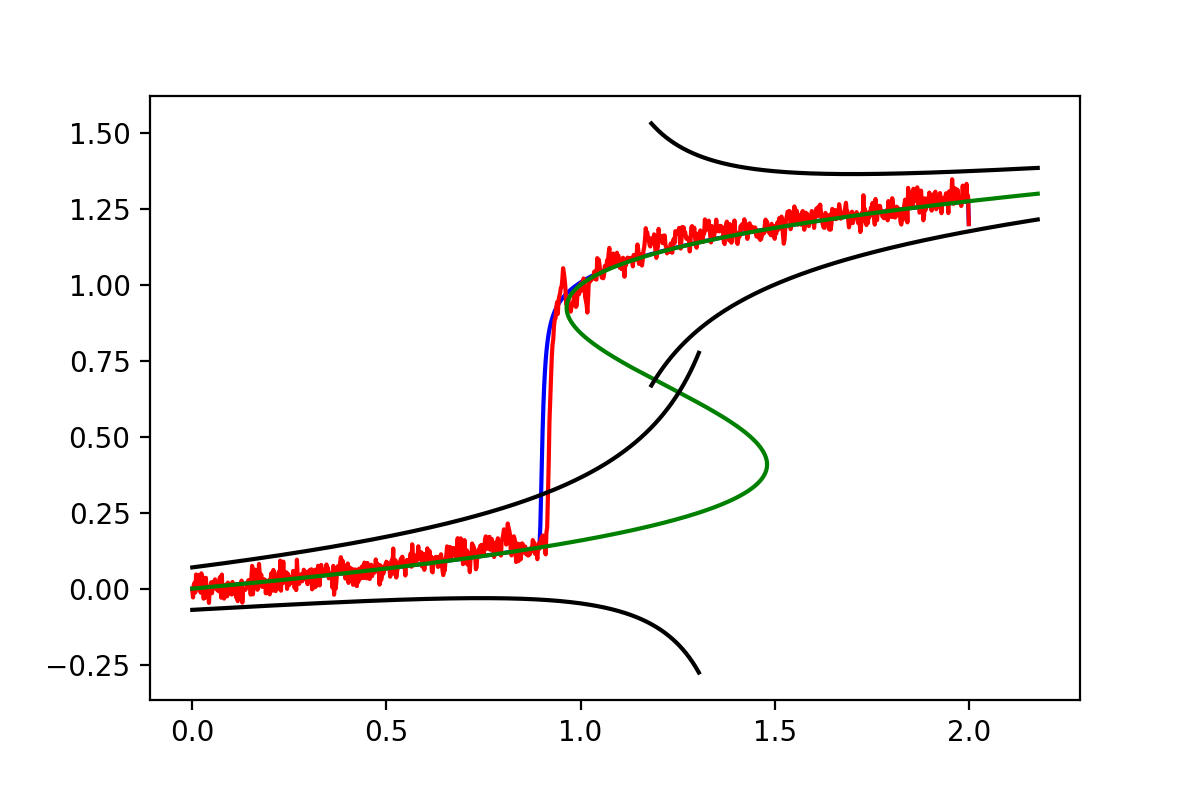}
\put(60,-1){\footnotesize slow variable}
\put(-1,30){\rotatebox{90}{\footnotesize fast variable}}
\put(10,-24){\parbox{5.5cm}{Process starting on $(1.2,2)$ and moving backward in time.\newline We also give a neighborhood for the second stable branch; $\sigma_2 = 0.1$, $h=3$}}
\end{overpic}
\vspace{2.3cm}
\caption{\label{fig_H08}Equation \eqref{eq_redDyn} simulated for Hurst parameter $H=0.8$, $\eps = 0.01$. and different noise intensity. The stochastic solution is displayed red, the deterministic one is blue, the critical manifold is in green and the neighborhood $\mathcal{B}(h)$ for varying $h$ is in black.}
\end{figure}
\newpage

%%%%%%%%%%%%%%%%%%%%%%%%%%%%%%%%%%%%%%%%%%%%%%%%%%%%%%%%%%%%%%%%%%%%%%%%%%%%%%%%%%%%%%%%%%%%%%%%%%%%%%%%%%%%%%%%%%%%%%%%%%%%%%%%
\section{The Multi-Dimensional Case}
\label{ch_MDcase}
In this section, we make the first steps towards extending our theory to the multi-dimensional case. Note that we keep the same notation as for the one-dimensional objects. We start again with uniform slow dynamics and consider the fast-slow system system in slow time
\begin{equation}
\label{eqMDSysinSlowTime}
\begin{aligned}
\diff x_t &= \frac{1}{\eps} f(x_t, y_t,\eps) \diff t + \frac{\sigma}{\eps^H} \diff W^H_t \\
\diff y_t &= 1 \diff t,
\end{aligned}
\end{equation}
under the following assumptions. Let $m\geq 1$.

\begin{ass} \label{assMDStable}
Stable Autonomous Multi-Dimensional Case
\begin{enumerate}
\item \emph{Regularity:} The function $f \in \C^2(\R^m \times [0,\infty)^2;\R)$, as well as its derivatives up to order $2$ are uniformly bounded by a constant $M \geq 0$ on an interval $I = [0, \infty)$ or $I = [0,T]$, $T > 0$.
\item \emph{Critical manifold:} There is an $x^*: [0,\infty) \rightarrow \R^m$ such that
\begin{align*}
f(x^*(t),t,0) = 0
\end{align*}
for all $t \in [0,\infty)$.
\item \emph{Stability:} The critical manifold is asymptotically stable, i.e. the Jacobian matrix 
$$A(t) := \partial_x f(x^*(t),t,0)$$
only contains eigenvalues with negative real part. In addition, its linearization is independent of time, i.e. $A(t) \equiv A.$
\item {\em Noise:} $(W^{H}_{t})_{t\geq 0}$ is an $m$-dimensional fractional Brownian motion.
\end{enumerate}
\end{ass}

These assumptions guarantee the existence and uniqueness of \eqref{eqMDSysinSlowTime} due to Remark~\ref{rem_uniquenesSDEhigherDim}(R2).
Furthermore, recall that under these assumptions there is a slow manifold 
\begin{align*}
C_{\eps} = \left\{ (x,t) \in \R^m \times I: x = \bar{x}(t,\eps) = x^*(t) + \Onot(\eps) \right\},
\end{align*}
due to Fenichel-Tikhonov (Theorem \ref{TihonovThm}). We start again by examining the behavior of the linearized system around $C_{\eps}$. For a solution $(x_t)_{t \in I}$ of $\eqref{eqMDSysinSlowTime}$ set $\xi_t := x_t - \bar{x}(t, \eps)$, then $(\xi_t)_{t \in I}$ satisfies the equation
\begin{equation}
\label{eqMDTaylor}
\begin{aligned}
\diff \xi_t &= \frac{1}{\eps} \left[ f(\xi_t + \bar{x}(t, \eps), t,\eps) - f(\bar{x}(t, \eps), t,\eps)\right] \diff t + \frac{\sigma}{\eps^H} \diff W^H_t \\
&= \frac{1}{\eps} [A(t,\eps)\xi_t + B(\xi_t,t,\eps)] \diff t + \frac{\sigma}{\eps^H}  \diff W^H_t,
\end{aligned}
\end{equation}
where
\begin{align*}
A(t,\eps) &= \partial_x f(\bar{x}(t, \eps), t, \eps) = \partial_x f(x^*(t), t, 0) + \Onot(\eps), \\
\left\| B(x,t,\eps) \right\|_{2} &\leq M \left\| x \right\|_{2},
\end{align*}
with $\left\| \cdot \right\|_{2}$ being the operator norm with respect to the Euclidean norm.
For simplicity, we analyze the linearization with $A$ being the drift term instead of $A + \Onot(\eps)$, i.e.~we consider
\begin{align}\label{eqMDlinearization}
\diff \xi_t = \frac{1}{\eps} A \xi_t \diff t + \frac{\sigma}{\eps^H} \diff W^H_t.
\end{align}
The solution for $\xi_0 = 0$ ($(x_t)_{t \in I}$ starting on the slow manifold $C_{\eps}$) is given by
\begin{align}\label{eqMDsolution}
\xi_t = \frac{\sigma}{\eps^H} \int_0^t e^{A(t-u)} \diff W^H_u.
\end{align}
Its covariance (matrix) can be computed as
\begin{align*}
\sigma^2 \Xi(t) := \cov(\xi_t) = \frac{\sigma^2}{\eps^{2H}} \int_0^t \int_0^t e^{A(t-u)} e^{A^\top(t-v)} H (2H - 1) \left| u - v \right|^{2H-2} \diff u \diff v.
\end{align*}
In the same way as in the one-dimensional case the rescaled covariance $t \mapsto \Xi(t)$ inherits the fast-slow structure.
\begin{prop}\label{propMDfastSlowVariance}
The so-called renormalized covariance $\Xi(t)$ satisfies the fast-slow ODE
\begin{align}\label{eqMDfastSlowVarianceEqn}
\eps\frac{\diff}{\diff t} \Xi(t) = \eps \dot{\Xi}(t) &= A \Xi(t) + \Xi(t) A^\top + \frac{1}{\eps^{2H-1}} \left[ P(t) + P(t)^\top \right],
\end{align}
where 
\begin{align*}
P(t) = H(2H-1) \int_0^t e^{A(t-u)/\eps} (t-u)^{2H-2} \diff u.
\end{align*}
In particular, there is an (even globally) asymptotically stable slow manifold of the system of the form
\begin{align}\label{MDSlowManifVariance}
\bar{X}(t) = \int_0^{\infty} e^{Au}\left( Q(t) + Q(t)^\top \right) e^{A^\top u} \diff u + \Onot(\eps),
\end{align}
where
\begin{align*}
Q(t) = H(2H-1) \int_0^{\infty} e^{Au} u^{2H-2} \diff u.
\end{align*}
\end{prop}
\begin{proof}
We again differentiate $t \mapsto \Xi(t)$ to obtain the ODE
\begin{align*}
\eps\frac{\diff}{\diff t} \Xi(t) = \eps \dot{\Xi}(t) &= A \Xi(t) + \Xi(t) A^\top + \frac{1}{\eps^{2H-1}} \left[ P(t) + P(t)^\top \right],
\end{align*}
where
\begin{align*}
P(t) = H(2H-1)\int_0^t e^{A(t-u)/\eps} (t-u)^{2H-2} \diff u.
\end{align*}
In order to be able to take the singular limit $\eps \rightarrow 0$ and apply Fenichel-Tikhonov (Theorem \ref{TihonovThm}) we need to prove at least one times continuous differentiability in $\eps= 0$. To do this, rewrite $\frac{1}{\eps^{2H-1}}P(t)$ by substituting $v = \frac{t-u}{\eps}$
%&\frac{H(2H-1)}{\eps^{2H-1}} \int_0^t e^{A(t-u)/\eps} (t-u)^{2H-2} \diff u \\
\begin{equation}\label{eq_MDintAfterSubst}
\begin{aligned}
&H(2H-1) \int_0^{\frac{t}{\eps}} e^{Av} v^{2H-2} \diff v \longrightarrow H(2H-1) \int_0^{\infty} e^{Av} v^{2H-2} \diff v \quad \text{ for } \eps \rightarrow 0.
\end{aligned}
\end{equation}
This implies continuity in $\eps = 0.$ To see that the right hand side of \eqref{eqMDfastSlowVarianceEqn} is continuously differentiable in $\eps = 0,$ it is sufficient to check it for the integral $P(t)$
%&= \frac{\diff}{\diff \eps}\left( H(2H-1) \int_0^t e^{A(t-u)/\eps} (t-u)^{2H-2} \diff u \right) \\
\begin{align*}
\frac{\diff}{\diff \eps}\left( P(t) \right) 
&= - H(2H-1) \int_0^t e^{A(t-u)/\eps} A(t-u)/\eps^2 (t-u)^{2H-2} \diff u,
\end{align*}
where the limit for $\eps \rightarrow 0$ exists because the exponential term dominates the polynomial term in $\eps.$
The slow subsystem hence reads
%We need to prove that $\Phi(t,t-\eps v) \rightarrow e^{A(t)v}$ for $\eps \rightarrow 0$ in a controlled way. Magnus expansion may not converge :(\\
%But if we have strong enough assumptions (autonomous system would work) the singular limit looks like this
\begin{align*}
0 &= A \Xi(t) + \Xi(t) A^\top + \left[ Q(t) + Q(t)^\top \right],
\end{align*}
where
\begin{align*}
Q(t) = H(2H-1) \int_0^{\infty} e^{Au} u^{2H-2} \diff u.
\end{align*}
This is a Lyapunov equation, and according to Lemma \ref{lemLyapunoveq} it has the unique solution
\begin{align*}
X^*(t) = \int_0^{\infty} e^{Au}\left( Q(t) + Q(t)^\top \right) e^{A^\top u} \diff u.
\end{align*}
By Fenichel-Tikhonov (Theorem \ref{TihonovThm}) we conclude that there is an asymptotically stable manifold of the form
\begin{align*}
\bar{X}(t) = \int_0^{\infty} e^{Au}\left( Q(t) + Q(t)^\top \right) e^{A^\top u} \diff u + \Onot(\eps).
\end{align*}
Note again that the stability property, which carries over from the critical manifold, is even global due to linearity of the ODE~\eqref{eqMDfastSlowVarianceEqn}. 
\end{proof}

\begin{rem}
We need that the linearization~\eqref{eqMDlinearization} is autonomous in this section for taking the singular limit in \eqref{eq_MDintAfterSubst}. In the non-autonomous case we need to compute the limit of
$\Phi(t,t-\eps v)$ for $\eps \rightarrow 0.$ We suspect that $\Phi(t,t-\eps v) \rightarrow e^{A(t)v}.$
\end{rem}

In order to investigate a multi-dimensional Lyapunov-Equation, we rely on the following result, see \cite[Lemma~5.1.2]{NoiseSlowFastSys}.
\begin{lem}[Lyapunov Equation]\label{lemLyapunoveq}
Let $A \in \R^{p \times p}$ and $B \in \R^{q \times q}$ with eigenvalues $a_1, \dots, a_p$ and $b_1, \dots, b_q.$ Then the operator $L: \R^{p\times q} \rightarrow \R^{p\times q}$ defined by
\begin{align*}
LX = AX + XB
\end{align*}
has eigenvalues of the form $\left\{ a_i + b_j \right\}_{i=1,\dots,p,j=1,\dots,q}.$ In particular, L is invertible if and only if $A$ and $-B$ don't have any common eigenvalue.
Moreover, if all eigenvalues of $A$ and $B$ have negative real part, then for any $C \in \R^{p \times q}$ the unique solution of the so-called \emph{Lyapunov equation} $AX + XB + C = 0$ is of the form
\begin{align*}
X = \int_0^{\infty}e^{Au}Ce^{Bu}\diff u.
\end{align*}
\end{lem}

Note again that due to the linearity of the operator $LX = AX + XA^\top$ the rescaled covariance $t \mapsto \frac{1}{\sigma^2}\cov(\xi_t)$ as solution of~\eqref{eqMDfastSlowVarianceEqn} with starting value $\frac{1}{\sigma^2}\cov(\xi_0) = 0$ satisfies the following equation
\begin{align}
\label{eqMDSlowManifVarApproach}
\cov(\xi_t) = \sigma^2 \left( \bar{X}(t) - e^{At/\eps}\bar{X}(0) e^{A^\top t/\eps}\right),
\end{align}
which explicitly depicts the exponentially fast approach of the covariance towards the slow manifold, as it could have been already concluded by Fenichel-Tikhonov (Theorem~\ref{TihonovThm}).
This justifies the choice of our neighborhood this time, depending on the critical manifold 
\begin{align*}
\mathcal{B}(h) = \left\{ (x,t): t \in I, \langle x, X^*(t)^{-1} x \rangle < h^2 \right\}.
\end{align*}

As already previously mentioned in Remark \ref{remCriticalManifAlsoOK} choosing the neighborhood depending on the critical manifold instead of the slow manifold does not worsen our estimates. So we expect the same to be true in the higher dimensional case. Therefore, we have used the critical manifold $X^*$ (which is time-independent in our case) this time because our strategy depends on diagonalizing, and we do not spell out the additional technical details regarding the $\Onot(\eps)$-term.

%%%%%%%%%%%%%%%%%%%%%%%%%%%%%%%%%%%%%%%%%%%%%%%%%%%%%%%%%%%%%%%%%%%%%%%%%%%%%%%%%%%%%%%%%%%%%%%%%%%%%%%%%%%%%%%%%%%%%%%%%%%%%%%%%%%%%%
\subsection{Estimates on the Deviations}
\subsubsection{No Restrictions on the Linearization}

The proof of Theorem~\ref{thmvariantgaussineqHoldercont} can be immediately extended the multi-dimensional case by proving the  mean-square H\"{o}lder continuity in each component of the covariance. 
 
\begin{lem}
\label{lemMDHolderCont}
Let $t \geq r_1 > r_2 \geq 0$, then there is a constant $G = G(t,\eps,\sigma,H)>0$
such that 
\begin{align*}
\left\| \e\left[ \left(\xi_{r_1} - \xi_{r_2} \right) \left(\xi_{r_1} - \xi_{r_2} \right)^\top\right] \right\|_2 \leq G \left| r_1 - r_2 \right|^{2H}.
\end{align*}
\end{lem}
\begin{proof}
Let $t \geq r_1 > r_2 \geq 0$, then
\begin{align}
&\left\| \e\left[ \left(\xi_{r_1} - \xi_{r_2} \right) \left(\xi_{r_1} - \xi_{r_2} \right)^\top\right] \right\|_2 \notag \\
&=\frac{\sigma^2}{\eps^{2H}} \bigg\| \int_{0}^{r_1} \int_{0}^{r_1} \left( e^{A(r_1-u)/\eps} - \mathds{1}_{\{u \leq r_2\}}e^{A(r_2-u)/\eps}\right)  \notag \\
&\qquad\cdot\left( e^{A^\top(r_1-v)/\eps} - \mathds{1}_{\{v \leq r_2\}}e^{A^\top(r_2-v)/\eps}\right) \phi(u-v) \diff u \diff v \bigg\|_2 \notag
\\
&\leq\frac{\sigma^2}{\eps^{2H}} \int_{0}^{r_2} \int_{0}^{r_2} \left\| e^{A(r_1-u)/\eps} - e^{A(r_2-u)/\eps}\right\|_2 \left\| e^{A^\top(r_1-v)/\eps} - e^{A^\top(r_2-v)/\eps}\right\|_2 \phi(u-v) \diff u \diff v\label{eqMDHoldcont1}  \\
&+ \frac{\sigma^2}{\eps^{2H}} \int_{r_2}^{r_1} \int_{0}^{r_2} \left\| e^{A(r_1-u)/\eps} - e^{A(r_2-u)/\eps}\right\|_2 \left\| e^{A^\top(r_1-v)/\eps} \right\|_2 \phi(u-v) \diff u \diff v\label{eqMDHoldcont2} \\
&+ \frac{\sigma^2}{\eps^{2H}} \int_{0}^{r_2} \int_{r_2}^{r_1} \left\|e^{A(r_1-u)/\eps}\right\|_2 \left\| e^{A^\top(r_1-v)/\eps} - e^{A^\top(r_2-v)/\eps}\right\|_2 \phi(u-v) \diff u \diff v\label{eqMDHoldcont3}
\\
&+ \frac{\sigma^2}{\eps^{2H}} \int_{r_2}^{r_1} \int_{r_2}^{r_1} \left\| e^{A(r_1-u)/\eps} \right\|_2 \left\| e^{A^\top(r_1-v)/\eps} \right\|_2 \phi(u-v) \diff u \diff v.\label{eqMDHoldcont4} 
\end{align}
Since $A$ only has eigenvalues with negative real part we have for $r \geq u \geq 0$
\begin{align*}
\left\|e^{A(r-u)/\eps}\right\|_2 \leq 1, \quad\left\|e^{A^\top(r-u)/\eps}\right\|_2 \leq 1,
\end{align*}
and for $r_1 \geq u \geq 0$, $r_2 \geq u \geq 0$
\begin{align*}
\left\| e^{A(r_1-u)/\eps} - e^{A(r_2-u)/\eps}\right\|_2 \leq \frac{a}{\eps} \left| r_1 - r_2 \right|, \quad\left\| e^{A^\top(r_1-u)/\eps} - e^{A^\top(r_2-u)/\eps}\right\|_2 \leq \frac{a}{\eps} \left| r_1 - r_2 \right|,
\end{align*}
where $a := \max\{ |\lambda|: \lambda \text{ eigenvalue of } A\}.$
This enables us to prove the result similarly as in the one-dimensional case, i.e. a straightforward calculation using the last result now shows the required H\"older bounds by estimating \eqref{eqMDHoldcont1}-\eqref{eqMDHoldcont4}.
%we get
%\begin{align*}
%&\frac{\sigma^2}{\eps^{2H}} \int_{0}^{r_2} \int_{0}^{r_2} \left\| e^{A(r_1-u)/\eps} - e^{A(r_2-u)/\eps}\right\|_2 \left\| e^{A^\top(r_1-v)/\eps} - e^{A^\top(r_2-v)/\eps}\right\|_2 \phi(u-v) \diff u \diff v \\
%\leq &\frac{\sigma^2a^2}{\eps^{2H+2}} t^2 \left| r_1 - r_2 \right|^{2H}.
%\end{align*}
%Moreover,~\eqref{eqMDHoldcont2} can be estimated as follows
%\begin{align*}
%&\frac{\sigma^2}{\eps^{2H}} \int_{r_2}^{r_1} \int_{0}^{r_2} \left\| e^{A(r_1-u)/\eps} - e^{A(r_2-u)/\eps}\right\|_2 \left\| e^{A^\top(r_1-v)/\eps} \right\|_2 \phi(u-v) \diff u \diff v \\
%\leq &\frac{\sigma^2a}{\eps^{2H+1}} t \left| r_1 - r_2 \right|^{2H}.
%\end{align*}
%Similarly, we obtain the bound for~\eqref{eqMDHoldcont3}
%\begin{align*}
%&\frac{\sigma^2}{\eps^{2H}} \int_{0}^{r_2} \int_{r_2}^{r_1} \left\|e^{A(r_1-u)/\eps}\right\|_2 \left\| e^{A^\top(r_1-v)/\eps} - e^{A^\top(r_2-v)/\eps}\right\|_2 \phi(u-v) \diff u \diff v \\
%\leq &\frac{\sigma^2a}{\eps^{2H+1}} t \left| r_1 - r_2 \right|^{2H}.
%\end{align*}
%Finally, for~\eqref{eqMDHoldcont4} we have
%\begin{align*}
%&\frac{\sigma^2}{\eps^{2H}} \int_{r_2}^{r_1} \int_{r_2}^{r_1} \left\| e^{A(r_1-u)/\eps} \right\|_2 \left\| e^{A^\top(r_1-v)/\eps} \right\|_2 \phi(u-v) \diff u \diff v \\
%\leq &\frac{\sigma^2}{\eps^{2H}} \left| r_1 - r_2 \right|^{2H}.
%\end{align*}
%By combining the four estimates we obtain the result. 
\end{proof}

The mean-square H\"{o}lder continuity of $(\xi_t)_{t \in I}$ implies the same for each component. Hence, we can establish the following qualitative result.

\begin{theorem}
Let $t \in I.$ Then under Assumption \ref{assMDStable} there is a constant $K = K(t,\eps,\sigma,H) >0$ such that for $h >0$ the following estimate holds true for $\eps >0$ small enough
\begin{align*}
\p \left( \sup_{0 \leq r \leq t} \langle \xi_r, X^*(r)^{-1} \xi_r \rangle \geq h^2 \right) 
\leq \sum_{k=1}^m K t \left( \sqrt{\lambda_k h}d_k^*\right)^{\frac{1}{H}} \exp\left( -\frac{\lambda_k h^2}{2\sigma^2}(1 - \Onot(\eps))\right),
\end{align*}
where $\lambda_k \geq 0$ with $\sum_{k=1}^m \lambda_k = 1$ and $d^*_k$ denote the (time-independent) eigenvalues of $X^*.$
\end{theorem}

\begin{proof}
Note that the critical manifold $X^*$ is symmetric and in the autonomous case it is time independent in addition. This implies that it is diagonalizable with respect to an orthogonal matrix $O$ (independent of time). Let $O^\top = \left( O_1^\top, \dots,  O_m^\top \right)$, where $O_k$ denotes the $k$-th row of $O$ and $d^*_k(r) \equiv d^*_k$ be the corresponding eigenvalues.
This enables us to reduce the problem to the estimate of the one-dimensional problem, using the notation $D^*(r) = \mathrm{diag}(d_k^*),$
\begin{align*}
\p \left( \sup_{0 \leq r \leq t} \langle \xi_r, X^*(r)^{-1} \xi_r \rangle \geq h^2 \right) &= \p \left( \sup_{0 \leq r \leq t} \langle O \xi_r, D^*(r)^{-1} O \xi_r \rangle \geq h^2 \right) \\
&= \p \left( \sup_{0 \leq r \leq t} \sum_{k=1}^m  O_k\xi_r d^*_k(r)^{-1} O_k \xi_r \geq h^2 \right)\\
%\intertext{for $\lambda_k \geq 0$ with $\sum\limits_{k=1}^m \lambda_k = 1$}
&\leq \sum_{k=1}^m \p \left( \sup_{0 \leq r \leq t} d^*_k(r)^{-1}(O_k \xi_r)^2 \geq \lambda_k h^2 \right) \\
&\leq 2 \sum_{k=1}^m \p \left( \sup_{0 \leq r \leq t} \frac{ O_k \xi_r}{\sqrt{d^*_k(r)}} \geq \lambda_k h^2 \right) \\
&= 2 \sum_{k=1}^m \p \left( \sup_{0 \leq r \leq t} O_k \xi_r \geq \lambda_k h^2 d^*_k \right).
\end{align*}
We have already proven that $(\xi_t)_{t \in I}$ is mean-square H\"{o}lder continuous in Lemma \ref{lemMDHolderCont}, which directly implies the same property for the components in the $O$-coordinate system. This means that we can apply Theorem \ref{gaussineqforHolderCont} for $O_k \xi$. This leads to
\begin{align*}
\p \left( \sup_{0 \leq r \leq t} O_k \xi_r \geq \lambda_k h^2 d^*_k \right) 
\leq K t \left( \sqrt{\lambda_k h}d_k^*\right)^{\frac{1}{H}} \exp\left( -\frac{\lambda_k h^2d^*_k}{2\sup_{0 \leq r \leq t} \var(O_k\xi_r)}\right).
\end{align*}
Now note that \eqref{eqMDSlowManifVarApproach} written in the $k$-th component in the $O$-coordinate system reads as %\textcolor{red}{[TODO: Starting from here, some of the matrix vector notation does not seem consistent, at least to me, because we have many times the situation $vA$, i.e., a \textbf{column vector} gets acted on from the right by a matrix and this is impossible using normal matrix-vector conventions so one has to say what we mean by $vA$ at least once... e.g.~it could mean $vA:=v^\top A$...]}
\begin{align*}
\var(O_k\xi_t) &= \left(O\cov(\xi_t)O^\top\right)_{kk}\notag \\
&= \sigma^2 \left( O\left(X^*(t) + \Onot(\eps)\right)O^\top - O e^{At/\eps}\bar{X}(0) e^{A^\top t/\eps}O^\top\right)_{kk} \\
&= \sigma^2 \left( \left(d_k^*(t) + \Onot(\eps)\right) -  e^{2a_kt/\eps}\left(O\bar{X}(0)O^\top\right)_{kk}\right) \\
&= \sigma^2 \left( \left(d_k^* + \Onot(\eps)\right) -  e^{2a_kt/\eps}\left(O\bar{X}(0)O^\top\right)_{kk}\right), 
\end{align*}
This further implies
\begin{align*}
\p \left( \sup_{0 \leq r \leq t} O_k \xi_r \geq \lambda_k h^2 d^*_k \right) 
&\leq K t \left( \sqrt{\lambda_k h}d_k^*\right)^{\frac{1}{H}} \exp\left( -\frac{\lambda_k h^2d^*_k}{2\sup_{0 \leq r \leq t} \var(O_k\xi_r)}\right) \\
&\leq K t \left( \sqrt{\lambda_k h}d_k^*\right)^{\frac{1}{H}} \exp\left( -\frac{\lambda_k h^2}{2\sigma^2}\frac{d^*_k}{d^*_k + \Onot(\eps)}\right) \\
&\leq K t \left( \sqrt{\lambda_k h}d_k^*\right)^{\frac{1}{H}} \exp\left( -\frac{\lambda_k h^2}{2\sigma^2}(1 - \Onot(\eps))\right).
\end{align*} 
Summing over the dimensions yields
\begin{align*}
\p \left( \sup_{0 \leq r \leq t} \langle \xi_r, X^*(r)^{-1} \xi_r \rangle \geq h^2 \right) 
\leq \sum_{k=1}^m K t \left( \sqrt{\lambda_k h}d_k^*\right)^{\frac{1}{H}} \exp\left( -\frac{\lambda_k h^2}{2\sigma^2}(1 - \Onot(\eps))\right)
\end{align*}
and this finishes the proof.
\end{proof}

In the case when $A$ is normal we get a nice description of the $d_k^*$. The corresponding $k$-th component of the diagonalized critical manifold results in
\begin{align*}
d^*_k = \:&H(2H-1) \frac{1}{\left| a_k + \overline{a_k}\right|} \int_0^{\infty} \left(e^{a_kv} + e^{\overline{a_k}v}\right) v^{2H-2} \diff v,
\intertext{which simplifies even further if all eigenvalues are real}
d^*_k = &\:H(2H-1) \frac{1}{\left| a_k\right|} \int_0^{\infty} e^{a_kv} v^{2H-2} \diff v \\
= &\:H(2H-1)\frac{1}{\left| a_k \right|^{2H}} \Gamma (2H-1) 
= \frac{1}{\left| a_k \right|^{2H}} H \Gamma (2H).
\end{align*}

%%%%%%%%%%%%%%%%%%%%%%%%%%%%%%%%%%%%%%%%%%%%%%%%%%%%%%%%%%%%%%%%%%%%%%%%%%%%%%%%%%%%%%%%%%%%%%%%%%%%%%%%%%%%%%%%%%%%%%%%%%%%%
\subsubsection{Symmetric Linearization}

From now on, we consider the case when $A$ is a symmetric matrix (i.e.~$A = A^\top$). The reason for this restriction is that in the following the proves to bound the probability of $(\xi_t)_{t \in I}$ exiting the neighborhood up to time $t$
\begin{align*}
\p \left( \exists \:0 \leq r \leq t: \: (x_r,y_r) \notin \mathcal{B}(h) \right) 
\end{align*}
is based on linearizing the underlying system and understanding the structure of the eigenvalues of the covariance. We actually require normality of $A$ for this strategy to work as it is sufficient to use the functional equality of the matrix exponential. Furthermore, $e^A$ inherits the normality structure, which is a necessary and sufficient criterion to characterize the eigenvalues of $e^{A\lambda}e^{A^\top\mu},$ $\lambda,\mu \geq 0$.
To be able to generalize the result of variant 1 it is crucial that the eigenvalues of $A$ are all real. These two criteria already imply that $A$ is symmetric. In particular, we see that the critical manifold is of the form
\begin{align*}
X^*(t) &= H(2H-1) \int_0^{\infty} e^{Au}\left( \int_0^{\infty} \left(e^{Av} + e^{A^\top v} \right) v^{2H-2} \diff v \right) e^{A^\top  u} \diff u \\
&= H(2H-1) \int_0^{\infty} \int_0^{\infty} \left(e^{A(u+v) + A^\top u} + e^{Au + A^\top (v+u)}\right) v^{2H-2} \diff v \diff u.
\end{align*}
Thanks to the discussion above the we can consider the diagonalization $ D^*(t) := UX^*(t)U^\top $. Its $k$-th ($k=1,\dots,m$) diagonal component is given by
\begin{align*}
d^*_k(t) := \left( UX^*(t)U^\top  \right)_{kk}
&= H(2H-1) \int_0^{\infty} \int_0^{\infty} \left(e^{a_k(u+v) + a_ku} + e^{a_k u + a_k(v+u)}\right) v^{2H-2} \diff v \diff u \\
&= H(2H-1) \frac{1}{\left| a_k\right|} \int_0^{\infty} e^{a_kv} v^{2H-2} \diff v \\
&= H(2H-1)\frac{1}{\left| a_k \right|^{2H}} \Gamma (2H-1) = \frac{1}{\left| a_k \right|^{2H}} H \Gamma (2H).
\end{align*}
Similarly we can rewrite the covariance
\begin{align*}
\cov(\xi_t) &= \frac{\sigma^2}{\eps^{2H}} \int_0^t \int_0^t e^{A(t-u)} e^{A^\top (t-v)} H (2H - 1) \left| u - v \right|^{2H-2} \diff u \diff v \\
&= \frac{\sigma^2}{\eps^{2H}} H (2H - 1) \int_0^t \int_0^t e^{A(t-u) + A^\top (t-v)} \left| u - v \right|^{2H-2} \diff u \diff v. 
\end{align*}
and diagonalize it with respect to the same $U$ and consider its $k$-th ($k=1,\dots,m$) component %\textcolor{red}{[TODO: please double-check if I understood the following correctly in the old version as there was a confusion between variance vector, its entries and the covariance matrix. Of course, also above one works with covariance matrices at certain points not vectors.]}
\begin{align*}
\left(U\cov(\xi_t)U^\top \right)_{kk} &= \left(\cov(U\xi_t)\right)_{kk} \\
&= \frac{\sigma^2}{\eps^{2H}} H (2H - 1) \int_0^t \int_0^t e^{a_k(t-u) + a_k(t-v)} \left| u - v \right|^{2H-2} \diff u \diff v.
\end{align*} 
We obtain the following result
\begin{theorem}
Let $t \in I.$ Then under Assumption \ref{assMDStable} and if $A$ is in addition symmetric with real eigenvalues $a_1, \dots, a_m$ there is a constant  such that for $h >0$ the following estimate holds for $\eps >0$ small enough
\begin{align*}
\p \left( \sup_{0 \leq r \leq t} \langle \xi_r, X^*(r)^{-1} \xi_r \rangle \geq h^2 \right)
2e \left\lceil \frac{\left| a_+ \right|t}{\eps} \frac{h^2}{\sigma^2} (1 + \Onot (\eps))\right\rceil \exp\left( -\frac{h^2}{2\sigma^2m}\right),
\end{align*}
where $a_{+} = \max\{|\mu|: \mu \text{ eigenvalue of }A\}.$
\end{theorem}
\begin{proof}
Let $U^\top  = \begin{pmatrix}  U_1^\top  & \cdots  &U_m^\top  \end{pmatrix}$, where $U_k$ denotes the $k$-th row of $U.$ 
Now
\begin{align*}
\p \left( \sup_{0 \leq r \leq t} \langle \xi_r, X^*(r)^{-1} \xi_r \rangle \geq h^2 \right) &= \p \left( \sup_{0 \leq r \leq t} \langle U \xi_r, D^*(r)^{-1} U \xi_r \rangle \geq h^2 \right) \\
&= \p \left( \sup_{0 \leq r \leq t} \sum_{k=1}^m (U_k \xi_r) d^*_k(r)^{-1} (U_k \xi_r) \geq h^2 \right)
\intertext{for $\lambda_k \geq 0$ with $\sum\limits_{k=1}^m \lambda_k = 1$}
&\leq \sum_{k=1}^m \p \left( \sup_{0 \leq r \leq t} d^*_k(r)^{-1} (U_k \xi_r)^2 \geq \lambda_k h^2 \right) \\
&= 2 \sum_{k=1}^m \p \left( \sup_{0 \leq r \leq t} \frac{U_k \xi_r}{\sqrt{d^*_k(r)}} \geq \sqrt{\lambda_k} h \right).
\end{align*}
Due to the normality of $e^{A(t-u)}$ (inherited by $A$) the Gaussian process $U_k \xi$ has variance
\begin{align*}
\var(U_k \xi_r) = \frac{\sigma^2}{\eps^{2H}} \int_0^t \int_0^t e^{a_k(t-u)/\eps} e^{a_k(t-v)/\eps} H (2H - 1) \left| u - v \right|^{2H-2} \diff u \diff v.
\end{align*} 
In particular, we can show that the process $\left(e^{-a_kt}U_k\xi_t\right)_{t \geq 0}$ satisfies the assumptions of Lemma \ref{bernsteintypeineq}, so that we can apply the Bernstein-type inequality. (The proof is completely analogous as in the one-dimensional case, see proof of Theorem \ref{thmvariantbernsteinineq}.)
To get a relation between the $k$-th value of the diagonalized covariance and the corresponding component of the  critical manifold consider \eqref{eqMDSlowManifVarApproach} in the $U$-coordinate system
\begin{align}
\var(U_k\xi_t) &= \left(U\cov(\xi_t)U^*\right)_{kk}\notag \\
&= \sigma^2 \left( U\left(X^*(t) + \Onot(\eps)\right)U^* - U e^{At/\eps}\bar{X}(0) e^{A^\top t/\eps}U^*\right)_{kk}\notag \\
&= \sigma^2 \left( \left(d_k^*(t) + \Onot(\eps)\right) -  e^{2a_kt/\eps}\left(U\bar{X}(0)U^*\right)_{kk}\right)\notag \\
&= \sigma^2 \left( \left(d_k^* + \Onot(\eps)\right) -  e^{2a_kt/\eps}\left(U\bar{X}(0)U^*\right)_{kk}\right), \label{eqMDcomponentVarandCritManif}
\end{align}
as $d^*(t) \equiv d^*$ is actually independent of time in our case.
Now, we can use the same strategy as variant 1 in the one dimensional case for each $k.$ For $\gamma \in (0, 1/2)$ let $0 = t_0 < t_1 < \ldots < t_N$ be a partition containing the interval $[0,t]$ such that 
\begin{align*}
-a_k(t_{i+1}-t_i) = \eps \gamma \quad \quad \text{for} \quad \quad 0 \leq i < N = \left\lceil \frac{\left| a_kt \right|}{\eps \gamma} \right\rceil .
\end{align*}
We start by estimating the probability of the exit time on $[t_i,t_{i+1})$ for $i \in \{0, \dots, N-1 \}$
\begin{align*}
&\p \left( \sup_{t_i \leq r < t_{i+1}} \frac{U_k \xi_r}{\sqrt{d^*_k(r)}} \geq \sqrt{\lambda_k} h \right) \\
\leq\: &\p \left( \sup_{t_i \leq r < t_{i+1}} e^{-a_kr/\eps} U_k \xi_r \geq \sqrt{\lambda_k} h \inf_{t_i \leq r < t_{i+1}} e^{-a_kr/\eps} \sqrt{d^*_k(r)} \right). 
\intertext{Applying Lemma \ref{bernsteintypeineq}}
\leq \: &\exp\left( -\frac{1}{2} \frac{\lambda_k h^2 \inf_{t_i \leq r < t_{i+1}} e^{-2a_kr/\eps} d^*_k(r)}{e^{-2a_kt_{i+1}/\eps} \cov(U_k \xi_{t_{i+1}})}\right) \\
\overset{\eqref{eqMDcomponentVarandCritManif}}{\leq} &\exp\left( -\frac{\lambda_k h^2}{2\sigma^2} \frac{d^*_k}{\left(d^*_k + \Onot(\eps)\right)} e^{2a_k(t_{i+1}-t_{i})/\eps}\right) \\
\leq\: &\exp\left( -\frac{\lambda_k h^2}{2\sigma^2} e^{-2\gamma}(1 - \Onot (\eps))\right).
\end{align*}
Taking the union of the events that $U_k\xi$ has exited $\mathcal{B}(h)$ in $[t_i,t_{i+1})$ and using the subadditivity of the probability measure, yields
\begin{align*}
\p \left( \sup_{0 \leq r \leq t} \frac{U_k \xi_r}{\sqrt{d^*_k(r)}} \geq \sqrt{\lambda_k} h \right)
&\leq \sum_{i = 0}^{N-1} \p \left( \sup_{t_i \leq r < t_{i+1}} \frac{U_k \xi_r}{\sqrt{d^*_k(r)}} \geq \sqrt{\lambda_k} h \right) \\
&\leq \left\lceil \frac{\left| a_k \right|t}{\eps \gamma} \right\rceil \exp\left( -\frac{\lambda_k h^2}{2\sigma^2} e^{-2\gamma}(1 - \Onot (\eps))\right) \\
&\leq \left\lceil \frac{\left| a_k \right|t}{\eps \gamma} \right\rceil \exp\left(\frac{\lambda_k h^2}{\sigma^2} \gamma (1 + \Onot (\eps))\right) \exp\left( -\frac{\lambda_k h^2}{2\sigma^2}\right).
\end{align*}
Finding the minimal bound with respect to $\gamma$ now corresponds to optimizing
\begin{align*}
\gamma \mapsto \frac{\left| a_k \right|t}{\eps \gamma} \exp\left(\frac{\lambda_k h^2}{\sigma^2} \gamma (1 + \Onot (\eps))\right) \exp\left( -\frac{\lambda_k h^2}{2\sigma^2}\right),
\end{align*}
due to the monotonicity of $\lceil\cdot\rceil$. The optimal value is achieved for
\begin{align*}
\gamma = \frac{\sigma^2}{\lambda_k h^2} \left( 1 + \Onot(\eps)\right)^{-1}.
\end{align*}
Plugging this in the estimate gives the bound for the $k$-th component
\begin{align*}
\p \left( \sup_{0 \leq r \leq t} \frac{U_k \xi_r}{\sqrt{d^*_k(r)}} \geq \sqrt{\lambda_k} h \right) 
\leq e \left\lceil \frac{\left| a_k \right|t}{\eps} \frac{\lambda_k h^2}{\sigma^2} (1 + \Onot (\eps))\right\rceil \exp\left( -\frac{\lambda_k h^2}{2\sigma^2}\right).
\end{align*}
Summing over the dimensions
\begin{align*}
&\p \left( \sup_{0 \leq r \leq t} \langle \xi_r, X^*(r)^{-1} \xi_r \rangle \geq h^2 \right) \\
\leq\: &2 \sum_{k=1}^m \p \left( \sup_{0 \leq r \leq t} \frac{U_k \xi_r}{\sqrt{d^*_k(r)}} \geq \sqrt{\lambda_k} h \right) \\
\leq\: &2e \sum_{k=1}^m \left\lceil \frac{\left| a_k \right|t}{\eps} \frac{\lambda_k h^2}{\sigma^2} (1 + \Onot (\eps))\right\rceil \exp\left( -\frac{\lambda_k h^2}{2\sigma^2}\right) \\
\leq\: &2e \sum_{k=1}^m \left\lceil \frac{\left| a_+ \right|t}{\eps} \frac{\lambda_k h^2}{\sigma^2} (1 + \Onot (\eps))\right\rceil \exp\left( -\frac{\lambda_k h^2}{2\sigma^2}\right),
\end{align*}
where $a_+ := \max\{|\mu|: \mu \text{ eigenvalue of }A\}.$ The optimal value is now attained by choosing $\lambda_k = \frac{1}{m}.$ This yields 
\begin{align*}
\p \left( \sup_{0 \leq r \leq t} \langle \xi_r, X^*(r)^{-1} \xi_r \rangle \geq h^2 \right)
\leq 2e \left\lceil \frac{\left| a_+ \right|t}{\eps} \frac{h^2}{\sigma^2} (1 + \Onot (\eps))\right\rceil \exp\left( -\frac{h^2}{2\sigma^2m}\right).
\end{align*}
The proof is complete.
\end{proof}

Due to the symmetry of $A$ we could have diagonalized the SDE in the beginning (i.e.~look at it in the $U$-coordinate system) and done the whole theory established in Chapter 2 to get the existence of a slow manifold $\zeta_k(t)$ for $\var(U_k \xi_t)$, which is of the form
\begin{align*}
\zeta_k(t) &= \frac{1}{\left| a_k \right|^{2H}} H \Gamma (2H) + \Onot(\eps) \\
&= d^*_k + \Onot(\eps).
\end{align*}
However, we decided to use the results on the higher dimensional systems as much as possible to clearly indicate which steps of the proof can be generalized to more general classes of matrices beyond symmetric ones.

%%%%%%%%%%%%%%%%%%%%%%%%%%%%%%%%%%%%%%%%%%%%%%%%%%%%%%%%%%%%%%%%%%%%%%%%%%%%%%%%%%%%%%%%%%%%%%%%%%%%%%%%%%%%%%%%%%%%%%%%%%%%%%%%%%
\section{Outlook}
\label{ch_outlook}

This work provides a first step towards the investigation of fast-slow systems driven by fBm using sample paths estimates. So far we have examined the behavior close to a normally hyperbolic attracting invariant manifold in finite dimensions. Numerous extensions could be considered as next steps.

Having covered the uniformly attracting case, it is then natural to conjecture that there are scaling laws for the fluctuations as fast subsystem bifurcation points are approached, i.e., when hyperbolicity is lost. These results are available in the fast-slow Brownian motion case~\cite{KuehnCT2}. However, even when the fast dynamics is dominated by nonlinear terms~\cite{KuehnRomano} or one considers fast-slow maps with bounded noise~\cite{KuehnMalavoltaRasmussen} using modified proofs and additional technical tools it is possible to save many results. This robustness of the scaling laws near the loss of normal hyperbolicity leads one to conjecture that it will still be possible to prove such results for the fast-slow fBm case when $H\in(1/2,1)$.  

However, the analysis of fast-slow systems for $H \in(0,1/2)$ is expected to be more complicated due to several reasons. First of all, a different integration theory has to be considered, see for instance~\cite{Malliavin1,StochCalcforfBm}. Furthermore, the kernel~\eqref{phi} we have used to develop an approximation of the variance by means of the slow manifold has a non-integrable singularity for $H<1/2$. Last but not least, Bernstein-type inequalities as established in Lemma~\ref{bernsteintypeineq} do not hold true anymore, since the covariance function of the fractional Brownian motion is negative. Consequently, one has to develop completely different techniques in this case. Another related extension would be to analyze the dynamics of fast-slow systems driven by multiplicative noise. This issue, however, requires a more general theory than It\^{o}-calculus because the fractional Brownian motion is not a semi-martingale. 

Furthermore, one could consider other stochastic processes with memory. More precisely, one could think of other stochastic processes whose covariance functions are represented by 
        \begin{align*}
        \int\limits^{\min\{s,t\}}_{0} K(s,r) K(t,r) ~\diff r,~~\mbox{ for }s,t\geq 0,
        \end{align*}
for suitable square integrable kernels $K$, recall~\eqref{int:kernel}. Beyond fBm, further examples in this sense are the multi-fractional Brownian motion or the Rosenblatt process~\cite{CMaslowski}. However, the analysis of fast-slow systems in this case is a challenging question, since these processes do not have in general stationary increments and are no longer Gaussian (as e.g.~Rosenblatt processes).

Finally, one can also broaden the scope of the applications. Although climate dynamics is certainly a very important topic, where time-correlated noise is well-motivated by data such as temperature measurements, it is not the only possible application. Other areas, where fast-slow systems with fBm could be considered are financial markets. For example, assets could be modelled as fast variables influenced by fBm stochastic forcing, while the slow variables are political/social factors influencing the market, which change on a much slower time scale in many cases. Similar remarks and examples of concrete applications are likely also exist in many contexts in neuroscience, ecology and epidemiology, where stochastic fast-slow systems with Brownian motion are already used frequently.\medskip

\textbf{Acknowledgments:} CK \& AN would like to thank the German Science Foundation (DFG) for support via grants (KU 3333/2-1) and (GN 109/1-1). CK would like to thank the VolkswagenStiftung for support via a Lichtenberg Professorship. The authors are extremely grateful to Professor Andrey Dorogovtsev for valuable suggestions and for pointing out Theorem~A.1. CK also acknowledges partial support of the EU within the TiPES project funded the European Unions Horizon 2020 research and innovation programme under grant agreement No.~820970. The authors  thank the referees for the valuable suggestions.

%%%%%%%%%%%%%%%%%%%%%%%%%%%%%%%%%%%%%%%%%%%%%%%%%%%%%%%%%%%%%%%%%%%%%%%%%%%%%%%%%%%%%%%%%%%%%%%%%%%%%%%%%%%%%%%%%%%%%%%%%%%%%%%%%
\newpage
\appendix
\section{On the Limit Superior of Gaussian Processes}
\label{app_LimSup}

The following proof has been developed in personal communication with Professor Andrey Dorogovtsev.

\begin{theorem}\label{thm_limsupGauss}
Let $(Y_t)_{t\geq 0}$ be a centered Gaussian process with covariance function
$R(t,s) = \e[Y_t Y_s] $
satisfying
\begin{enumerate}
\item $R(t,s) \rightarrow 0$ for $t-s \rightarrow \infty,$
\item $ R(t,t) = 1$ for all $t\geq 0.$
\end{enumerate}
Then
\begin{align*}
\limsup_{t \rightarrow \infty} Y_t = \infty~~\mbox {a.s}.
\end{align*}
\end{theorem}
\begin{proof} 
We aim to construct a sequence of independent random variables inductively to apply the Borel-Cantelli lemma. The strategy is highly based on the fact that for Gaussian random variables independence is equivalent to zero covariance. Let $t_1 := 0.$ Given $t_1 < \dots < t_n$ we apply Gram-Schmidt orthogonalization to the variables $Y_{t_1}, \dots, Y_{t_n}$ in $L^2(\Omega,\mathcal{F},\p)$ to obtain uncorrelated and normalized random variables $S_1, \dots, S_n$ satisfying for each $1 \leq i \leq n$
\begin{align*}
S_i = \sum_{j=1}^n a_{ij} Y_{t_j}, \quad
Y_{t_i} = \sum_{j=1}^n b_{ij} S_j,
\end{align*}
where the coefficients $a_{ij}$ and $b_{ij}$ only depend on the on the values $R(t_i,t_j)$ for $1 \leq i,j \leq n$ by construction. Now for any $t > t_n$ we can project $Y_t$ on the space $\Span_{1\leq j \leq n}\{Y_{t_j}\} = \Span_{1\leq j \leq n}\{S_j\}$
\begin{align*}
Y_{t} = \sum_{j=1}^n \e[Y_{t}S_j]S_j + Y'_{t}.
\end{align*}
Observe that by construction $Y'_{t}$ is uncorrelated, and hence independent, of the the variables $Y_{t_1}, \dots, Y_{t_n}.$ This holds in particular if we choose $t_{n+1} \geq \max\{t_n, n+1\}$ with
\begin{align*}
\e\left[ \left(\sum_{j=1}^{n+1} \e[Y_{t_{n+1}}S_j]S_j\right)^2\right] \leq \frac{1}{2^n},
\end{align*}
where the latter is possible because $R(t,s) \rightarrow 0$ for $t-s \rightarrow \infty.$
Inductively we get for each $n \in \N$
\begin{align*}
Y_{t_n} = Y'_{t_n} + \gamma_n.
\end{align*}
For the sequence of independent random variables $(Y'_{t_n})_{n \in \N}$
\begin{align*}
\sum_{j=1}^{\infty} \p \left( Y'_{t_n} > c \right) &\geq \sum_{j=1}^{\infty} \p \left( Y_{t_n} > c \right) + \p \left( -\gamma_n > c \right) \\
&= \sum_{j=1}^{\infty} \p \left( Y_{t_1} > c \right) + \p \left( -\gamma_n > c \right) \\
&= \infty,
\end{align*}
because $\var(Y_{t_n}) = R(t_n,t_n) = 1$ by assumption. By the Borel-Cantelli lemma we obtain 
\begin{align*}
\limsup_{n \rightarrow \infty} Y'_{t_n} = \infty~~\mbox{ a.s.}
\end{align*}
Now note that $\gamma_n \rightarrow 0$ as $n \rightarrow \infty$ as for any $c>0$
\begin{align*}
\p\left( |\gamma_n| > c\right) &\leq \frac{\e \left[|Y_{t_n} - Y'_{t_n}|^2\right]}{c^2} \\
&= \frac{\e\left[ \left(\sum_{j=1}^{n} \e[Y_{t_{n}}S_j]S_j\right)^2\right]}{c^2} \\
&\leq \frac{1}{2^{n-1}c^2} 
\longrightarrow 0, \text{ as } n \rightarrow \infty.
\end{align*}
This implies that
\begin{align*}
\limsup_{n \rightarrow \infty} Y_{t_n} = \limsup_{n \rightarrow \infty} (Y'_{t_n} + \gamma_n) = \infty~~\mbox{a.s}
\end{align*}
as required.
\end{proof}

Now we can apply this result to our setting. 
\begin{cor}\label{cor_limsupxi}
Under the Assumptions of Theorem \ref{thmvariantbernsteinineq} we have
\begin{align*}
\limsup_{t \rightarrow \infty} e^{-\alpha(t)/\eps}\xi_t = \infty~~\mbox{a.s.}
\end{align*}
\end{cor}
\begin{proof}
Define $Z_t := \frac{e^{-\alpha(t)/\eps}\xi_t}{\sqrt{\var\left(e^{-\alpha(t)/\eps}\xi_t\right)}}.$ Then $\mbox{Var}(Z_{t})=\e( Z_t Z_t)=1$ by construction. To prove the second assumption of Theorem~\ref{thm_limsupGauss}, observe that 
$$R(t,s) = \e\left[\frac{\xi_t}{\sqrt{\var\left(\xi_t\right)}} \frac{\xi_s}{\sqrt{\var\left(\xi_s\right)}}\right].$$
Now note that 
%we can w.l.o.g. restrict ourselves to the interval $[\delta, \infty)$ for $\delta > 0$ small and that the 
$\var(\xi_t)$ is bounded for $t \in [0, \infty).$ %(An easy way to see this is to consider the exponential fast approach towards the slow manifold, which is uniformly bounded, given in \eqref{relSlowManVar}.)
This means it suffices to prove for every $s >0$
\begin{align*}
\e\left[\xi_t \xi_s\right] \rightarrow 0 \text{ as } t \rightarrow 0.
\end{align*}
The correlation function is given by
\begin{align*}
&\e\left[\xi_t \xi_s\right] \\
&= \frac{\sigma^2}{\eps^{2H}} \int_0^s \int_0^t e^{\alpha(t,u)/\eps} e^{\alpha(s,v)/\eps} F(u) F(v) H (2H - 1) \left| u - v \right|^{2H-2} \diff u \diff v \\
&\leq \frac{\sigma^2F^2_+}{\eps^{2H}} \int_0^s \int_0^t e^{\alpha(t,u)/\eps} e^{\alpha(s,v)/\eps}H (2H - 1) \left| u - v \right|^{2H-2} \diff u \diff v,
\end{align*}
where $F_+ := \sup\limits_{r \in [0, \infty)} F(r).$
Let $\eps > 0.$ Then there is $T>0$ such that for $a := \sup\limits_{r \in [0, \infty)} a(r)$
\begin{align*}
H (2H - 1) \frac{\sigma^2F^2_+}{\eps^{2H}} |T|^{2H-2} \frac{1}{|a|^2} < \frac{\eps}{2}.
\end{align*}
Now choose $t_0> s + T$ such that for all $t \geq t_0$
\begin{align*}
&\frac{\sigma^2F^2_+}{\eps^{2H}} \int_0^s \int_{0}^{t} e^{\alpha(t,u)/\eps} e^{\alpha(s,v)/\eps}H (2H - 1) \left| u - v \right|^{2H-2} \diff u \diff v \\
=\: &\frac{\sigma^2F^2_+}{\eps^{2H}} e^{\alpha(t,s+T)} \underbrace{\int_0^s \int_{0}^{s+T} e^{\alpha(s+T,u)/\eps} e^{\alpha(s,v)/\eps}H (2H - 1) \left| u - v \right|^{2H-2} \diff u \diff v}_{\text{independent of }t} < \frac{\eps}{2},
\end{align*}
where we have used the semi-group property of $e^{\alpha(t,u)/\eps}.$
Putting this together we get for all $t \geq t_0$
\begin{align*}
&\e\left[\xi_t \xi_s\right] \\
&\leq \frac{\sigma^2F^2_+}{\eps^{2H}} \int_0^s \int_0^t e^{\alpha(t,u)/\eps} e^{\alpha(s,v)/\eps}H (2H - 1) \left| u - v \right|^{2H-2} \diff u \diff v \\
&=\frac{\sigma^2F^2_+}{\eps^{2H}} \int_0^s \int_0^{s+T} e^{\alpha(t,u)/\eps} e^{\alpha(s,v)/\eps}H (2H - 1) \left| u - v \right|^{2H-2} \diff u \diff v \\
&\quad+ \frac{\sigma^2F^2_+}{\eps^{2H}} \int_0^s \int_{s+T}^t e^{\alpha(t,u)/\eps} e^{\alpha(s,v)/\eps}H (2H - 1) \left| u - v \right|^{2H-2} \diff u \diff v \\
&< \frac{\eps}{2} +H (2H - 1) \frac{\sigma^2F^2_+}{\eps^{2H}} |T|^{2H-2} \int_{s+T}^t e^{\alpha(t,u)/\eps} \diff u \int_0^se^{\alpha(s,v)/\eps}  \diff v \\
&\leq \frac{\eps}{2} + H (2H - 1) \frac{\sigma^2F^2_+}{\eps^{2H}} |T|^{2H-2} \frac{1}{|a|^2} < \eps.
\end{align*}
Now, Theorem~\ref{thm_limsupGauss} proves the claim.
\end{proof}
\begin{rem}
	Note that one cannot directly apply classical probabilistic results such as the Borel-Cantelli Lemma, law of iterated logarithm or ergodic theorems to prove Corollary~\ref{cor_limsupxi}, since the process $\xi$ has neither stationary nor independent increments.
\end{rem}


\newpage

\addcontentsline{toc}{section}{Bibliography}
\begin{thebibliography}{10}

\bibitem{Ashkenazyetal}
Y.~Ashkenazy and D.R.~Baker and H.~Gildor and S.~Havlin.
\newblock Nonlinearity and multifractality of climate change in the past 420,000 years.
\newblock {\em  Geophysical research letters}, 30(22), 2003.

\bibitem{Barbozaetal}
L.~Barboza and B.~Li and M.P.~Tingley and F.G.~Viens.
\newblock Reconstructing past temperatures from natural proxies and estimated climate forcings using short-and long-memory models.
\newblock {\em  The Annals of Applied Statistics}, 8(4):1966--2001, 2014.

\bibitem{BerglundGentz8}
N.~Berglund and B.~Gentz.
\newblock The effect of additive noise on dynamical hysteresis.
\newblock {\em Nonlinearity}, 15(3):605--632, 2002.

\bibitem{NoiseSlowFastSys}
N.~Berglund and B.~Gentz.
\newblock {\em Noise-Induced Phenomena in Slow-Fast Dynamical Systems}.
\newblock Springer-Verlag London, 2006.

\bibitem{StochCalcforfBm}
F.~{Biagini}, Y.~{Hu}, B.~{\O ksendal}, and T.~{Zhang}.
\newblock {\em Stochastic Calculus for Fractional Brownian Motion and
  Applications}.
\newblock Springer-Verlag London Limited, 2008.

\bibitem{BlenderFraedrich}
R. Blender and K. Fraedrich
\newblock {Long time memory in global warming simulations}.
\newblock {\em Geophys. Res. Lett.}, 30(14), 2003.

\bibitem{bogachev2007measure}
V.I. Bogachev.
\newblock {\em Measure Theory}.
\newblock Number Bd. 1 in Measure Theory. Springer Berlin Heidelberg, 2007.

\bibitem{CMaslowski}
P.~{\v{C}}oupek and B.~Maslowski.
\newblock {Stochastic evolution equations with Volterra noise}.
\newblock {\em Stochastic Process. Appl.}, 127(3):877--900, 2017.

\bibitem{DavidsenGriffin}
J. Davidsen and J. Griffin
\newblock {Volatility of unevenly sampled fractional Brownian motion: An application to ice core records}.
\newblock {\em Phys. Rev. E}, 81(1):016107, 2010.

\bibitem{Malliavin1}
L.~Decreusefond et~al.
\newblock Stochastic analysis of the fractional {B}rownian motion.
\newblock {\em Potential Anal.}, 10(2):177--214, 1999.

\bibitem{decreusefond2008hitting}
L.~Decreusefond and D.~Nualart.
\newblock {Hitting times for Gaussian processes}.
\newblock {\em Ann. Probab.}, 36(1):319--330, 2008.

\bibitem{Dijkstra}
H.A. Dijkstra.
\newblock {\em Nonlinear Climate Dynamics}.
\newblock CUP, 2013.

\bibitem{Eichneretal}
J.F. Eichner and E. Koscielny-Bunde and A. Bunde and S. Havlin and H.J. Schellnhuber.
\newblock Power-law persistence and trends in the atmosphere: A detailed study of long temperature records.
\newblock {\em Phys. Rev. E}, 68(4):046133, 2003.

\bibitem{Fenichel4}
N.~Fenichel.
\newblock {Geometric Singular Perturbation Theory for Ordinary Differential
  Equations}.
\newblock {\em J. Differential Equat.}, 31(1):53--98, 1979.

\bibitem{neuron}
F.-Y. Gao, Y.-M. Kang, X.~Chen, and G.~Chen.
\newblock Fractional {G}aussian noise-enhanced information capacity of a
  nonlinear neuron model with binary signal input.
\newblock {\em Phys. Rev. E}, 97(5):052142, 2018.

\bibitem{Gardiner}
C.~Gardiner.
\newblock {\em Stochastic Methods}.
\newblock Springer, Berlin Heidelberg, Germany, 4th edition, 2009.

\bibitem{GarridoLuSchmalfuss1}
M.~J. Garrido-Atienza, K.~Lu, and B.~Schmalfuss.
\newblock Local pathwise solutions to stochastic evolution equations driven by
  fractional {B}rownian motions with {H}urst parameter ~${H} \in (1/3, 1/2] $.
\newblock {\em Discrete Cont. Dyn.-B}, 20, 11 2014.

\bibitem{GnannKuehnPein}
M.~Gnann, C.~Kuehn, and A.~Pein.
\newblock Towards sample path estimates for fast-slow {SPDEs}.
\newblock {\em Euro. J. Appl. Math.}, 30(5):1004--1024, 2019.

\bibitem{GubinelliLejayTindel}
M.~Gubinelli, A.~Lejay, and S.~Tindel.
\newblock Young integrals and {SPDE}s.
\newblock {\em Potential Anal.}, 25(4):307--326, 2006.

\bibitem{GubinelliTindel}
M.~Gubinelli, S.~Tindel, et~al.
\newblock Rough evolution equations.
\newblock {\em Ann. Probab.}, 38(1):1--75, 2010.

\bibitem{GH}
J.~Guckenheimer and P.~Holmes.
\newblock {\em Nonlinear Oscillations, Dynamical Systems, and Bifurcations of
  Vector Fields}.
\newblock Springer, New York, NY, 1983.

\bibitem{Hairer}
M.~Hairer and X.-M. Li.
\newblock Averaging dynamics driven by fractional {B}rownian motion.
\newblock {\em arXiv:1902.11251}, 2019.

\bibitem{Hek}
G.~Hek.
\newblock Geometric singular perturbation theory in biological practice.
\newblock {\em J. Math. Biol.}, 60(3):347--386, 2010.

\bibitem{HesseNeamtu1}
R.~Hesse and A.~Neam\c{t}u.
\newblock Local mild solutions for rough stochastic partial differential
  equations.
\newblock {\em J.~Differential Equat.}, 267(11):6480--6538, 2019.

\bibitem{Jones}
C.K.R.T. Jones.
\newblock {Geometric Singular Perturbation Theory}.
\newblock In {\em Dynamical Systems (Montecatini Terme, 1994)}, volume 1609 of
  {\em Lect. Notes Math.}, pages 44--118. Springer, 1995.

\bibitem{KaperEngler}
H.~Kaper and H.~Engler.
\newblock {\em Mathematics and Climate}.
\newblock SIAM, 2013.

\bibitem{Kaerner}
O.~K{\"a}rner.
\newblock On nonstationarity and antipersistency in global temperature series.
\newblock {\em Journal of Geophysical Research: Atmospheres}, 107(D20):ACL-1, 2002.

\bibitem{Kolmogorov}
A.N. Kolmogorov.
\newblock {Wienersche Spiralen und einige andere interessante Kurven im
  Hilbertschen Raum}.
\newblock {\em Comptes Rendus (Doklady) Acad. Sci. URSS}, 26(2):115--118, 1940.

\bibitem{KoutsoyiannisMontanari}
D.~Koutsoyiannis and A.~Montanari.
\newblock Statistical analysis of hydroclimatic time series: uncertainty and
  insights.
\newblock {\em Water Resources Research}, 43(5):W05429, 2007.

\bibitem{KuehnCT2}
C.~Kuehn.
\newblock {A mathematical framework for critical transitions: normal forms,
  variance and applications}.
\newblock {\em J. Nonlinear Sci.}, 23(3):457--510, 2013.

\bibitem{MultTimeScaleDyn}
C.~Kuehn.
\newblock {\em {Multiple Time Scale Dynamics}}, volume 191 of {\em Applied
  Mathematical Sciences}.
\newblock Springer, 2015.

\bibitem{KuehnMalavoltaRasmussen}
C.~Kuehn, G.~Malavolta, and M.~Rasmussen.
\newblock Early warning signs for bifurcations with bounded noise.
\newblock {\em J. Math. Anal. Appl.}, 464(1):58--77, 2018.

\bibitem{Kuznetsov}
Y.A. Kuznetsov.
\newblock {\em Elements of Applied Bifurcation Theory}.
\newblock Springer, New York, NY, 3rd edition, 2004.

\bibitem{MaslowskiNualart}
B.~Maslowski and D.~Nualart.
\newblock Evolution equations driven by a fractional {B}rownian motion.
\newblock {\em J. Funct. Anal.}, 202(1):277 -- 305, 2003.

\bibitem{Mishura}
Y.~Mishura.
\newblock Financial applications of fractional Brownian motion.
\newblock {\em Stochastic Calculus for Fractional Brownian Motion and Related
  Processes}, pages 301--326, 2008.

\bibitem{Mobergetal}
A.~Moberg, D.M. Sonechkin, K.~Holmgren, N.M. Datsenko, and W.~Karl{\'e}n.
\newblock Highly variable {Northern Hemisphere} temperatures reconstructed from
  low- and high-resolution proxy data.
\newblock {\em Nature}, 433(7026):613, 2005.

\bibitem{Molchan}
G.M. Molchan.
\newblock On the maximum of a fractional {B}rownian motion.
\newblock {\em Theory Probab. Appl.}, 44(1):97--102, 2000.

\bibitem{SelAspOfFBM}
I.~Nourdin.
\newblock {\em Selected Aspects of Fractional Brownian Motion}.
\newblock Springer-Verlag Mailand, 2012.

\bibitem{nerve}
D.~J. Odde, E.~M. Tanaka, S.~S. Hawkins, and H.~M. Buettner.
\newblock Stochastic dynamics of the nerve growth cone and its microtubules
  during neurite outgrowth.
\newblock {\em Biotechnol. Bioeng}, 50(4):452--461, 1996.

\bibitem{AsMethforGaussProc}
V.I. {Piterbarg}.
\newblock {\em Asymptotic Methods in the Theory of Gaussian Processes and
  Fields}.
\newblock American Mathematical Society, 1996.

\bibitem{RangarajanSant}
G.~Rangarajan and D.A. Sant.
\newblock A climate predictability index and its applications.
\newblock {\em Geophys. Res. Lett.}, 24(10):1239--1242, 1997.

\bibitem{na}
A.~Richard, P.~Orio, and E.~Tanr{\'e}.
\newblock An integrate-and-fire model to generate spike trains with long-range
  dependence.
\newblock {\em Journal of computational neuroscience}, pages 1--16, 2018.

\bibitem{KuehnRomano}
F.~Romano and C.~Kuehn.
\newblock Analysis and predictability for tipping points with leading-order
  nonlinear terms.
\newblock {\em Int. J. Bifurc. Chaos}, 28(8):1850103, 2018.

\bibitem{SadhuKuehn}
S.~Sadhu and C.~Kuehn.
\newblock Stochastic mixed-mode oscillations in a three-species predator-prey
  model.
\newblock {\em Chaos}, 28(3):033606, 2017.

\bibitem{Slepian}
D.~Slepian.
\newblock The one-side barrier problem for {G}aussian noise.
\newblock {\em Bell System Tech.}, 41(2):436--501, 1962.

\bibitem{Sonechkin}
D.M.~Sonechkin.
\newblock Climate dynamics as a nonlinear Brownian motion.
\newblock {\em International Journal of Bifurcation and Chaos}, 8(04):799--803, 1998.

\bibitem{Stone}
H.~Stone.
\newblock Calibrating rough volatility models: a convolutional neural network
  approach.
\newblock {\em arXiv:1812.05315}, 2018.

\bibitem{SuRubinTerman}
J.~Su, J.~Rubin, and D.~Terman.
\newblock Effects of noise on elliptic bursters.
\newblock {\em Nonlinearity}, 17(1):133--157, 2004.

\bibitem{networkt}
H.~Rootz\'en T.~Mikosch, S.~Resnick and A.~Stegeman.
\newblock Is network traffic appriximated by stable {L}\'evy motion or
  fractional {B}rownian motion?
\newblock {\em Ann. Appl. Probab.}, 12(1):23--68, 2002.

\bibitem{Tikhonov}
A.N. Tikhonov.
\newblock Systems of differential equations containing small parameters in the
  derivatives.
\newblock {\em Mat. Sbornik N. S.}, 31(3):575--586, 1952.

\bibitem{Verhulst}
F.~Verhulst.
\newblock {\em Methods and Applications of Singular Perturbations: Boundary
  Layers and Multiple Timescale Dynamics}.
\newblock Springer, 2005.

\bibitem{Weiss}
M.~Weiss.
\newblock Single-particle tracking data reveal anticorrelated fractional
  {B}rownian motion in crowded fluids.
\newblock {\em Phys. Rev. E}, 88(1):010101, 2013.

\bibitem{Xu}
Y.~Xu, R.~Guo, and W.~Xu.
\newblock A limit theorem for the solutions of slow-fast systems with
  fractional {B}rownian motion.
\newblock {\em Theoretical and Applied Mechanics Letters}, 4(1):013003, 2014.

\bibitem{Yuanetal}
N.~Yuan, Z.~Fu, and S.~Liu.
\newblock Extracting climate memory using fractional integrated statistical model: A new perspective on climate prediction.
\newblock {\em Scientific reports}, 4:6577, 2014.

\end{thebibliography}
\end{document}